\documentclass[a4paper,reqno]{amsart}
\pdfoutput=1
\usepackage[T1]{fontenc}
\usepackage{lmodern}
\usepackage{dsfont}

\usepackage{amsmath,amsfonts,amssymb,amsthm,stmaryrd,bbm,graphicx,mathtools,enumerate}
\usepackage{mathrsfs}
\usepackage[utf8]{inputenc}  

\usepackage[english]{babel}

\usepackage[tracking,spacing,kerning,babel]{microtype}
\usepackage{wasysym} 
\usepackage{color}
\usepackage{xcolor}
\usepackage{framed} 

\usepackage{wrapfig} 

\usepackage[normalem]{ulem}
\usepackage{soul}

\usepackage[final]{hyperref}   
\hypersetup{
    linktoc=page,
    linkcolor=red,          
    citecolor=blue,        
    filecolor=blue,      
    urlcolor=cyan,
   colorlinks=true           
}

\usepackage{lipsum} 

\usepackage[toc,page]{appendix}

\DeclareRobustCommand{\SkipTocEntry}[5]{}

\usepackage{multicol}

\newtheoremstyle{mine}
{\baselineskip}
{\baselineskip}
{\itshape}
{
}
{\bfseries}
{.}
{.5em}
{#1 #2\ifx#3\relax\else~(#3)\fi}

\theoremstyle{mine}

\newtheorem{theorem}{Theorem}
\numberwithin{theorem}{section}
\newtheorem{corollary}[theorem]{Corollary}
\newtheorem{proposition}[theorem]{Proposition}
\newtheorem{prop}[theorem]{Proposition}
\newtheorem{lemma}[theorem]{Lemma}

\newtheorem{definition}[theorem]{Definition}
\newtheorem{conjecture}{Conjecture}
\newtheorem{OP}{Open Problem} 
\numberwithin{equation}{section}

\theoremstyle{remark}
\newtheorem{remark}{Remark}

\theoremstyle{remark}

\colorlet{shadecolor}{blue!10}

\newcommand{\margin}[1]{\textcolor{magenta}{*}\marginpar[\textcolor{magenta} {  \raggedleft  \footnotesize  #1 }  ]{ \textcolor{magenta} { \raggedright  \footnotesize  #1 }  }}

\newcommand{\margingreen}[1]{\textcolor{darkgreen}{*}\marginpar[\textcolor{darkgreen} {  \raggedleft  \footnotesize  #1 }  ]{ \textcolor{darkgreen} { \raggedright  \footnotesize  #1 }  }}
\renewcommand{\margin}[1]{}
\renewcommand{\margingreen}[1]{}
\def\rm{\reversemarginpar}

\let\qed=\QED

\newcommand{\mbb}{\mathbb} 
\newcommand{\mrm}{\mathrm}

\newcommand{\R}{\mathbb{R}}
\newcommand{\D}{\mathbb{D}}

\newcommand{\Z}{\mathbb{Z}}

\renewcommand{\S}{\mathbb{S}}
\newcommand{\B}{\mathbb{B}}

\newcommand{\1}{\mathbf 1}

\def\calH{\mathcal{H}}

\def\calP{\mathcal{P}}

\def \eps{\varepsilon}
\def \vareps{\varepsilon}

\def\dist{\mathrm{dist}}

\def\P{\mathbb{P}} 
\def\E{\mathbb{E}} 
\def\md{\mid}

\def\Bb#1#2{{\def\md{\bigm| }#1\bigl[#2\bigr]}}

\def\Eb{\Bb\E}

\def\FK#1#2#3{{\def\md{\bigm| } \P_{#1}^{\,#2}  \bigl[  #3 \bigr]}}

\def \p {{\partial}}

\def\<#1{\langle #1\rangle}

\newcommand{\purple}[1]{{\color{purple}#1}}

\definecolor{darkgreen}{rgb}{0,0.6,0.05}
\newcommand{\green}[1]{{\color{darkgreen}#1}}

\def\bi{\begin{itemize}}  
\def\ei{\end{itemize}}
\def\bnum{\begin{enumerate}} 
\def\enum{\end{enumerate}}
\def\ni{\noindent}
\def\bf{\bfseries}

\title[BKT for small perturbations of the spin-$O(N)$ model]{Small perturbations of the $2d$ classical Heisenberg model}
\author{Christophe Garban}
\author{Nathan de Montgolfier}
\date{}

\address
{Universit\'e Claude Bernard Lyon 1, CNRS UMR 5208, Institut Camille Jordan, 69622 Villeurbanne, France, and Courant Institute (NYU), New York, USA}
\email{christophe.garban@gmail.com}

\address
{Universit\'e Claude Bernard Lyon 1, CNRS UMR 5208, Institut Camille Jordan, 69622 Villeurbanne, France}
\email{demontgolfier@math.univ-lyon1.fr}

\newcommand{\newsection}[2][]{%
  \if\relax\detokenize{#1}\relax
    \section{#2}%
  \else
    \section{\for{toc}{#1}\except{toc}{#2}}%
  \fi
}

\begin{document}

\begin{abstract}
Polyakov predicted in 1975 (\cite{polyakov1975interaction}) that the  spin $O(3)$-model (or classical Heisenberg model) should have exponential decay of correlations at all positive temperatures on $\Z^2$. 

Our main result is as follows: if we replace the Euclidean scalar product 
\begin{align*}
\<{\sigma_i , \sigma_j} = \sigma_i^x \sigma_j^x + \sigma_i^y \sigma_j^y + \sigma_i^z \sigma_j^z
\end{align*}
by the $\lambda$-modified scalar product
\begin{align*}
\<{\sigma_i , \sigma_j}_\lambda := \sigma_i^x \sigma_j^x + \sigma_i^y \sigma_j^y + \lambda \sigma_i^z \sigma_j^z
\end{align*}
we show that for all $\lambda\in(-1,1)$, the $\S^2$ valued spin system with Gibbs measure $\exp(\beta \sum_{i\sim j} \<{\sigma_i,\sigma_j}_\lambda)$ has power-law decay of correlations if $\beta$ is chosen large enough. 

This is the first evidence that arbitrary small smooth modifications of the classical Heisenberg model produce a BKT phase in $2d$. It was proved in the 1970s and 80s that there exist arbitrary small smooth modifications of the classical Heisenberg model which produce instead an Ising phase. In these work, they considered the regime $\lambda>1$ while in this paper, we consider the regime $\lambda<1$ which cannot be analyzed via reflection positivity and chessboard estimate. 

The two main ingredients used in our proof are 
\bnum
\item The stability of the BKT phase under percolation type disorder (\cite{dario2025phase}). 
\item Weak regularity type properties of the {\em groundstates} or $\S^2$-valued {\em harmonic maps} associated to our $\lambda$-energy functional. This analytic step is interesting on its own. 
\enum
Finally, the exponential and power-law decay regime should be separated by a curve $\lambda \mapsto \beta_c(\lambda)$. We prove a rigorous lower bound in the case of the $\lambda$-anisotropic hemisphere model $\S^2_+$ which matches with RG flow predictions.
\end{abstract}


\maketitle

\tableofcontents

\section{Introduction}

\subsection{Context.}

$ $

The classical Heisenberg model on the $2$-dimensional lattice $\Z^2$ is defined as follows: for any finite domain $\Lambda \subset \Z^2$ and any $\S^2$-valued configuration $\sigma \in (\S^2)^{\Lambda}$, one assigns a Gibbs weight to $\sigma$ proportional to $\exp(\beta \sum_{i\sim j \in \Lambda} \sigma_i \cdot \sigma_j )$ (where $\beta \equiv \tfrac 1 T$ stands for the inverse temperature). We thus consider the Gibbs  measure 
\begin{align*}
& \hskip 2 cm \mu_{\Lambda,\beta}(d\sigma) := \frac 1 {Z_{\Lambda,\beta}}  \exp(\beta \sum_{i\sim j \in \Lambda} \sigma_i \cdot \sigma_j ) \prod_{i\in \Lambda} d\sigma_i   \\
& \,\, \text{ where the partition function } Z_{\Lambda,\beta} := \int_{(\S^2)^\Lambda}  \exp(\beta \sum_{i\sim j \in \Lambda} \sigma_i \cdot \sigma_j ) \prod_{i\in \Lambda} d\sigma_i  \,, 
\end{align*}
(N.B. The base measure on $(\S^2)^\Lambda$ is the product of uniform measures on $\S^2$ over the sites of $\Lambda$).
It is not known whether a unique limit exists as $\Lambda \nearrow \Z^2$ and by classical Heisenberg model on $\Z^2$ we shall mean any subsequential scaling limit as $\Lambda \nearrow \Z^2$. 

Polyakov famously predicted in \cite{polyakov1975interaction} that the classical Heisenberg model ($\S^2$-valued) as well as its generalization, the spin-$O(N)$ model ($\S^{N-1}$-valued) behave completely differently than 
\bi
\item the $XY$ model (or spin-$O(2)$ model) which is $\S^1$-valued  \\
and
\item  the Ising model (or spin-$O(1)$ model) which is $\S^0= \{-1,1\}$-valued. 
\ei
Indeed, the latter is known to exhibit a {\em Berezinskii--Kosterlitz--Thouless (BKT)} phase transition when $d=2$ where correlations decay at polynomial speed when $\beta$ is large enough. And the former (Ising model) has a classical symmetry breaking phase transition.

More precisely, Polyakov predicted that the classical Heisenberg model (and all spin-$O(N)$ models for $N \ge 3 $) have exponential decay of correlations at all $\beta<\infty$. 
\begin{conjecture}[\cite{polyakov1975interaction}]
For any $\beta<\infty$, there exists a constant $c(\beta)>0$ such that for any infinite volume limit  $\<{\cdot}_\beta$ of the classical Heisenberg model on $\Z^2$, the two-point function satisfies 
\begin{align*}
\<{\sigma_x \cdot \sigma_y}_{\beta} \leq \exp\big(- c(\beta) \| x-y\|_2 \big)\,.
\end{align*}
The same holds for the spin-$O(N)$ model with a constant $c_N(\beta) \geq c(\beta)$ as long as $N\geq 3$. 
\end{conjecture}

This conjecture was arguably one of the most important unsolved conjectures in mathematical physics. A proof of this conjecture has just been announced by OpenAI in \cite{openAI} (N.B. Our work neither relies on, nor follows from this recent breakthrough). 
We refer for example to 
\cite{
peled2019lectures,
friedli2017statistical,
frohlich1981kosterlitz,
garban2022continuous, 
aru2025percolation}
for some useful references on classical Heisenberg model, $XY$ model, spin-$O(N)$ model and their properties in dimension $d=2$ as well as when $d\geq 3$. 

\smallskip

The goal of this paper is to investigate what happens under small smooth perturbations of the sphere $\S^2$. For example, the perturbation which is most studied in the literature (see for example \cite{frohlich1978phase,bricmont_inequalities_1985,dunlop1985correlation, friedli2017statistical}) consists in slightly modifying the scalar product as follows: for any $\lambda\in \R$,  the $\lambda$-modified scalar product on $\R^3$ is defined by 
\begin{align}\label{e.lambda}
\<{u,v}_{\purple{\lambda}} := u^x v^x + u^y v^y + \purple{\lambda} u^z v^z  \;\;\;\;\;\; \text{      for any $u,v\in \R^3$\,.}
\end{align}
(N.B. This is a proper scalar product only when $\lambda>0$). 
This leads to the following $\lambda$-perturbation of the classical Heisenberg model 
\begin{align*}
 \mu_{\Lambda, \lambda, \beta}(d\sigma) \propto  \exp(\beta \sum_{i\sim j \in \Lambda} \<{\sigma_i , \sigma_j}_{\lambda} ) \prod_{i\in \Lambda} d\sigma_i  \,
\end{align*}
(see Definition \ref{d.modif} for a more precise definition as well as Definition \ref{d.modif2} for other natural smooth deformations of the classical Heisenberg model).

When only one direction gets more weight than the others -- i.e. when $\lambda>1$, it has been proved in the series of works \cite{malyshev1975phase,kunz_inequalities_1976,frohlich1978phase} that the system exhibits an Ising-type long-range-order in the third direction\footnote{Interestingly the first two references do not rely on the chessboard estimate used in the later one.}. (Some of these works require working with a similar deformation of the $XY$ model known as the anisotropic $XY$ model).  See also the book \cite[Theorem 10.17 and Remark 10.21]{friedli2017statistical}. 


In this paper, we focus instead on the case where $\lambda<1$. Bricmont, Lebowitz and Pfister conjectured in \cite{bricmont_inequalities_1985} that in this case, there must be  a BKT phase (power-low decay phase) at low enough temperature. See their predicted phase diagram in Figure \ref{f.PhaseDiagram1}. Their conjecture was proved by Dunlop (\cite{dunlop1985correlation}) when $\lambda \ll 1$ with a proof based on Wells' inequality (\cite{wells1977some}). Our main result it to prove Bricmont, Lebowitz, Pfister's conjecture for all $\lambda<1$. Our method is very different from \cite{dunlop1985correlation} and is based on regularity properties of $\lambda$-harmonic maps. It is robust enough to extend to many other smooth deformations of the classical Heisenberg model, see Subsection \ref{ss.MoreSmooth} and Theorem \ref{th.allS}.

Let us emphasise that such anisotropic models turn out to be very relevant (if not more)  in the context of Quantum spin systems where perturbations of the XXX spin chain (or Heisenberg Ferromagnet) are considered and lead to the so called XXZ spin chain with Hamiltonian $\calH = \sum_{i\sim j} JS_i^x S_j^x + JS_i^y S_j^y + \tilde J S_i^z S_j^z$ (where $S_i^{x,y,z}$ stand for Pauli matrices). See in particular the recent work \cite{bjornberg2025exponential} which mentions a similar predicted phase diagram as in Figure \ref{f.PhaseDiagram1} for the occurrence of BKT in $d=2$ for such anisotropic quantum spin chains. See also \cite{koma1992decay,benassi2017decay,kjall2012phase}.

\subsection{Main results.}

Let us first define precisely  the main smooth deformation of the classical Heisenberg model on $\S^2$ we will consider. As hinted above, we will discuss in Subsection \ref{ss.MoreSmooth} other types of smooth deformations which can be analyzed as well with our methods, with the notable exception of the metric deformation  which -- we believe --  may have a very different behaviour at large $\beta$.  See Remark \ref{r.metric} and Open Problem \ref{op.1}.

\begin{definition}[$\lambda$-perturbation of the classical Heisenberg model]\label{d.modif}
For any $\lambda\in \R$ and any finite $\Lambda \subset \Z^2$, we consider the following Hamiltonian on the state space $(\S^2)^{\Lambda}$,
\begin{align*}
H_\lambda(\sigma):= - \sum_{i \sim j } \<{\sigma_i , \sigma_j}_\lambda \,, \,\,\, \forall \sigma \in (\S^2)^{\Lambda} \,. 
\end{align*}
We may also fix some boundary conditions in the usual way, i.e. if $\bar \sigma \in (\S^2)^{\Z^2}$, then $H_{\Lambda}(\sigma,\bar \sigma):= - \sum_{i,j \in \Lambda} \<{\sigma_i,\sigma_j}_\lambda - \sum_{i\in \Lambda, j\in \p \Lambda} \<{\sigma_i, \bar \sigma_j}_\lambda$.

For any $\beta \geq 0$, the Gibbs measures we will consider are the infinite volume limit(s) (as the box $\Lambda\subset \Z^2 \nearrow \Z^2$) of the following probability measures 
\begin{align}
\mu_{\Lambda, \lambda, \beta}(d\sigma) \propto \exp(-\beta H_\lambda(\sigma) ) \prod_{i\in \Lambda} d\sigma_i
\end{align}
where $d\sigma_i$ denotes the uniform measure on $\S^2$ (and we may also prescribe boundary conditions $\bar \sigma$ on $\p \Lambda$ as $\Lambda \nearrow \Z^2$). 

The value $\lambda=1$ thus corresponds to the classical Heisenberg model.  (See Figure \ref{f.simulations} for a simulation of this model at $\lambda=0.99$ versus $\lambda=1$).
\end{definition}

\begin{figure}[htbp]
  \centering
   \includegraphics[width=\textwidth]{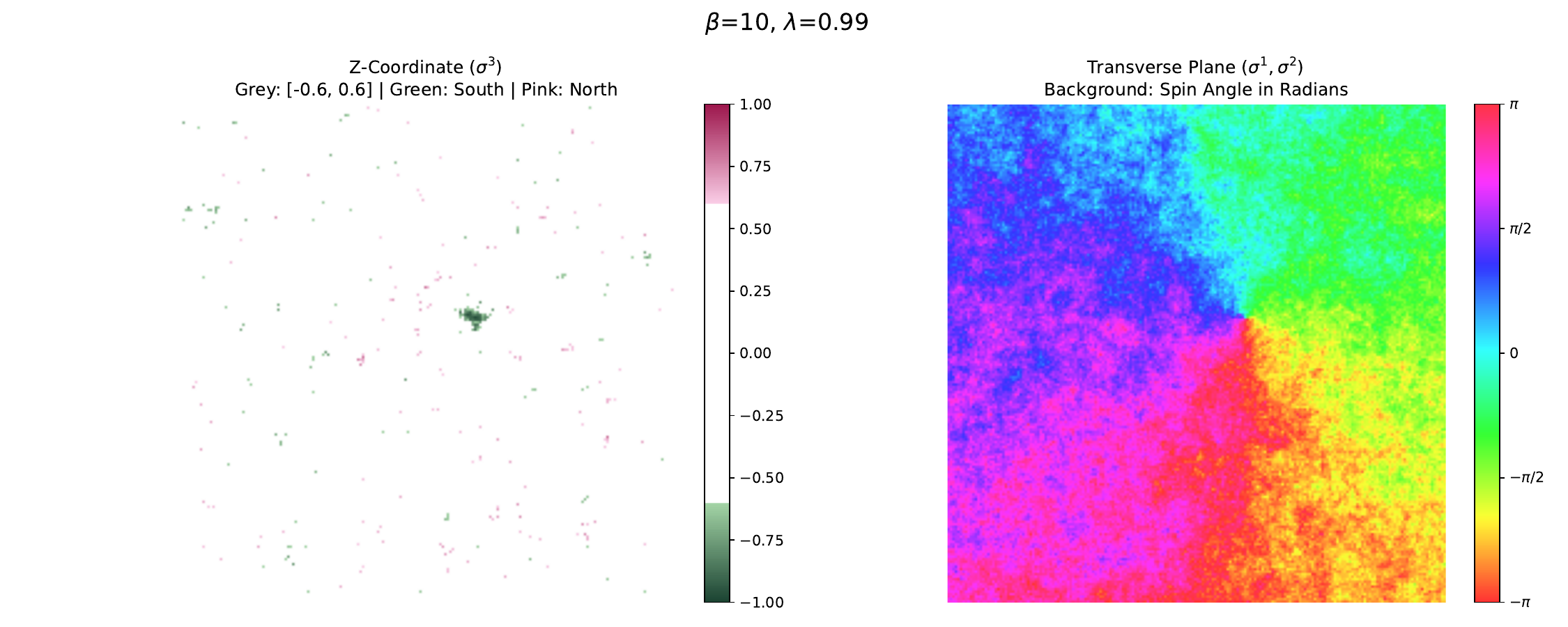} \\
   \includegraphics[width=\textwidth]{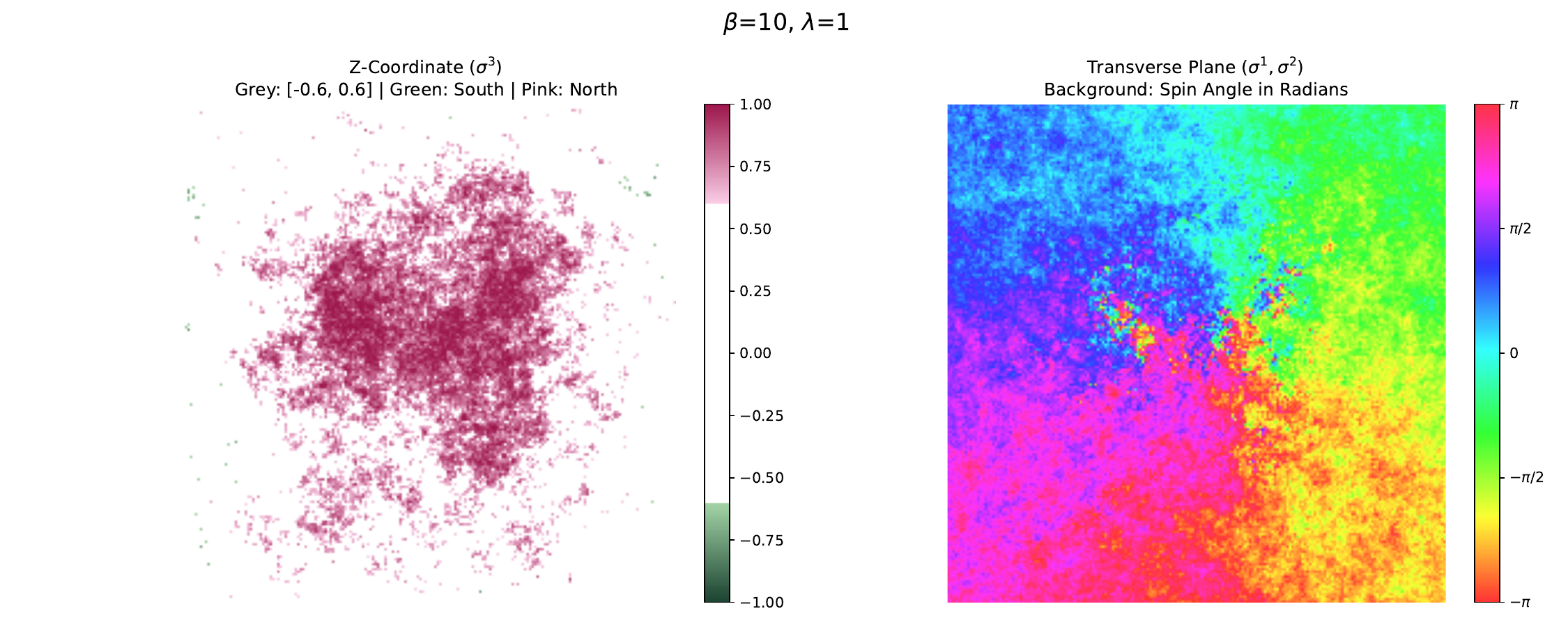}
   \caption{Simulations of the anisotropic $\S^2$ model on a $200\times 200$ square box with a boundary condition that winds around the box once at a regular speed. 
   \textbf{Top/Left:} $\beta = 10$ and $\purple{\lambda = 0.99}$ (vertical axis $\sigma_i^{(3)}$ on the left, spin arguments of $(\sigma_i^{(1)}, \sigma_i^{(2)})$ on the right). 
   \textbf{Bottom/Right:} $\beta = 10$ and $\purple{\lambda = 1}$ (vertical axis $\sigma_i^{(3)}$ on the left, spin arguments of $(\sigma_i^{(1)}, \sigma_i^{(2)})$ on the right).}
   \label{f.simulations}
 \end{figure}

Our main result is the following (as well as Theorem \ref{th.allS} which extends to a wider class of smooth deformations). 

\begin{theorem}\label{th.main}
For any $\lambda\in(-1,1)$, there exists $\bar \beta=\bar \beta(\lambda)$ such that the  $\lambda$-perturbation of the classical Heisenberg --  Definition \ref{d.modif} --  satisfies the following. For all $\beta \geq \bar \beta(\lambda)$, any infinite volume limit obtained from the free or Dirichlet boundary conditions  has power-law decay in the $xy$ plane. This means that for all $\beta\geq \bar \beta(\lambda)$, there exists $\alpha(\beta)>0$ so that uniformly in $a \neq b \in \Z^2$, 
\begin{align*}
\Eb{\sigma_a^x \sigma_b^x + \sigma_a^y \sigma_b^y} \geq \frac c {\|a - b\|_2^{\alpha(\beta)}}\,.
\end{align*}
See Figure \ref{f.PhaseDiagram1} for an illustration of the corresponding phase diagram. The result applies to a larger class of boundary conditions, see Remark \ref{rk:boundary_conditions} below.
\end{theorem}

\begin{figure}[!htp]
\begin{center}
\includegraphics[width=0.49\textwidth]{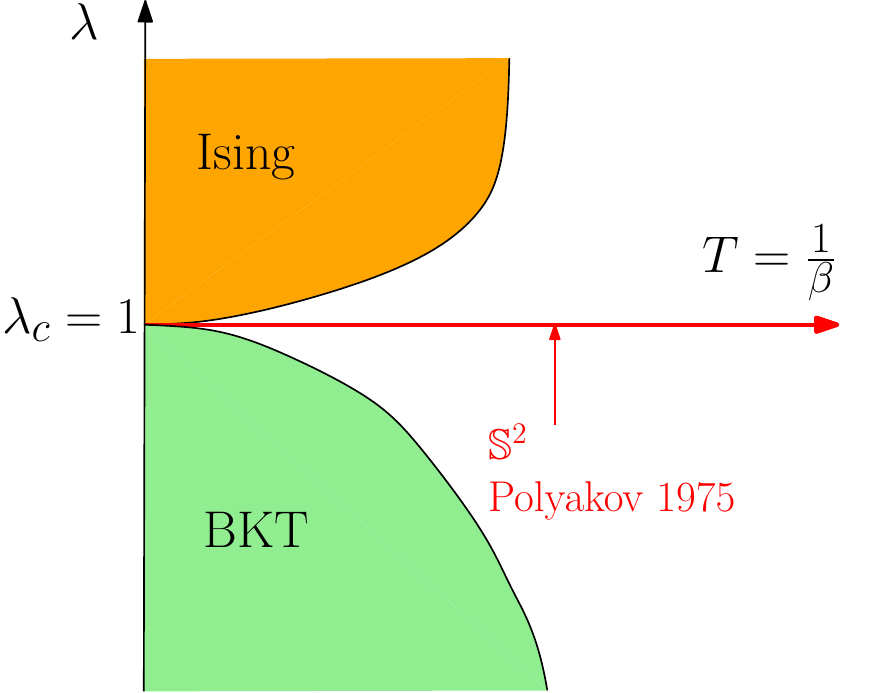}\,
\includegraphics[width=0.49\textwidth]{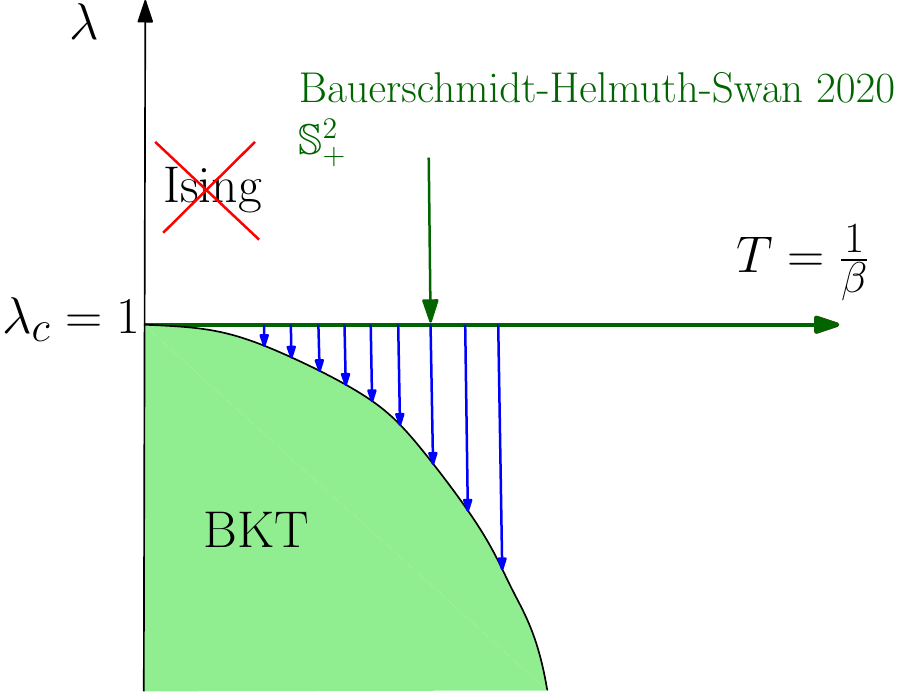}
\end{center}
\caption{
This phase diagram for $\S^2$ on the left-hand side was already conjectured by Bricmont, Lebowitz and Pfister in \cite{bricmont_inequalities_1985}. The vertical axis corresponds to the $\lambda$-modification of the scalar product $\<{\cdot}_\lambda$.  The line  $\lambda=\lambda_c=1$ is  the classical $O(3)$ invariant Heisenberg model where exponential decay is predicted to hold at all positive temperatures $T>0$ (horizontal axis). The existence of an Ising phase for all $\lambda>1$ is known since the works \cite{malyshev1975phase,kunz_inequalities_1976,frohlich1978phase}. Our main result, Theorem \ref{th.main} states that the BKT region (in green) goes all the way to the predicted exponential phase. The line separating the BKT phase from the exponential phase  should behave like $\lambda_c(T) \approx 1 - e^{- C/T}$. 
This can be seen by doing an RG flow computation in the $(T,\lambda)$ plane (we thank here Slava Rychkov who explained to us how to obtain this curve).  Our second main result -- Theorem \ref{thm:exponential_decay}--- shows that there is indeed exponential decay up to this  predicted line  for the $\lambda$-anisotropic hemisphere model $\S^2_+$. This is pictured with the blue arrows on the right-hand side.  In both left and right cases ($\S^2$ and $\S^2_+$), our proof of power-law decay in the green (BKT) region does not go all the way to the predicted $\lambda_c(T) \approx 1 - e^{- C/T}$  but can be made quantitative and gives a polynomial bound of order $\lambda_c(T) \gtrapprox 1 - T^{1/6}$ as discussed in Section \ref{sec:thmmain_quant}.}\label{f.PhaseDiagram1}
\end{figure}

The proof easily extends to perturbations of spin-$O(N)$ model when $N\geq 4$. In this case, spins belong to $\S^{N-1}\subset \R^N$ and one may consider the modified scalar-product with $\lambda<1$ 
\begin{align}\label{e.SPN}
\<{\sigma_i,\sigma_j}_\lambda 
:= \sigma_i^1 \sigma_j^1 + \sigma_i^2 \sigma_j^2 + 
\lambda
\sum_{k=3}^N \sigma_i^k \sigma_j^k\,.
\end{align}
Or more generally for any $\bar \lambda: = (\lambda_k)_{3 \leq k \leq N}$ with $\| \bar \lambda\|_\infty <1$, 
\begin{align}\label{e.SPN2}
\<{\sigma_i,\sigma_j}_{\bar \lambda}
:= \sigma_i^1 \sigma_j^1 + \sigma_i^2 \sigma_j^2 + 
\sum_{k=3}^N  \lambda_k \, \sigma_i^k \sigma_j^k\,.
\end{align}
Let us state in particular the special case of $N=4$ which is of special interest, and for which Section \ref{s.CIP} provides a more direct proof which avoids the use of regularity of harmonic maps (this approach requires $\lambda_3=\lambda_4$ in~\eqref{e.SPN2}).
\begin{corollary}
In the special case of the modified spin-$O(4)$ model on $\Z^2$, if $\lambda\in \R$ is the parameter in~\eqref{e.SPN}, we have 
\bi
\item If $\lambda \in \{-1,1\}$, then by Polyakov's prediction, we expect exponential decay at all inverse temperatures $\beta<\infty$. 
\item If $\lambda \in (-1,1)$, there is a BKT  phase  (power-law decay) in the $\<{e_1,e_2}$ plane at small enough temperature ($\beta \geq \bar \beta(\lambda)$).
\item If $|\lambda| > 1$, there is also a BKT  phase  (power-law decay) in the $\<{e_3,e_4}$ plane at small enough temperature. 
\ei
\end{corollary}

%

\begin{remark}\label{rk:boundary_conditions}
  Let $\sigma = (r \tau^1 , r\tau^2, \sigma^z)$ with $(\tau^1, \tau^2) \in \S^1$. Conditionally, on the field $(\sigma^{z}_i)_{i \in \Lambda}$, the field $(\tau_1 \tau_2)$ is an XY model with Hamiltonian \[H(\tau) = - \sum_{i = 1}^n r_i r_j \tau_i \cdot \tau_j~.\] Theorem \ref{th.main} applies to any boundary condition that has the property that, conditionally on the field $(\sigma^3_i)$, $\tau$ has free or Dirichlet boundary condition (or any boundary conditions for which Ginibre's inequalities hold). This is the case both for free and Dirichlet boundary condition on the initial model $\sigma$.
\end{remark}

\begin{remark}\label{rk:other lattices}
Theorem \ref{th.main} extends to the triangular lattice by the same proof.
Notice that $\lambda < 0$ case does not provide any extra information in the square lattice because the square lattice $\Z^2$ is bipartite and one may replace $\sigma_i = (\sigma_i^x, \sigma_i^y, \sigma_i^z)$ by  $(\sigma_i^x, \sigma_i^y, - \sigma_i^z)$ on half of the sites $i$ to go from the $\lambda$-anisotropic model to the $(- \lambda)$-anisotropic model. On the triangular lattice this trick no longer works. It also fails  for the hemisphere model even on the square lattice (see Section \ref{sec:hemisphere_intro}).
\end{remark}

We also establish a version of the Simon-Lieb inequality (\ref{thm:simon_lieb}) for the  anisotropic $O(3)$ and $O(4)$ models. Our proof follows closely the proof of Aizenman and Simon of the same inequality for the isotropic models  in \cite{aizenman_local_1980} which is based on local Wards identities. This implies (see e.g. \cite{van2023elementary}) that the BKT phase transition of these models is sharp, meaning that the speed of decay of the two point function is either polynomial or exponential.

\begin{proposition}[No exponential decay implies polynomial decay at a slow rate]
  \label{cor:sharpness}
  For $N= 3, 4$, any $-1\le  \lambda \le 1$, any $c < \frac{1}{\beta}$ and $k \ge 1$ and for any subsequential limit of the free boundary condition, if $\<{ \sigma_i^1 \sigma_j^1 }_{\beta, \lambda} \le \frac{c}{8 |i-j| } $ for all $i, j \in \Z^2$ with $|i-j| \ge k$, then $\<{ \sigma_i^1 \sigma_j^1 } \le \exp(- c' |i-j|) $ for some $c'$ that depends on $c$, $\beta$ and $k$.
\end{proposition}
We prove Simon--Lieb inequality and recall how the above proposition follows from it in Section \ref{sec:simon-Lieb}.
\begin{remark}
  This proposition also applies to the $O(4)$ model with Hamiltonian as in Equation (\ref{e.SPN2}).
\end{remark}

\subsection{Hemisphere model}\label{sec:hemisphere_intro} 

In Section \ref{s.hemisphere}, we study  an anisotropic spin model where the spins are constrained to take values in the upper hemisphere of $\S^2$. Let $\S^2_+ = \{ \tau \in \S^2 \text{ such that } \tau^3 > 0\}$. The model we consider has the distribution of the $\lambda$-anisotropic spin model conditioned on the event that all spins lie in $\S^2_+$; that is, its Gibbs measure over a finite box $\Lambda$ is proportional to 
\[ 
\mu^+_{\Lambda,  \lambda, \beta}(d\sigma) \propto \exp\big( - \beta H_{\lambda }(\sigma )\big) \prod_{i \in \Lambda} \nu^+(d \sigma_i)\,, 
\]
where $\nu^+$ is the restriction of the uniform measure of $\S^2$ to $\S^2_+$.
 
When $\lambda = 1$, this corresponds to the classical Heisenberg model conditioned to staying in the upper hemisphere. The great advantage in this case is that Bauerschmidt, Helmuth and Swan \cite[Theorem 8.3]{bauerschmidt2021geometry} established that the two-point function in the equator plane decays exponentially fast at all inverse temperatures $\beta \ge 0$ on any graph (with free boundary conditions). More precisely, they proved that $\langle \sigma_i^{1} \sigma_j^2 \rangle^+_{\beta} \le \exp\left( - c(\beta) d(i,j)\right)$ for some constant $c(\beta) > 0$ depending only on $\beta$ and the maximum degree of the graph. Thus, the isotropic $\S^2_+$ model provides an $O(2)$-symmetric setting where the equivalent of Polyakov's conjecture is rigorously known to hold. (N.B. before the breakthrough by OpenAI mentioned above).  

\smallskip
Remarkably, using the exact same proof as for Theorem \ref{th.main},
\begin{theorem}\label{th.mainS+}
Theorem \ref{th.main} also applies to the $\lambda$-anisotropic hemisphere model for $\lambda \in (-1, 1)$.
\end{theorem}

In fact, establishing the BKT phase for the hemisphere model yields a strictly stronger result than Theorem \ref{th.main}, thanks to the following correlation inequality.

\begin{prop}\label{prop:half-sphere-is-dominated-by-full-sphere}
For any $0 \le \lambda \le 1$ and any graph $\Lambda$, we have the domination
\[ 
\langle \sigma_i^1 \sigma_j^1 + \sigma_i^2 \sigma_j^2 \rangle_{\beta, \lambda}^+ \le  \langle \sigma_i^1 \sigma_j^1 + \sigma_i^2 \sigma_j^2 \rangle_{\beta, \lambda} \ .
\]
\end{prop}

A closely related inequality was used by Dunlop \cite{dunlop1985correlation} in his proof of the BKT transition for small enough $\lambda$, relying on ideas from \cite{dunlop_correlation_1976} and Kunz, Pfister, and Vuillermot \cite{kunz_inequalities_1976}.
While these papers contain highly relevant ideas, a direct proof of this specific domination appears to be absent from the literature. For completeness, we provide a probabilistic proof of Proposition \ref{prop:half-sphere-is-dominated-by-full-sphere} in Appendix \ref{sec:comparing_half_and_full}. 

By Theorem \ref{th.mainS+}, we know that for any $\lambda \in (-1,1)$, there exists some $\bar{\beta}(\lambda)$ such that for all $\beta \ge \bar{\beta}(\lambda)$, the transverse correlations decay only polynomially fast. It is not \textit{a priori} clear whether $\bar{\beta}(\lambda) \to \infty$ as $\lambda \to 1$. Building on the isomorphisms theorems techniques of \cite{bauerschmidt2021geometry}, we prove that this is indeed the case. More precisely, for the $\lambda$-anisotropic model on $\S^2_+$ taking place on $\Z^d$ (or any non-exponential growth transitive graph), the following holds.

\begin{theorem}\label{thm:exponential_decay}
Let $\lambda \in [0,1)$.
Let $\langle \cdot \rangle_{\beta, \lambda}^+$ be any translation-invariant subsequential limit of the finite-volume Gibbs measures $\langle \cdot \rangle_{\beta, \lambda, \Lambda}^+$. Then, there is an absolute constant $C > 0$ such that for any $\beta > 0$ satisfying 
\[
\exp\left(C \beta \right) < \frac{1}{1- \lambda} \ ,
\]
there exists $c'(\beta) > 0$ such that the transverse two-point function decays exponentially:
\[ 
\langle \sigma_i^1 \sigma_j^1 \rangle_{\beta, \lambda}^+ \le \exp\left(- c'(\beta) d(i,j)\right) \ . 
\]
\end{theorem}


Furthermore, Theorem \ref{thm:exponential_decay} provides -- in the case of $\S^2_+$ -- a rigorous lower bound on $\bar{\beta}$ that matches physics predictions. Indeed, using renormalization group techniques and free energy estimates, physicists (\cite{khokhlachev_two-dimensional_1976,hikami_phase_1980,pires_classical_1996}, see also \cite{bricmontrandom1986}), have indeed predicted that $\beta_{BKT}(\lambda) \approx |\log(1-\lambda)|$ as $\lambda \to 1$. Theorem \ref{thm:exponential_decay} implies exponential decay holds up to $\beta =  \frac{1}{C}|\log(1-\lambda)|$ and thus matches exactly the prediction. Conversely, the quantitative analysis of the proof of Theorem \ref{th.main} in Section \ref{sec:thmmain_quant} only provides an upper bound for $\bar{\beta}(\lambda)$ which is polynomial in $(1-\lambda)^{-1}$. See Figure \ref{f.PhaseDiagram1}. 

\subsection{More smooth deformations of the classical Heisenberg model.}\label{ss.MoreSmooth}

We now introduce other natural smooth deformations of the classical Heisenberg model which can be analysed with the techniques of this paper (basically all those listed below except one for which we expect the phenomenology to be different). 

\begin{definition}
\label{d.modif2}
In each of the items below, we specify in which target space $S$ the spins live, what is the base measure $\nu(dx)$ on the space $S$ and what is the Hamiltonian $H$ used for the interaction between spins. As before, the Gibbs measures we shall consider are the infinite volume limit(s) of measures (as the box $\Lambda\subset \Z^2 \nearrow \Z^2$) of the form
\begin{align}
\mu_{\Lambda,\beta}(d\sigma) \propto \exp(-\beta H(\sigma) ) \prod_{i\in \Lambda} \nu(d\sigma_i)
\end{align}
\bnum

\item  State space: the unit ball $\B^3$ with base measure the Lebesgue measure on $\B^3$ and Hamiltonian $H_\lambda = - \sum_{i\sim j \in \Lambda} \<{\sigma_i, \sigma_j}_\lambda$ (with $\lambda \in \R$).

\item State space: the $\eps$-ellipsoid. Let $\eps$ be in $[0,1)$. 
\begin{align*}
S=\S^2_\eps:=\{ (x,y,z)\in \R^3, x^2+y^2+\tfrac1 {1-\eps} z^2 =1 \}\,,
\end{align*} with its canonical base measure $\nu_\eps$ (induced by $\R^3$) and Hamiltonian $H_1 = - \sum_{i\sim j \in \Lambda} \<{\sigma_i, \sigma_j}$.

\item State space: the filled $\eps$-ellipsoid
\begin{align}
\mathbb{B}^3_\eps := \{ (x,y,z)\in \R^3, x^2+y^2+\tfrac1 {1-\eps} z^2  \leq 1 \} 
\end{align}
with base-measure the Lebesgue measure on $\B^3_\eps$ and Hamiltonian $H_1$. 

\item (Small confining potential). State space, $S=\S^2$ with base measure $\nu$
the uniform measure on $\S^2$ and Hamiltonian 
\begin{align*}
\tilde H_\eps := - \sum_{i\sim j \in \Lambda} \<{\sigma_i, \sigma_j} + \eps \sum_{i\in \Lambda} (\sigma_i^z)^2 
\end{align*}

\item(**) Metric perturbation. State space: the sphere $\S^2$. Let $\{g_\eps\}_{\eps >0}$ be a family of Riemannian metric tensors on $\S^2$ which are $SO(2)$ invariant and converging to the round metric as $\eps\to 0$. Base measure, $\nu_\eps$ the area measure associated to $g_\eps$ and Hamiltonian 
\begin{align}\label{e.metric}
\hat H_\eps := \frac 1 2 \sum_{i\sim j} d_{g_\eps}^2(\sigma_i, \sigma_j)
\end{align}
 We indicated this last example by (**) to highlight that this is the only $\eps$-perturbation which is not covered by the techniques of this paper. See Open Problem \ref{op.1}. 
\enum
\end{definition}

\begin{theorem}\label{th.allS}
For all the smooth perturbations listed in Definition \ref{d.modif2} except the last one~\ref{e.metric}, a BKT phase exists at large enough $\beta$ (i.e. at arbitrary values of the perturbations $\lambda<1$ or $\eps>0$). 
\end{theorem}

\begin{remark}
We did not include in our list the following hard-core modification 
\begin{align}\label{e.HC}
H^{\eps-\mathrm{h.c.}}(\sigma):= - \sum_{i\sim j} \sigma_i \cdot \sigma_j + \infty \cdot 
\sum_{i} 1_{\| \sigma_i - NP\|_2 \wedge \| \sigma_i - SP\|_2 \geq \eps}
\end{align}
(where $NP$ and $SP$ denote the North and South poles of $\S^2$) as this non-smooth perturbation leads to an elliptic disorder on its associated $XY$ model and as such artificially removes the key role played by the north and south poles in the creation of mass in the Heseinberg model as discussed in \cite{aru2025percolation} as well as in the works by Patrascioiu--Seiler \cite{patrascioiu1992phase,patrascioiu1993percolation}.
\end{remark}

\begin{remark}\label{r.metric}
We do not have a solid prediction for the metric perturbation~\eqref{e.metric} but we believe the phenomenology is radically different. See Open Problem \ref{op.1}.  
\end{remark}

\subsection{In dimension $d\geq 3$.}

$ $
Our proof of our main Theorem \ref{th.main} does extend to the case of higher dimension $d\geq 3$ where the BKT regime at low temperature -- specific to $d=2$ -- is replaced by a long-range-order property in the $xy$ plane. See Theorem \ref{th.d3} for the corresponding statement and Section \ref{s.d3} for a sketch of proof and a discussion.

\subsection{Strategy of proof.}
$ $

\margin{C: still in progress, but I will try to finish this in the Rhone Express... $:)$}

The proof of our main Theorem \ref{th.main} relies on the observation that for the classical Heisenberg model, if we rewrite the spin field $(\sigma_i^1,\sigma_i^2,\sigma_i^3)_{i\in \Lambda}$ as follows $(r_i e^{i \theta_i}, \sigma_i^3)_{i\in \Lambda}$ with $r_i := \sqrt{1-(\sigma_i^3)^3}$ (and we assign a uniformly chosen angle $\theta_i$ in case $r_i=0$) then, conditioned on the vertical coordinates $\{ \sigma_i^3\}_{i\in \Lambda}$  the conditional law of the angles $\{ \theta_i \}_{i\in \Lambda}$ is an $XY$ model on $\Lambda$ with inhomogeneous ferromagnetic coupling constants:
\begin{align}\label{e.Jij}
J_{ij}:= \beta r_i r_j  = \beta \sqrt{1- (\sigma_i^3)^2} \sqrt{1-(\sigma_j^3)^3}\,.
\end{align}
I.e. 
\begin{align*}\label{}
\P_{\Lambda, \beta , (r_i)_{i\in \Lambda}}(d\theta) \propto \exp( \beta \sum_{i\sim j} r_i r_j \cos(\theta_i -\theta_j))\,.
\end{align*}

This observation played a key role in the works \cite{patrascioiu1992phase,patrascioiu1993percolation}, see also \cite{aru2025percolation}. 

It turns out that this conditional identity in law still holds for our $\lambda$-modified model defined in Definition \ref{d.modif}. More precisely, the conditional law is exactly the same and does not depend on $\lambda$, only the distribution of the unconditional $\{ \sigma_i^3\}_{i\in \Lambda}$ depends on $\lambda$. 

Our strategy is to focus on the random field $\{ \sigma_i^3 \}_{i \in \Lambda} \sim \P_{\Lambda,\beta, \lambda}$ and to prove that as long as $\lambda<1$ and $\beta > \beta(\lambda)$, this field is ``stochastically'' sufficiently far away from the North and South poles so that the field of random conductances $\{ \beta r_i r_j \}_{i \sim j}$ is ``stochastically large''. 

This is best seen in the easier regime where $\lambda< \lambda_0$ is small enough, for which we find a direct probabilistic proof of \cite{dunlop1985correlation} by proving the following property. For any $p<1$, if  $\beta$ is sufficiently large (larger than $\beta(\lambda, p)$, then under the $\lambda$-modified measure $\P_{\beta, \lambda}$, we have the following exact stochastic domination for the the random set of edges
\[
\{ 1_{|\sigma_i^3| \leq \tfrac 1 2  \text{ and } |\sigma_j^3| \leq \tfrac 1 2} \}_{i\sim j} 
\succcurlyeq 
(p\delta_1 + (1-p)\delta_0)^{\otimes E(\Lambda)}\,,
\]
where $E(\Lambda)$ denotes the edges of $\Lambda$. 

From then on, when $\lambda<\lambda_0$, we end the proof by relying on Ginibre inequality plus \cite{dario2025phase} which shows that the $XY$ model on a $p$-supercritical cluster has powerlaw decay when $\beta\geq \bar \beta(p)$ is large enough. In our case the effective $\beta$ on the $p$ supercritical percolation cluster is lower-bounded thanks to \eqref{e.Jij} by $\beta * \sqrt{1- (1/2)^2} \sqrt{1-(1/2)^2}$. 

This ``percolation proof'' is given in Section \ref{s.short}. This section can be read independently of the rest of the paper. Both proofs (\cite{dunlop1985correlation} and the above one) stop working when $\lambda$ gets closer to 1. 

When $\lambda \nearrow 1$, we proceed as follows. First we fix a well-chosen scale $R=R(\lambda)$ which is required to  diverge as $\lambda \nearrow 1$. Then we analyse the behaviour of the groundstates of $\P_{\Lambda_{R(\lambda)}, ``\beta=\infty'', \lambda}$ \underline{uniformly over all boundary conditions} $\bar \sigma \in (\S^2)^{\p \Lambda_{R(\lambda)}}$. We prove that uniformly over these boundary conditions, the groundstates always have less than $C R^{1/2} \log(R)$ points $i$ inside $\Lambda_{\tfrac R 2}$ whose vertical coordinate satisfy $|\sigma_i^3|> \tfrac 1 2 $. We call this phenomenon: {\em the polar region is sparse}. This is the focus of Section \ref{s.sparse}.  This section is entirely deterministic and we were inspired by the analysis of harmonic maps from a continuum base space $\Omega \subset \R^2$ to spheres $\S^{N-1}$. See also  our companion paper \cite{HarmonicMaps} which pushes this connection between classical harmonic maps and lattice based harmonic maps further.

Then, we prove that when we turn on the temperature, at least when $\beta$ is large enough $\beta\geq \beta(\lambda)$, then uniformly over all boundary conditions, the number of polar sites (with $|\sigma_i^3|> \tfrac 1 2$) is still smaller than $R^{3/4}$ with conditional probability very close to 1. This allows us to coarse-grain the previous picture when $\lambda<\lambda_0$ to all values of $\lambda<1$ by showing that at the scale $R=R(\lambda)$, the set of ``good'' $R(\lambda)$-boxes  --- which contain a connected circuit of lattice edges $i\sim j$  with large conductances  $\beta r_i r_j$ for the projected $XY$ model --- stochastically dominates (when $\beta$ is again chosen large enough) a supercritical site percolation. We conclude using Ginibre inequality together with the coarse-grained analysis from \cite{dario2025phase}.

The structure of this proof is robust enough to extend with minor changes to the case of the half-sphere $\S^2_+$ (for which we analyse the other side of the phase diagram, namely the {\em massive} phase in Section \ref{s.hemisphere}) as well as to the case of the other notions of small deformations of the classical Heisenberg model as listed in Definition \ref{d.modif2}. 

In the special case of $\S^3$, more symmetries and rotations are at hand. In Section \ref{s.CIP} we give a different proof of Theorem \ref{th.main} which is based instead on a correlation inequality from \cite{kunz_inequalities_1976}.

\subsection{Structure of the article.}

In Section \ref{s.short}, we revisit Dunlop's theorem \cite{dunlop1985correlation} for small enough values $\lambda \leq \lambda_0 \ll 1$  with a more probabilistic proof. This section is independent of the rest of the paper. It gives a flavour of the proof in the general case $\lambda<1$. Section \ref{s.sparse} shows that the set of polar points (close to North or South poles) are {\em sparse} (in a quantitative way) for groundstates. This section is entirely deterministic and is inspired by works on harmonic maps as well as Ginzburg-Landau solutions. In Section \ref{sec:proof_of_main_thm}, we show that when temperature is turned on, at least at low temperature, the sparse structure of polar points established for groundstates (at zero temperature) still holds with high probability. Section \ref{sec:proof_of_main_thm} ends the proof of our main Theorem \ref{th.main} by relying on the work \cite{dario2025phase} which analysed the $XY$ model on supercritical percolation clusters.
Section \ref{s.OtherD} extends the proof to other notions of deformations of the classical Heisenberg model. Section \ref{s.CIP} provides a different proof in the case of $\S^3$ building instead on the correlation inequality from \cite{kunz_inequalities_1976}.  Section \ref{s.d3} shortly discusses the cases of higher dimensions  $d\geq 3$. 
Finally, Section \ref{s.hemisphere} analyses the {\em massive phase} (i.e. exponential-decay phase) of the $\lambda$-modified classical Heisenberg model on $\S^2_+$. It shows that the exponential decay proof at $\lambda=1$ and valid at all $\beta$ from the work \cite{bauerschmidt2021geometry} extends to an explicit region of the phase diagram in $(\beta,\lambda)$ which matches with RG predictions.

\medskip
\ni
{\bf Acknowledgments.} 
We wish to thank Fanghua Lin for very inspiring discussions on harmonic maps and  Slava Rychkov for his key insights on the RG analysis of this perturbation of the classical Heisenberg model. We also wish to thank Lucas d'Alimonte, Roland Bauerschmidt, Paul Cahen, Petru Mironescu, Sylvia Serfaty, Stephen Shenker and Daniel Stern for very useful discussions.

The authors acknowledge support from the ERC grant VORTEX 101043450. 

\medskip
\ni 
{\bf Statement of AI use.}
The proofs in this work were completed in 2025 and no AI was used in the process. The results were announced in April 2026. The only use of AI  was as follows:
\bnum
\item Finding the reference \cite{kunz_inequalities_1976}. (N.B. This reference does not play any role in the proofs of the main statements).
\item Doing the simulation leading to Figure \ref{f.simulations}.
\enum
Finally, this version of the  paper was not proofread using AI. 

%
%


\medskip
\noindent
\textbf{Notations.} Throughout this article, we employ the following conventions:

\begin{itemize}
    \item \textbf{Lattice and Domains:} We work primarily on the lattice $\Z^2$ (or $\Z^d$). We write $i \sim j$ to denote that sites $i$ and $j$ are nearest neighbors. For a finite subset $\Lambda \subset \Z^d$, its \textbf{interior} boundary is defined as $\partial \Lambda := \{ j \in \Lambda \mid \exists i \in \Lambda^c \text{ with } i \sim j \}$. The square box of side-length $2n$ centered at the origin is denoted by $\Lambda_n := \{-n, \dots, n\}^d \subset \Z^d$.
    
    \item \textbf{Distances and Norms:} The graph distance (or $\ell_1$-distance) between two sites $a, b \in \Z^d$ is denoted by $d(a,b)$. For spin vectors $u, v \in \R^N$, their Euclidean norm is denoted interchangeably by $|u|$ or $\|u \|$.
    
    \item \textbf{Spin Coordinates:} Spins at a site $i$ are denoted by $\sigma_i \in \S^{N-1}$. We denote by $\sigma_i = (\sigma_i^1, \sigma_i^2,\cdots,  \sigma_i^N)$ for the coordinates of the spin in $\R^n$. For any $1 \le k \le l \le n$, we denote by  $\sigma_i^{k,l} := (\sigma_i^k,\sigma_i^{k+1}, \cdots , \sigma_i^l)$.
    
%
    \item \textbf{Measures and Expectations:} The finite-volume Gibbs measure on a domain $\Lambda$ at inverse temperature $\beta$ with deformation parameter $\lambda$ is denoted by $\mu_{\Lambda, \beta, \lambda}$. The corresponding thermal expectation is denoted using angular brackets, $\langle \cdot \rangle_{\Lambda, \beta, \lambda}$. When the parameters or domain are unambiguous from the context, the subscripts may be dropped (e.g., $\langle \cdot \rangle$). 
    
    \item \textbf{Constants:} The letters $C, c, C', c'$ denote generic, strictly positive constants that may change from line to line. Unless otherwise specified, they depend only on the lattice dimension $d$ and the spin dimension $N$.

    \item \textbf{Asymptotic comparison notations}. If $A$ and $B$ are two real valued functions on $\R$, we write $A \ll B$ in the neighborhood of a point $a$, if $A = O(B)$ in the neighborhood of $A$.
    \item \textbf{Sum over edges}. We write $\sum_{x \sim y \in \Lambda}$ for the sum over \textbf{unoriented edges} of a graph $\Lambda$.
\end{itemize}

%

%
%
%

\section{\for{toc}{A short proof of BKT transition when $\lambda \ll 1$}\except{toc}{A short proof of BKT transition for strong perturbations of $\S^2$}}\label{s.short}

As mentioned above, the existence of a BKT phase in the case of a strong flattening 
--- i.e. $0 \leq \lambda \leq \lambda_1 < 1$ --- 
was established by Dunlop in \cite{dunlop1985correlation}. His proof relies on rather tedious computations and the barrier $\lambda_1<1$ seems like a rigid barrier. The purpose of this section is to give another short proof (with yet another rigid barrier $\lambda_2<1$) which is less computational and more in the spirit of the proof of our main Theorem which goes all the way to $\lambda=\lambda_c=1$.

\subsection{Local ground states steer away from the poles when the flattening is large enough.}\label{ss.sp}

$ $
\smallskip

\begin{figure}[!htp]
\begin{center}
\includegraphics[width=0.7\textwidth]{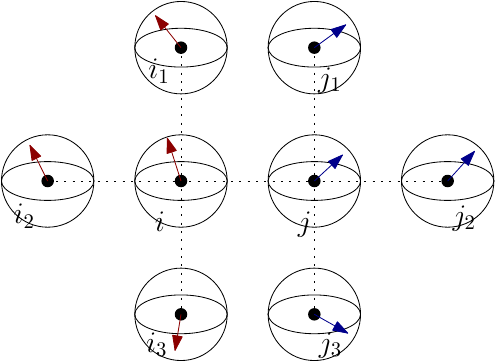}
\end{center}
\caption{}\label{f.sigmabar}
\end{figure}

Let us consider any fixed edge $e=(i,j)$ of $\Z^2$, where $i,j\in \Z^2, i\sim j$. Without loss of generality this may be a horizontal edge as in Figure \ref{f.sigmabar}. Let us label the six neighbouring vertices of $i,j\in \Z^2$ as $i_1,i_2,i_3$ and $j_1,j_2,j_3$. For any spin configuration surrounding the edge $i,j$, i.e. for any $(\sigma_{i_1},\ldots, \sigma_{j_3})\in (\S^2)^6$, let us use the shorthand notation $\bar \sigma$ for $(\sigma_{i_1}, \sigma_{i_2}, \sigma_{i_3}, \sigma_{j_1}, \sigma_{j_2},\sigma_{j_3})$. 
We are interested in the groundstates at sites $i,j$. I.e. the set of configurations in $(\S^2)^2$ defined by 
\begin{align*}
\mathrm{GS}_\lambda(\bar \sigma):=\mathrm{Argmin}\left\{ (\sigma_i,\sigma_j) \in (\S^2)^2 \mapsto  \Big(\sum_{k=1}^3 \<{\sigma_i, \sigma_{i_k}}_\lambda +
 \<{\sigma_j, \sigma_{j_k}}_\lambda \Big) + \<{\sigma_i,\sigma_j}_\lambda \right\} 
\end{align*}
N.B. The set $\mathrm{GS}_\lambda(\bar \sigma)$ may contain several groundstates, but {\em generically} of course $|\mathrm{GS}_\lambda(\bar \sigma)| =1$.

\begin{proposition}\label{pr.EasyCase}
If $\lambda$ is small enough, then there exists $\delta>0$ so that uniformly on the local boundary condition $\bar \sigma \in (\S^2)^6$, the groundstates spins stay $\delta$ away from both the south and north poles of $\S^2$. (resp. $N=(0,0,1)$ and $S=(0,0,-1)$)  i.e. 
\begin{align*}
\min_{\bar \sigma \in (\S^2)^6 } \min_{(\sigma_i,\sigma_j)\in \mathrm{GS}_\lambda(\bar \sigma)} \dist(\sigma_i, N\cup S) \wedge \dist(\sigma_j, N \cup S) \geq \delta\,.
\end{align*}
\end{proposition}

Before proving this proposition, let us explain briefly why it implies our main Theorem \ref{th.main} when $\lambda$ is chosen sufficiently small. (More details will be given in later sections \ref{s.TurnOn} and \ref{sec:proof_of_main_thm} on how to conclude the proof from an analysis of the associated groundstates). 

For such a sufficiently small $\lambda$, the Proposition implies that uniformly in the six spins around an edge, the $\lambda-$groundstates stay $\delta$ away from the poles. When turning on the temperature, it is not hard to deduce from this Proposition and from the local structure of the Gibbs measure that if $\beta$ is chosen large enough, then with a probability $\bar p=\bar p(\beta)$ which tends to 1 as $\beta \to \infty$, then uniformly over boundary conditions around a given edge, the two $\S^2$ spins on the inside edge are at least $\frac \delta 2$ away from the poles with probability at least $\bar p=\bar p(\beta)$. 
This step which goes from the behaviour of groundstates to the high probability behaviour of the local Gibbs states is proven in the general case (for balls $\Lambda_R$ of arbitrary large radius $R$) in Section \ref{s.TurnOn}.

 Now, for any $p>p_c(\Z^2)$, a celebrated stochastic domination with product measures from \cite{liggett_domination_1997} shows that if $\bar p=\bar p(\beta)$ is sufficiently close to 1, then the set of edges whose endpoints are $\frac \delta 2$ far from the poles stochastically dominates the $p$-product measure overs edges of $\Z^2$. This is explained in the general setting in subsection \ref{ss.LSS}. We may then conclude from Ginibre inequality and from the result of \cite{dario2025phase} which shows that BKT phase a.s. holds on supercritical percolation cluster if the effective temperature (here $\tfrac 1 4 \delta^2 * \beta$ is chosen large enough).

\smallskip
\ni
{\em Proof of Proposition \ref{pr.EasyCase}.}

\ni
\textbf{Step 1. }
We first notice that it is sufficient to prove the above property when $\lambda=0$.

Indeed, let us assume that there exists $\delta_0>0$ such that 
\begin{align*}
\min_{\bar \sigma \in (\S^2)^6 } \min_{(\sigma_i,\sigma_j)\in \mathrm{GS}_0(\bar \sigma)} \dist(\sigma_i, N\cup S) \wedge \dist(\sigma_j, N \cup S) \geq \delta_0\,.
\end{align*}
This implies that if we define the following function of $\bar \sigma \in (\S^2)^6$ and $\lambda\in [-1,1]$, 
\begin{align*} 
G(\bar\sigma,\lambda)& := \max_{\substack{(\sigma_i,\sigma_j) \in (\S^2)^2, \\ \dist(\sigma_i, N\cup S) \wedge \dist(\sigma_j, N \cup S) \geq \delta_0}} 
\Big(\sum_{k=1}^3 \<{\sigma_i, \sigma_{i_k}}_\lambda +
 \<{\sigma_j, \sigma_{j_k}}_\lambda \Big) + \<{\sigma_i,\sigma_j}_\lambda \\
& \,\,\, - 
 \max_{\substack{(\sigma_i,\sigma_j) \in (\S^2)^2\\, \dist(\sigma_i, N\cup S) \wedge \dist(\sigma_j, N \cup S) \leq \frac{\delta_0} 2}} 
\Big(\sum_{k=1}^3 \<{\sigma_i, \sigma_{i_k}}_\lambda +
 \<{\sigma_j, \sigma_{j_k}}_\lambda \Big) + \<{\sigma_i,\sigma_j}_\lambda\,,
\end{align*}
then for any fixed $\bar \sigma$, we have $G(\bar \sigma, \lambda=0) >0$. 
By compactness of $(\S^2)^6 \times [-1,1]$, and continuity of $\bar \sigma, \lambda \mapsto  G(\bar \sigma, \lambda)$ this implies that there exists $\lambda_0>0$ such that 
\begin{align*} 
\min_{\bar \sigma \in (\S^2)^6, |\lambda| \leq \lambda_0} G(\bar \sigma, \lambda) >0\,,
\end{align*}
which in turn implies that none of the groundstates in $\mathrm{RG}_\lambda(\bar \sigma)$ can be at distance less than $\frac {\delta_0} 2$ from the set $\{ N \cup S\}$. This concludes the proof of the Proposition in the case $\lambda=0$.

\medskip
\ni
\textbf{Step 2.} {\em Analysis of the flat case $\lambda=0$.}
The advantage of $\lambda=0$ is that vertical components no longer contribute. In particular, instead of considering spins $\sigma$ in the sphere $\S^2$, it is sufficient to view spins as elements of the unit disk $\D \subset \R^2$.  To $\sigma=(\sigma^1,\sigma^2,\sigma^3)$, we associate with a slight abuse of notation the point $\sigma:=(\sigma^1,\sigma^2) \in \D$. Our six neighbouring configurations are now six vectors in $\D$ which we still call $\bar \sigma= (\sigma_{i_1},\ldots, \sigma_{j_3}) \in \D^6$.  Let us suppose by contradiction that Proposition \ref{pr.EasyCase} fails when $\lambda=0$. This means there exists some neighbouring configuration $\bar \sigma(0)$ for which either $\sigma_i$ or $\sigma_j$ --- when $(\sigma_i,\sigma_j)\in \mathrm{GS}_0(\bar \sigma(0))$ --- sits at the North or South pole. In our present setup, this means either $\sigma_i$ or $\sigma_j$ sits at the origin $0=(0,0)\in \D$. Without loss of generality, let us assume this is the right spin, i.e. $\sigma_j =0$, for the configuration $\bar \sigma(0)$.
Let us introduce the following two vectors of $\R^2$
\begin{align*} 
\begin{cases}
\bar \sigma_i(0):= \sum_{k=1}^3 \sigma_{i_k}(0) \in \R^2 \\
\bar \sigma_j(0):= \sum_{k=1}^3 \sigma_{j_k}(0) \in \R^2
\end{cases}
\end{align*}
Without loss of generality again -- possibly by rotating all spins in $\bar \sigma(0)$ -- we may assume $\sigma_j(0) \in \R_+$ (i.e. $\sigma_j(0) = (\|\sigma_j(0)\|_{\R^2} , 0)$). 

Under these hypotheses, it is clear where the optimal vector $\sigma_i$ must be. Recall we assume $(\sigma_i,\sigma_j)\in \mathrm{GS}_0(\bar \sigma(0))$ and furthermore $\sigma_j=0$. $(\sigma_i,\sigma_j) \in \D^2$ is maximizing 
\begin{align*} 
(\sigma_i,\sigma_j) \mapsto \<{\sigma_i, \bar \sigma_i(0)}_{\R^2} + \<{\sigma_j, \bar \sigma_j(0)}_{\R^2} + \<{\sigma_i, \sigma_j}_{\R^2}  
 = \<{\sigma_i, \bar \sigma_i(0)}_{\R^2} \,,
\end{align*}
since we assumed $\sigma_j=0$. 
We now distinguish two cases.

\ni
\underline{Case 1.}  If $\bar \sigma_i(0) \in \R^2 \setminus 0$. In this case, to maximize the above quantity, we must choose $\sigma_i:=  \frac {\bar \sigma_i(0)} {\| \bar \sigma_i(0) \|} \in \p \D$.  We will now find a contradiction by finding a better groundstate than this pair $(\sigma_i,\sigma_j=0)$. Indeed, if $\sigma_i = e^{i \theta} \in \p \D$, then 
\bi
\item[a)] If $\theta\in (0,\pi)$, notice that by setting $\sigma_j:= (0,1) = e^{i \pi/2}$ (instead of $\sigma_j=0$), we gain a higher value for $\<{\sigma_i, \bar \sigma_i(0)}_{\R^2} + \<{\sigma_j, \bar \sigma_j(0)}_{\R^2} + \<{\sigma_i, \sigma_j}_{\R^2}$  (the only term which changes is $\<{\sigma_i, \sigma_j}_{\R^2} = \<{ e^{i \theta}, e^{i \pi /2}}_{\R^2} =  \sin(\theta)>0$ which  gives us a contradiction. 
\item[b)] If $\theta \in(\pi, 2\pi]$, then we instead set $\sigma_j:=(0,-1)$ to reach a similar contradiction.
\item[c)] If $\theta = 0$, setting $\sigma_i = \sigma_j = (1,0)$ increases the quantity by at least $\<{\sigma_i, \sigma_j} = 1$ which is a contradiction.
\item[d)] Finally, if $\theta=\pi$, we need to reset both values of $\sigma_i$ and $\sigma_j$, for example:
\begin{align*} 
\sigma_i := e^{i (\pi - \alpha)} \text{  and  } \sigma_j := (0,1)
\end{align*}
If $\alpha$ is small, this slight change of angle looses at most $O(\alpha^2)$ in the scalar product $\<{\sigma_i , \bar \sigma_i(0)}_{\R^2}$ but is gaining $ \geq  c \alpha $ in the scalar product $\<{\sigma_i,\sigma_j}_{\R^2}$ which thus also leads to a contradiction. 
\ei

\ni
\underline{Case 2.}  If $\bar \sigma_i(0) = 0$. We reach a contradiction in this case by noticing that by setting both $\sigma_i$ and $\sigma_j$ to be both equal to $(1,0)$, we obtain a strictly better energy.

\qed

\subsection{Why does this viewpoint break when $\lambda$ gets closer to 1 ?}
$ $ 

This is due to the following property (keeping the same setup and notations as in the previous proof). 
\begin{lemma}\label{l.VP1}
If the 6 neighbours ($i_1,\ldots, j_3$) are pointing in the North direction, then if $\lambda$ is close enough to 1, the ground state at $i\sim j$ also needs to point in the north direction.
\end{lemma}

This is easy to prove. The next Lemma will indicate a stronger rigidity property as $\lambda \nearrow 1$ which justifies why our analysis below in Sections \ref{s.sparse}  to \ref{sec:proof_of_main_thm} will require to consider larger and larger domains as $\lambda \nearrow 1$.


\subsection{Why is it necessary to consider large domain $\Lambda_R$ as $\lambda \nearrow 1$ ?}
$ $



\begin{lemma}\label{l.NP}
For any $\lambda < 1$, there exists a scale $R(\lambda) \asymp \frac{1}{\sqrt{(1-\lambda) \log \frac 1 {1-\lambda}}} $ (N.B. which goes to $\infty$ as $\lambda\to 1$) which is such that for any scale $R\leq R(\lambda)$ if one imposes a North Pole ($NP$) boundary condition along $\p \Lambda_R$, then there is a unique ground state for the $H_\lambda$ energy functional inside $\Lambda_R$, and it consists in the constant North pointing field  $\sigma(x):=  NP 1_{\Lambda_R}(x)$.  
\end{lemma}

\begin{remark} 
A consequence of this Lemma is that 
\bi
\item If we were to follow the same idea for $\lambda$ very close to 1 as in Subsection \ref{ss.sp},  we would  need to solve the variational principle analogous to Proposition \ref{pr.EasyCase} but now uniformly on a much larger set, i.e. uniformly over all configurations $\sigma  \in (\S^2)^{\p \Lambda_R}$. This seems to be a very difficult task. 
\item It also implies that for the different technique we shall follow in the next two sections, it is in any case necessary to work in diverging size local domains as $\lambda \nearrow 1$. For example the main engine behind the proof, i.e. Proposition  \ref{pr.GS_more_general_on_sites} below,  
necessarily requires that the constant $C(\lambda)$ diverges with $\lambda \nearrow 1$. 
\ei
\end{remark}

Since this lemma is only needed to justify the use of diverging scales $R(\lambda) \nearrow \infty$ as $\lambda \to 1$, we defer its proof to Appendix \ref{a.NP}.

\section{The polar region is sparse}\label{s.sparse}
\label{sec:polar_lines} 
There is no randomness involved in this section. Its purpose is to provide a type of weak regularity property for the groundstates of $\lambda$-Hamiltonian (as defined in Definition \ref{d.modif}) uniformly in the boundary condition $\sigma\in (\S^{N-1})^{\p \Lambda}$ for a box $\Lambda \in \Z^d$. We change the definition of $H_{\lambda}(\sigma)$ by a constant amount to guarantee that $H_{\lambda}(\sigma) = 0$ if $\sigma = Cste$. This is more natural from an analytic point of view. Namely, we set
\[
H_{\lambda}^{\Lambda}(\sigma) = \frac{1}{2}\sum_{\substack{i \sim j \\ i , j \in \Lambda}}  \left |\sigma_i - \sigma_j \right |^2 + (1-\lambda) \sum_{l = 3}^{N} \sigma_i^l \sigma_j^l ~.
\]
We define the interaction along one edge
\begin{align}\label{e.fl}
f_{\lambda}(u,v) = \frac{1}{2} \left |u- v \right |^2 +   (1-\lambda) \sum_{l = 3}^{N} u^l v^l\,,
\end{align}
for any $u,v \in \mbb{B}^{N}$. Notice that if $\lambda\in[-1,1]$, then
\begin{equation} \label{eq: lower_bound_on_f_eps}
0 \le \min(1-\lambda, 1+ \lambda)  f_{0}(u,v) \le f_{\lambda}(u,v) \le 2~.
\end{equation} 
Indeed,
$ f_{\lambda}(u,v) = 1 - u^1 v^1 - u^2 v^2 - \lambda \sum_{l=3}^N u^l v^l$ and  $| \sum_{l=3}^N u^l v^l| \le 1 - u^1 v^1  - u^2 v^2 = f_{0}(u,v) $.
%
With this convention the Hamiltonian of any configuration in $(\mbb{B}^{N})^{\Lambda}$ is thus a number between $0$ and $2 |E(\Lambda)|$. If $A$ is any subset of $\Lambda$, we denote by
\[
H^A(\sigma) = H^A_\lambda(\sigma) := \sum_{  i , j \in A \,; \, i \sim  j } f_{\lambda}(\sigma_i , \sigma_j).
\]

If $u  = (u^1, u^2, \dots , u^N) \in \R^N$ and $1 \le k \le l $, we denote $u^{k,l} = (u^{k}, , u^{k+1}, \dots, u^{l})$. If $u$ and $v$ are in $\R^N$, we denote $u^{k,l} \cdot v^{k, l } = \sum_{j = k}^l u^j v^j$.



\begin{definition}[Ground state configuration]
  \label{def.gs}
  For any finite subset $\Lambda \subset \Z^2$, we say that  a configuration $\sigma \in (\S^{N-1})^{\Lambda}$ is a ground state configuration if $\sigma$ is a minimizer of $H_{\lambda}^{\Lambda}$ among configurations that coincides with $\sigma$ on $\p \Lambda$. That is \[ H_{\lambda}(\sigma) = \min_{\sigma' \in (\S^{N-1})^{\Lambda}, \,  \sigma'_{|\p \Lambda} = \sigma_{|\p \Lambda}} H_{\lambda}^{\Lambda}(\sigma') ~.\]
  If $\tau \in (\S^{N-1})^{\p\Lambda}$ is a boundary condition we say that $\sigma$ is a ground state configuration for $\tau$ if $\sigma_{|\p \Lambda} = \tau$ and $\sigma$ is a ground state configuration.
\end{definition}
Ground state configurations exist for any boundary condition by compactness of $(\S^{N-1})^{\Lambda}$.


\begin{proposition}\label{pr.GS_more_general}
For any $\delta > 0$ small enough, any $\lambda \in(-1,1)$, there exists a constant $C$ such that for any $R \ge 1$, uniformly in the boundary condition $\tau \in (\S^{N-1})^{\p \Lambda_R}$, any ground state configuration $\hat \sigma$
is such that
\[H^{\Lambda_{R/2}}_\lambda(\hat \sigma) \le \frac{C}{(1+\lambda)^{1/2} \wedge (1 - \lambda)^{1/2}} R^{d - \frac 3 2} \log(R) ~.\]


\end{proposition}
We say that a point $\sigma \in \R^{N}$ is in $\mbb{S}^1$ if $\sigma \in \mbb{S}^1 \times \{0\}^{N-2}$. As we mostly have in mind the $ N = 3$ case, we will refer to the set of points $u$ that have
$(u^1)^2 + (u^2)^2 < \delta^2 $
as being the  polar regions of $\mbb{S}^{N-1}$, and we will denote this region of $\S^{N-1}$ by $\mathcal{P}_\delta$. Proposition \ref{pr.GS_more_general} implies an upper bound on the number of points in the polar regions $\mathcal{ P}_\delta$. 

\begin{proposition}\label{pr.GS_more_general_on_sites}
For any $\delta > 0$ small enough, any $\lambda\in(-1,1)$, any $R \ge 1$, uniformly in the boundary condition, any ground state configuration $\hat \sigma_{\Lambda}$
is such that few points in $\hat \sigma_\lambda$ are in $\mathcal{P}_{\delta}$ in the bulk of $\Lambda_R$. More quantitatively,
\begin{align}
\sup_{\sigma_{|\p \Lambda_R} \in (\S^{N-1})^{\p \Lambda_R}}\sharp \Big\{ x\in \Lambda_{{R}/2} \text{ s.t. } \hat \sigma_x \in \mathcal{P}_{\delta}  \Big\}  \leq C(\lambda, \delta, N, d ) R^{d - \frac 3 2 } \log(R) ~.
\end{align}
\end{proposition}

\begin{remark}
  One may choose $C( \lambda, \delta, N, d ) = \frac{C(N,d)}{( 1 - \delta^2)^{1/2} \delta \left ((1+\lambda)^{3/2} \wedge (1 - \lambda)^{3/2}\right ) }$
\end{remark}

Proposition \ref{pr.GS_more_general} implies Proposition \ref{pr.GS_more_general_on_sites} thanks to   \begin{align}\label{eq:energy_vs_points}
\#\{e = \{i,j\}  \text{ in } A \mid \sigma_i \in \mathcal{P}_{\delta} \text{ or } \sigma_j \in \mathcal{P}_{\delta}\} &\le \frac{1}{1 - \delta^2} H_0^{A} (\sigma) \nonumber \\ &\le  \frac{1}{(1 - \delta^2) * (1+\lambda) \wedge (1- \lambda)} H_{\lambda}^A(\sigma)  ~. 
\end{align}

Equation (\ref{eq:energy_vs_points}) follows from

\begin{align*}
 H_\lambda(\sigma ) & \ge (1+\lambda) \wedge (1 - \lambda)  H_{0}(\sigma) \\
  & \ge (1+\lambda) \wedge (1 - \lambda)   \sum_{i \sim j, \sigma_i \text{ or } \sigma_j \in P_\delta  } (1  - \sigma^{1,2}_i \cdot \sigma_j^{1,2} ) \\
  &  \ge (1+\lambda) \wedge (1 - \lambda)  *  (1 - \delta^2) \,\, \#\{e = \{i,j\} \mid \sigma_i \in \mathcal{P}_{\delta} \text{ or } \sigma_j \in \mathcal{P}_{\delta}\} ~.
\end{align*}
where the first inequality is Equation (\ref{eq: lower_bound_on_f_eps}).

The proof of this result is divided into five main steps  which each correspond to the next subsections.

\subsection{Step1. Exploration until a square  with minimal energy is found (Fubini step).}
\label{s.Fubinistep}

\begin{lemma}[Some layers have small Hamiltonian]
 For any $n\geq 1$, let $\sigma: \Lambda_n \to \S^{N-1}$ be any ground state configuration. There exists an integer $k$ that satisfies  $9/10 n \le  k \le n $ such that $H^{\partial \Lambda_{k}}_{\lambda}(\sigma) \le C n^{d-2}$ for some constant $C$ that depends only on $d$ and not on $\lambda$ or $N$.
\end{lemma}
\begin{proof}
This is a straightforward application of the discrete and finite Fubini theorem.
First, notice that $H_{\lambda}(\sigma) \le 4 d (2n+1)^{d-1}  $. Indeed, if we set $\sigma' = (1, 0 , \dots, 0)$ in $\Lambda_{n-1}$, then $H_{\lambda}(\sigma) \leq H_{\lambda}(\sigma') \le 4d ( 2 n + 1)^{d-1}  $ since there are $4\cdot (2 n +1)^{d-1}$ edges between $\partial \Lambda_n$ and $\partial \Lambda_{n-1}$ and since $f_\lambda \leq 2$. 
As $f_{\lambda}(u,v) \ge 0$, 
\[ 8 d (2n+1)^{d-1}   \ge H_{\lambda}(\sigma) \ge \sum_{s = \frac{9 n }{10} }^{n} H_{\lambda}^{\partial \Lambda_s }(\sigma) \ge \frac{n}{10} \min_{\frac{9 n }{10 } \le k \le n }  H_{\lambda}^{\partial \Lambda_{k} }(\sigma) ~.  \]
\end{proof}

\subsection{Step2. Projection from $\S^{N-1}$ to $\S^1$ with small energy cost.}
Let $\delta > 0$. If $u \in \S^{N-1}$, we define
\[
p(u) = \begin{cases} \frac{1} {|u^{1,2}|}(u^1, u^2, 0, \cdots, 0) \text{ if } 
(u^1)^2 + (u^2)^2 \ge \delta^2 \\
   ( 1, 0, \cdots, 0 ) \text{ otherwise}
\end{cases}
\]


For any configuration $\sigma: \partial \Lambda_n \to \mbb{S}^{N-1}$,  we define $\sigma': \partial \Lambda_n \cup \partial \Lambda_{n-1} \to \S^{N-1}$ by $\sigma'(x) = \sigma(x)$ for $x \in \partial \Lambda_n$, and $\sigma'(y) =  p(\sigma(x))$ for $y$ in $\partial \Lambda_{n-1}$ which has $x$ as its unique neighbor in $\p \Lambda_n$. For the $2^d$ corners of $\Lambda_{n-1}$ (which are the only points in $\Lambda_{n-1}$ to have more than one neighbor in $\Lambda_n$), we set $p(\sigma_x)  = (1, 0, \cdots , 0)$ arbitrarily, this choice does not matter as the number of corners does not depend on $n$
 
\begin{lemma}[One layer projection to $\mbb{S}^1$]
There exists a constant $C$ that depends only on $N$, $d$ and $\delta$ 
such that for any function $\sigma: \partial \Lambda_n \to \mbb{S}^{N-1}$, the function $\sigma': \partial \Lambda_n \cup \partial \Lambda_{n-1} \to \mbb{S}^{N-1}$ defined above and that is such that $\sigma'_{|\Lambda_{n-1}} \in \S^1$ satisfies
\[
H_{\lambda}^{\p \Lambda_n \cup \partial \Lambda_{n-1}} (\sigma') \le \frac{C}{(1+\lambda) \wedge (1 - \lambda)} H_\lambda^{\p \Lambda_n }(\sigma) + 2 d 2^d \, .
\]
\end{lemma}

\begin{remark}
  One may choose $C  = \frac{C(N,d)}{\delta^2 (1 - \delta^2)}$.
\end{remark}

\begin{proof} 

By Equation (\ref{eq:energy_vs_points}), it suffices to establish that $H_\lambda^{\p \Lambda_n \cup \partial \Lambda_{n-1}}(\sigma') \le C H_{0}^{\p \Lambda_n}(\sigma)  + 2 d 2^d $.
  Let us first analyse the contribution of the edges inside $\partial \Lambda_{n-1}$. Note that as $p(u) \in \mbb{S}^1$ for any $u$, this contribution of the edge $u \sim v$ in $\p \Lambda_{n-1}$ is the sum of gradients of $p(\sigma)$ along edges. If $u$ and $v$ are not in $P_{\delta} $, 
\begin{align*}
 |p(u) - p(v)|^2 &= \left | \frac{u^{1,2} - v^{1,2}}{|u^{1,2}|} + v^{1,2} \left ( \frac{1}{|u^{1,2}|} -  \frac{1}{|v^{1,2}|} \right )\right |^2\\ 
 &  \leq \frac 2 {\delta^2} \bigl(   |u^{1,2} - v^{1,2} |^2  +   (|u^{1,2}| - |v^{1,2}|)^2\bigr)  \\
 &   \leq \frac 4 {\delta^2} \bigl(   |u^{1,2} - v^{1,2} |^2  \bigr)  \\
 & =  \frac 8 {\delta^2}  f_{\lambda=0}(u,v)  
 \end{align*}
(recall $f_\lambda(u,v)$ was introduced in Equation \eqref{e.fl}).


Thus, for any $i,j \in \Lambda_{n-1}$ that both have a unique neighbor in $\partial \Lambda_n$,

\[ |p(\sigma_{x_i}) - p(\sigma_{x_j})|^2 \le  \frac{8}{ \delta^2} f_{0}(x_i,x_j) + 2 \1_{(\sigma_{x_i} \in \mathcal{P}_\delta \text{ or } \sigma_{x_j}  \in \mathcal{P}_\delta)} \] where $x_i$ and $x_j$ are the neighbors of $i $ and $j$ respectively.

If $i$ is a corner of $\partial \Lambda_{n-1}$ and $j$ is a neighbor of $i$ in $\partial \Lambda_{n-1}$, then we simlply bound the energy of the edge by $2$, as there are $2^d$ corners that all have $d$ neigbors, we can bound the contribution of these edges by $2 d 2^d$. 


 Thus, the contribution of edges in $\partial \Lambda_{n-1}$ is bounded as follows
\[
\sum_{\substack{i, j \in \partial \Lambda_{n-1} \\ i \sim j }} |\sigma'_i - \sigma'_j|^2 \le  \frac{8}{ \delta^2} H_0^{\p \Lambda_n} (\sigma) + 2 \#\{e = \{i,j\} \mid \sigma_i \in \mathcal{P}_{\delta} \text{ or } \sigma_j \in \mathcal{P}_{\delta}\} + 2 d 2^d~. 
\]

Now by Equation (\ref{eq:energy_vs_points}),
\[ \#\{e = \{i,j\} \mid \sigma_i \in \mathcal{P}_{\delta} \text{ or } \sigma_j \in \mathcal{P}_{\delta}\} \le \frac{1}{1 - \delta^2} H_{0}^{\p \Lambda_n} (\sigma) ~. \]

Let us now analyze the contribution of the interaction between the outer and inner boundary.
Let $i \in \partial \Lambda_{n-1}$ that is not a corner. Let $j$ be the unique neighbor of $i$ in $ \partial \Lambda_{n}$. If $\sigma_j \in \calP_\delta$, we bound the contribution of the $i \sim j $ edge by $2$ as we have already bounded the number of such $i$ by $C H_0 (\sigma)$. If $\sigma_j \notin \calP_\delta $, we have   
\begin{align*}
\frac{1}{2} |\sigma_j - p(\sigma_i )|^2 + (1-\lambda)  \sum_{l \ge 3} \sigma_j^l p(\sigma_j)^l &= \frac{1}{2}|\sigma_j - p(\sigma_j )|^2  \\ 
&= \frac{1}{2}\left(\frac{1}{\sqrt{1 - |\sigma_j^{3,N}| ^2}} - 1 \right)^2 |\sigma_j^{1,2}|^2 + |\sigma_j^{3,N}|^2 \\ 
&\le  \left(\frac{1}{\sqrt{1 - |\sigma_j^{3,N}| ^2}} - 1 \right)^2  +|\sigma_j^{3,N}|^2 \\
& \le C_1  |\sigma_j^{3,N}| ^4  + C_2 |\sigma_j^{3,N}|^2 \le C' |\sigma_j^{3,N}| ^2  
\end{align*}

Now it holds that \[ \sum_{  j \in \p \Lambda_n}   |\sigma_{j}^{3,N}|^2 \le 2 H_0 (\sigma) \]


Indeed, for any $j\sim k$ along $\p \Lambda_n$,
\[
f_{0}(\sigma_j, \sigma_k)  = 1 -  \sigma_j^{1,2} \cdot \sigma_k^{1,2}  \ge 1 - |  \sigma_j^{1,2}| = \frac{1 - |  \sigma_j^{1,2}|^2 }{1 + |  \sigma_j^{1,2}|} \ge \frac{1}{2}  |\sigma_j^{3,N}|^2~.
\]
Finally, as before, if $i$ is in a corner of $\Lambda_{n-1}$, we bound the contribution of its $d$ neighbors in $\Lambda_n$ by $2$.

\end{proof}

\subsection{Step3. Bounds on the Dirichlet Energy of the harmonic extension in $\B^{N}$}

If $\psi: \Lambda \to \R^{N}$ and $A \subset \Lambda $, we denote $\text{DE}_A = \frac 1 2 \sum_{e \in E(A)} |\psi_i - \psi_j|^2$ where $E(A) = \{ \{i,j\} \mid i \in A, j \in A, i \sim j  \}$ for the Dirichlet energy of a configuration.

If $\Lambda \subset \Z^d$, let $\Lambda^{\text{int}}$ denote the interior of $\Lambda$ that is the set of points $x \in \Lambda$ that have all their neighbors in $\Lambda$.
We say that $\psi: \Lambda\to  \R^N$ is harmonic on $\Lambda$ if for all $x \in \Lambda^{\text{int}}$, $\psi(x)= \sum_{y \in \Lambda, y \sim x } \psi(y)$. We denote by $\Delta$ the Laplace operator on $\Lambda$. That is for $x \in \Lambda$, $\Delta \psi (x) = \sum_{y \sim x} (\psi(y) - \psi(x))$. Harmonic functions are thus functions whose Laplace operators is $0$ in the interior of $\Lambda$. 
\begin{lemma}[Harmonic extensions of boundary conditions] Let $\psi: \Lambda_n \to \R^{N}$ be  harmonic inside $\Lambda_n$.
Then 
$\text{DE}_{\Lambda_n}(\psi) \le C n^{(d-1)/2} \sqrt{\text{DE}_{\partial \Lambda_n }(\psi)} \|\psi\|_\infty$ for some 
$C$ that depends only on $N$ and $d$.\end{lemma}\label{lemma:harmonic_extension}

\begin{proof}
 As every coordinate of a harmonic function $\psi: \Lambda \to \R^{N}$ is harmonic, we only need to deal with real valued $\psi$.
Notice first that by harmonicity, if $x \in \Lambda^{\text{int}}$, \[ \psi(x) = \E_x ( \psi(S_\tau)) = \sum_{z \in \partial \Lambda} K(x,z) \psi(z)  \] where $(S_m)_{m \geq 0}$ is a simple random walk on $\Z^d$ and $\tau $ is the hitting time of the set $\partial \Lambda $ and $K(x,z)$ is the Poisson kernel of the Simple Random Walk (SRW) in $\Lambda$ defined by $K(x,z):=\P_x[S_\tau = z]$. 

Thus, by a discrete integration by part,
\begin{align*}
2 \text{DE}(\psi) & = \sum_{x \in \partial \Lambda } - \psi(x) (\Delta \psi)(x) \\ 
&= \sum_{x \in \partial \Lambda } \sum_{y, y \sim x } \psi(x) (\psi(x )  - \psi(y))\\ 
&= \sum_{
x \in \partial \Lambda 
}
\left[\sum_{y \in \partial \Lambda , y \sim x } \psi(x) (\psi(x )  - \psi(y))\right] + \sum_{\substack{x \in \partial \Lambda \\ x \text{ is not a corner} }} \psi(x) (\psi(x )  - \psi(y_x)) 
\end{align*}
 where $y_x$ is the neighbor of $x$ in $\Lambda^{\text{int}}$. Notice that corners of $\partial \Lambda$ do not have any neighbours in $\Lambda^{int} := \Lambda\setminus \p \Lambda$.

Note that, 
\[ \sum_{ x \in \partial \Lambda} \sum_{y \in \partial \Lambda , y \sim x } \psi(x) (\psi(x )  - \psi(y))  = \sum_{ x \sim y \in \partial \Lambda} (\psi(x) - \psi(y))^2 \]  

and 

\[   (\psi(x )  - \psi(y_x)) = \sum_{z \in  \partial \Lambda } K(x,z) (\psi(x) - \psi(z))  \]

Thus, by another integration by part
\[ \text{DE}(\psi ) = \text{DE}_{\partial \Lambda }(\psi) + \frac{1}{2} \sum_{x, z \in \partial \Lambda } K(x,z) (\psi_x- \psi_z)^2 ~. \]

Now by Lemma \ref{lemma:poisson_kernel}, $K(x,z) \le \frac{C}{|x -z|^d} $

\[ \sum_{x, z \in \partial \Lambda } K(x,z) (\psi_x- \psi_z)^2  \le C \sum_{x, z \in \partial \Lambda}  \frac{|\psi(x) - \psi(z)|^2}{|x -z|^d} \]

To finish the proof we use the following fact.
\begin{lemma} There exists a constant $C  = C(d)$  such that for any cubic box $\Lambda$, for all $k \ge 1$  and for any $x,y \in \partial \Lambda$ that are at distance $k$ in $\partial \Lambda$, one can choose some path  $\gamma_{x y}$ of minimal length inside $\partial \Lambda$ between $x$ and $y$, 
in such a way that every edge appears in at most $C k^{d-1}$ such paths.  \end{lemma}
\label{lemma:wellchosenpath}

\begin{proof} For any $x, y \in \partial \Lambda$, we will choose $\gamma_{x y}$ that goes through at most three faces of {$\partial \Lambda$}. If two faces are chosen, they must of course be adjacent. We only allow three faces to be used if $x$ and $y$ are on opposite faces of $\partial \Lambda$. For any $x,y$ on opposite faces, we choose the intermediary face that connects the face $x$ and the face of $y$ arbitrarily. The union of $1$, $2$, or $3$ faces with one opposite to the other is isomorphic as a graph to a hyperrectangle with size lengths $n$, $2n$ or $3n$. We then choose $\gamma_{x y }$ as follows. Starting from $x$, we move coordinates by coordinates in increasing order of the coordinates (where we extend the coordinates of the face of $x$ on the other faces of the path in the unique way that makes the union of the faces a standard rectangle).
A path of length $k$ going through a given edge $e$ in a given rectangle is then characterized by the times the path changes direction and the position of the edge $e$ in this path. This corresponds to an increasing sequence of $d-2$ numbers between $1$ and $k$ for the turns of which there are at most $ k^{d-2}$ plus the choice of the position of the edge $e$ in the path for which there is at most $k$ possibility.
As each face appears in at most $4 d +1$ different rectangles, there are at most $C k^{d-1}$ paths going through $e$. \end{proof}

We then write  $\psi(x) - \psi(z) = \sum_{e \in \gamma_{x z }} \nabla_e \psi $ and obtain

\begin{align*}
\sum_{x, z \in \partial \Lambda } K(x,z) (\psi_x- \psi_z)^2  &\le C \sum_{x, z \in \partial \Lambda}  \frac{|\sum_{e \in \gamma_{x z} } \nabla_e \psi|^2}{|x -z|^d}  \\ 
& = C \sum_{k  = 1 }^{ 2 n } \sum_{x \in \partial \Lambda} \left(\sum_{\substack{ z \in \partial \Lambda \\  |x - z | = k }} \frac{1}{k^d} \left| \sum_{e \in \gamma_{x z} } \nabla_e \psi \right|^2\right) 
\end{align*}

We then split the sum on $k$ at some well-chosen integer $A$ that we will specify.

By Cauchy--Schwarz inequality, 
\begin{align*}
\sum_{1 \le k  \le A  } \frac{1}{k^d}  \sum_{\substack{x, z \in \partial \Lambda \\  |x - z | = k }} \left|\sum_{e \in \gamma_{x z} } \nabla_e \psi \right|^2  & \le    \sum_{k  = 1 }^{ A } \frac{1}{k^{d}}  \sum_{\substack{ x,z \in \partial \Lambda \\  |x - z | = k }}  k \sum_{e \in \gamma_{x z} } (\nabla_e \psi )^2 \\ 
&\le C  A  \text{ DE}_{\partial \Lambda}(\psi) \,,
\end{align*}
where we have used Lemma \ref{lemma:wellchosenpath} to see that 

\[ \frac{1}{k^{d-1}}  \sum_{\substack{ x,z  \in \partial \Lambda \\  |x - z | = k }}  \sum_{e \in \gamma_{x z} } (\nabla_e \psi )^2 \le C \text{ DE}_{\partial \Lambda} (\psi)~. \]

In particular, if $A = n $, we see that 
\begin{equation}\label{eq:whenAisn} \sum_{1 \le k  \le A  } \frac{1}{k^d}  \sum_{\substack{x, z \in \partial \Lambda \\  |x - z | = k }} \left|\sum_{e \in \gamma_{x z} } \nabla_e \psi \right|^2  \le C  n  \text{ DE}_{\partial \Lambda}(\psi)  \end{equation}

For $k \ge A$, we simply bound $|\psi(x) - \psi(z)|  $ by $2\|\psi\|_\infty$ to obtain, 

\[ \sum_{|k| > A } \sum_{\substack{x , z \in \partial \Lambda, \\  d(x,z) = k }} \frac{1}{k^d} (\psi(x) - \psi(z ))^2 \le 2 n^{d-1} \|\psi\|_\infty^2 \sum_{|k| \ge A }  \frac{k^{d -2}}{k^d}  \le   \frac{2 C n^{d-1} \|\psi\|_\infty^2}{A} \]
(N.B. Here we used that $\Lambda=\Lambda_n$ and therefore there are $\Omega(1) n^{d-1}$ points $x$ on $\p \Lambda_n$).
We thus see that 
\[ \text{DE}(\psi) \le \text{DE}_{\partial \Lambda} (\psi) + C \, A \,\,  \text{DE}_{\partial \Lambda}(\psi) + \frac{n^{d-1}}{2} C  \frac{\|\psi\|_\infty^2}{A} \] for some $C >0$. Optimizing on $A$ by setting $A = \frac{\|\psi\|_\infty n^{(d-1)/2}}{\sqrt{\text{DE}_{\partial \Lambda} (\psi)}}$ when $A \le n $, we obtain 

\[ \text{DE}(\psi) \le \text{DE}_{\partial \Lambda}(\psi) + C n^{(d-1)/2} \sqrt{ \text{DE}_{\partial \Lambda}(\psi) } \|\psi\|_\infty \]

The proof is complete as, by Cauchy--Schwarz inequality, one has when $\Lambda=\Lambda_n$ 
\begin{equation}\label{eq:upper_bound_de}
\text{DE}_{\partial \Lambda}(\psi) \le n^{(d-1)/2}  \|\psi\|_\infty  \sqrt{\text{DE}_{\partial \Lambda}(\psi)}\,.
\end{equation}

\end{proof}

\begin{remark} 
Notice that if we do not optimise on the cut-off value $1 \leq A \leq n$ in the above proof, it gives the following less optimal estimate which is also classical in harmonic analysis and harmonic maps: namely, there exists a constant $C>0$ s.t. for all $n\geq 1$, 
\begin{align}\label{e.LessOpt}
\text{DE}_{\Lambda_n}(\psi) \le C n \text{DE}_{\p \Lambda_n}(\psi)\,.
\end{align}
\end{remark}


\subsection{Step4. Randomized projection to $\S^1$ with few vortices.}

This step is strongly inspired by the analysis of harmonic maps from \cite{hardt1988stable}.


\begin{lemma}\label{lemma:projection_on_sphere_trhough_a_good_point} 
Let $D$ be any finite subset of a graph $G$. Let $\psi: D \to \R^{N}$ with $\psi(x) \in \mbb{S}^{N-1}  $  for any $x \in \partial D$ and $\|\psi(x)\| \le 1$ for any $x$ in $D$. Then, there exists $\sigma: D  \to \mbb{S}^{N-1} $ that coincides with $\psi$ on $\partial \Lambda$ and which satisfies 
 \[ \mrm{D E }(\sigma) \le \begin{cases} C_N \, \mrm{D E }(\psi) & \text{if } N\ge 3 \\ C_2 \mrm{D E }(\psi) \log(|D|) & \text{if } N= 2 \end{cases} \] for some constant $C_N > 0 $ that depends only on $N$ and where $|D|$ denotes the cardinal of $D$.
 \end{lemma}


\ni
{\em Proof of Lemma \ref{lemma:projection_on_sphere_trhough_a_good_point}.} 

 The proof relies on the following nice idea of Hardt, Kinderlehrer and Lin.

 A naive way to construct a function $\sigma: D \to \mbb{S}^{N-1} $ that coincides with $\psi$, is to set $\sigma(x) = \frac{\psi(x)}{|\psi(x)|}$. Of course, this is not necessarily well-defined and even when it is, it has no reason to have a Dirichlet energy that is under control. The choice of projection "from the point of view" of zero is somewhat arbitrary and one may wish to define a function $ \sigma_y (x) = \frac{\psi(x) - y}{|\psi(x) - y|} $ for some well-chosen $y$. The function $\sigma_y$ does not coincide with $\psi$ on  $\partial D$ with this definition but this can be circumvented as follows. For $y \in B(0,1)$, the function $f_y (x) = \frac{x - y }{|x -y |}$ is  bijective, bilipschitz on $\mbb{S}^{N-1}$ and the Lipschitz constant of $f_y$ and its inverse are bounded uniformly for $y \in B(0, 1/2)$.
 A suitable option should thus be to set
 $ \sigma(x) = f_y^{-1} (\sigma_y (x)) $
 for a well-chosen $y$.

What should then be the appropriate choice of $y$? The idea of Hardt, Kinderlehrer and Lin is very nice and simple: pick $y$ uniformly at random in $B(0,1/2)$ and hope it works with reasonable probability. Let $Y $ be a uniform random variable in $B(0,1/2)$. This strategy is indeed successful if $N \ge 3$ as we will see by computing the expectation of the  Dirichlet energy of $\sigma_Y$. (N.B. notice that since $f_Y$ is uniformly bilipschitz over $Y\in B(0,1/2)$, the Dirichlet energy of $x\mapsto f_Y^{-1}(\sigma_Y(x))$ is bounded from above by a constant times the Dirichlet energy of $x\mapsto \sigma_Y(x)$).

%
%
\ni
Now, as
  \begin{align*}
&|\sigma_Y (j) - \sigma_Y (j)| 
 =  \left| \frac{(\psi(i) - Y)  (|\psi(j) - Y |) - (\psi(j) - Y) (|\psi(i) - Y |)}{|\psi(j) - Y||\psi(i) - Y |} \right|  \\ 
 &= \frac{1}{2} \left| \left(  (\psi(i) - \psi(j))\frac{|\psi(i) -Y | + |\psi(j) - Y |}{|\psi(i) - Y ||\psi(j) - Y |}  + (|\psi(j) - Y | - |\psi(i) - Y | ) \frac{( \psi(j)  - Y + \psi(i)  - Y )}{|\psi(j) - Y||\psi(i) - Y |}\right)\right|  \\ 
 &\le  |\psi(i) - \psi(j)|  \left( \frac{1}{|\psi(i) - Y|} + \frac{1}{|\psi(j) - Y|}\right) ~.  
\end{align*}
In the last step of this computation, we have used that $|\psi(j) - Y | - |\psi(i) - Y |\le |\psi(j) - \psi(i)| $.


 We thus obtain 
 \[ \E(|\nabla_{i j } \sigma_Y |^2) \le |\nabla_{i j } \psi|^2  \frac{2}{\text{Vol}(B(0,1/2))} \int_{B(0,1/2)} \frac{ d r}{r^2}  \le C |\nabla_{i j } \psi|^2   \] if $N \ge 3$.

 Summing over edges $i j $, we see that 

 \[ \E(\text{DE}(\sigma_Y)) \le C \text{DE}(\psi) \] for some constant $C$ that depends only on $N$.

 We have crucially used the fact that $ \frac{1 }{r^2}$ is integrable near $0$ which only holds when $N \ge 3$. If $N = 2$, we need to slightly modify the distribution of $Y$ in order to solve the problem. Let $\tilde{Y}$ be a random variable which has the distribution of a  uniform random variable on  $B(0,1)$ conditioned on not lying in any of the $|D|$ balls of radius $\frac{1}{10 |D|}$ centered at the $\psi_i$, $i \in D$.

We obtain,
  \[ \E(|\nabla_{i j } \sigma_{\tilde{Y}} |^2) \le |\nabla_{i j } \psi|^2  \frac{2}{C_\psi} \int_{A_{(\frac{1}{10 D }, 1/2) }} \frac{ d r}{r^2} \]
  
where $C_\psi = \text{Vol}\left(B(0,1/2) \setminus \bigcup_{i \in D} B(\psi_i, \frac{1}{10 |D|}) \right)  \ge C $ for some $C>0$ and $A_{(\frac{1}{10 D }, 1/2) } $ is the annulus centered at $0$ with inner radius $\frac{1}{10 |D|}$ and outer radius $1/2$.

As (when $N=2$), \[ \int_{A_{(\frac{1}{10 D }, 1/2) }} \frac{ d r}{r^2} \le C \log(|D|) ~, \] we see that  $ \E(|\nabla_{i j } \sigma_{\tilde{Y}} |^2)  \le C (\nabla_{i j }(\psi ))^2 \log(|D|)$ for some other constant $C$ and 

\[ \E(\text{DE}(\sigma_{\tilde{Y}} (x))) \le C \text{DE}(\psi) \log(|D|) ~. \]
\qed

\subsection{Step 5. Conclusion}
$ $

Let $R > 10$.
Let $\sigma$ be some ground state configuration for $H_{\lambda}$, Let $k\in [\frac{9R} {10}, R]$ be as in step $1$. As $\sigma$ is a ground state configuration, the restriction of  $\sigma$ to $\Lambda_k$ is also a ground state configuration. It is thus sufficient to construct a configuration $\tilde \sigma$ in $\Lambda_k$ that coincides with $\sigma$ in $\partial \Lambda_k$  and that has $H^{\Lambda_{k}}(\tilde \sigma) \le C R^{d - \frac 3 2 } \log(R)$.  By step $2$, we know that we can extend $\sigma$ on $ \p \Lambda_{k}$ to some $\S^1$ valued $\sigma'$ on $ \partial \Lambda_{k-1}$ with $H_{\lambda}^{\p \Lambda_k \cup \p \Lambda_{k-1}}(\sigma') < \frac{C}{(1+\lambda) \wedge ( 1 - \lambda)}  R^{d-2}$.
By step 3, if we extend $\sigma'$ to a harmonic function in $\Lambda_{k-1}^{int}$, we obtain $\sigma''$ with $H_{\lambda}({\sigma''}_{|\Lambda_{k-1}}) \le \frac{C}{((1+\lambda) \wedge ( 1 - \lambda))^{\frac 12 }}  R^{(d-1)/2} R^{\frac{d-2}{2}} =\frac{C}{((1+\lambda) \wedge ( 1 - \lambda))^{\frac 12 }}  R^{d - \frac 3 2 } $.
 Projecting back to $\S^1$ using step 4 gives some $\tilde \sigma$ with $H_{\lambda}(\tilde \sigma) \le  \frac{C}{(1+\lambda)^{1/2} \wedge ( 1 - \lambda)^{1/2}}  R^{d - \frac 3 2 } \log(R) $.
N.B. at this stage, note that  the constant $C$ only depends on $N$ and the choice of $\delta$, but $\delta$ can be chosen to be $\tfrac 1 2$ for this energy estimate.

\begin{remark}\label{rmk:bootstrap}
 When $d = 2$ 
this argument maybe bootstrapped. Indeed, once we know that we have a box with energy at most  $\frac{C}{(1+\lambda)^{\frac 1 2}  \wedge ( 1 - \lambda)^{\frac 1 2}}  R^{1/2} \log(R)$, we can thus find a layer (using step 1) with Hamiltonian at most $\frac{10}{9} C\frac{R^{-1/2} \log(R)}{((1+\lambda)^{1/2} \wedge ( 1 - \lambda)^{1/2})^2}$, which gives us -- using steps 2 to 4 -- a configuration $\sigma''$ in a box of size $81/100 R$ with $H_{\lambda}(\sigma'') \le \frac{\frac{10}{9} C R^{1/4} \log(R)^2 }{(1+\lambda)^{\frac 1 2 } \wedge ( 1 - \lambda)^{\frac 1 2 }} $. 
  Bootstrapping $k$ times, we get a configuration with energy at most 
  $H_{\lambda}(\sigma'') \le \frac{(\frac{10}{9}C)^k R^{1/2^k} \log(R)^k }{\big( (1+\lambda) \wedge ( 1 - \lambda) \big)^{k/2}}$
\margin{\bf C: it is $k/2$ the exponent of $(1+\lambda)\wedge ...$ or instead $2^{-k}$  \green{ N: everytime we pay the constant, it is $k/2$} Ok, great thanks!}  
in a box of size $(\frac{9}{10})^k R$. Optimizing in $k$ leads to setting $k \approx \log \log(R)$ which gives us a subpolynomial estimate for $H_{\lambda}(\sigma)$: 
  $H_{\lambda}(\sigma)  \approx \exp( \log(\log(R))^2)$ in a box of size $\log(R)$. 

\end{remark}

\section{\for{toc}{Turning on the temperature}\except{toc}{Turning on the temperature, the polar region remains sparse}}\label{s.TurnOn}


Proposition \ref{pr.GS_more_general} allows us to control the Hamiltonian in the box of size $R/2$ of a typical configuration of the model in $\Lambda_R$ at very low temperature. For the sake of simplicity, we state this extension only when $d =2$ but it applies similarly when $d \ge 3$.
\begin{lemma}\label{lemma: energy}\label{lemma:uniform_decay_to_0_of_high_energy}
For $d =2$ and any $\lambda\in(-1,1)$, there exists a constant $C=C(\lambda)$ s.t. for any integer $R$,
\begin{align*} 
\sup_{\tau: \partial \Lambda_R \to \S^2}  
\FK{\beta, \lambda, \Lambda_R, \tau}{}{H^{\Lambda_{R/2}}(\sigma)  \ge  C R^{3/4}}\underset{\beta \to \infty}\longrightarrow 0 
\end{align*}
(notice that the  convergence is uniform in the boundary conditions $\tau: \partial \Lambda_R \to \S^2$).
\end{lemma}
\label{lemma: long_polar_lines_are_sparse_with_finite_beta}

\begin{proof}

Let us fix $R$ and $\tau$ and denote $\P_{\tau}  = \P_{\beta, \lambda, \Lambda_R,  \tau}$. 

First, notice that $\P_{\tau}(H(\sigma) \le 10 R )$ goes to $1$ as $\beta\to \infty$ uniformly in $\tau$.  Indeed, for a fixed $\tau$, let $\sigma^0$ be some ground state configuration on $\Lambda_R$ that coincides with $\tau$ on $\partial \Lambda_R$. Since $H(\sigma^0) \le 8 R $ as $f_{\eps}(u,v) \le 2$  (extend $\tau$ by any constant configuration along $\S^1$)  and since there are at most $4 R^2 $ edges in $\Lambda_R$, any configuration $\sigma$ such that for all sites $i \in \Lambda_R$, $ |\sigma_i - \sigma^0_i|^2 \le \frac{1}{ 4 R} $ satisfies $H_{\lambda}(\sigma) \le 9 R$. Now, the set of such $\sigma$ with $H(\sigma) \le 9 R $ thus has a positive measure that is bounded from below by $ C(R) = (\frac{1}{ 4 \pi R })^{4R^2}$.
We thus get a lower bound on the partition function namely $Z_{\tau } \ge C(R) \exp( - \beta  9 R )$.
Thus,
\begin{align*}
\P(H_{\lambda}(\sigma) \ge 10 R ) &= \frac{1}{Z_{\tau}} \int_{\sigma , H_\lambda(\sigma ) \ge 10 R  } \exp( - \beta H_\lambda(\sigma) ) d \sigma \\
&   \le \frac{\exp(- 10 \beta R)}{Z_{\tau} } \\
&\le \frac{\exp(- \beta R )}{C(R)}
\end{align*}

Now by the same argument as in the Fubini step of the previous section \ref{s.Fubinistep}, we see that conditionally on $H_{\lambda}(\sigma) < 10 R$, we can find some (random) integer $k$ between $9R/10$ and $R$ such that $H^{\partial \Lambda_{k} }(\sigma) \le 100 $ with probability that goes to $1$ uniformly in $\tau$ when $\beta \to \infty$. In order to define a specific random variable, let us say that $k$ is the largest random integer with this property and let $\tau'$ denote $\sigma$ restricted to $\partial \Lambda_k$.
As in the proof of  Proposition \ref{pr.GS_more_general}, we know that a ground state configuration $\sigma_0$  that coincides with $\tau'$ on $\partial \Lambda_k$ will have $H^{\Lambda_k}_\lambda(\sigma_0) \le C R^{3/4}$.
By the exact same reasoning as above we see that uniformly in $\tau'$ with $H^{ \p \Lambda_k} (\tau' ) \le 100$ and $k\in [\frac {9R} {10}, R ]$, 
$\P(H^{\Lambda_k} (\sigma) \le 2 C R^{3/4} | k , \tau' ) \to_{\beta \to \infty} 1$

As we can write 
\[
\P(H^{\Lambda^{R/2} }(\sigma) \le 2 C R^{3/4} ) \le \E(\P(H^{\Lambda^{k} }(\sigma) \le 2 C R^{3/4} | k , \tau' )) + \P(k \text{ is not defined}) 
\]
the proof of Lemma \ref{lemma:projection_on_sphere_trhough_a_good_point} is complete.

\end{proof}

\section{\for{toc}{Proof of Theorem \ref{th.main}}\except{toc}{Proof of Theorem \ref{th.main}}}\label{sec:proof_of_main_thm}

The proof of the main theorem relies on reducing the interaction of the $O(3)$ spins to an effective $XY$ model on a random subgraph, and then applying a renormalization argument to handle the resulting quenched disorder. This approach allows us to bypass the fact that localized "polar" defects will inevitably appear at any finite temperature, by showing that they do not destroy the macroscopic connectivity of the equator-aligned spins.

\subsection{Reduction to an $XY$ model on a random graph}

We begin by exhibiting an $XY$ model on a random graph that has a smaller two point function than the anisotropic model. Let $\delta > 0$. For any configuration $\sigma: \Lambda \to \S^2$, we define the "equatorial" graph $G_\delta(\sigma) = G_\delta = (V_\delta, E_\delta)$ where the vertices $V_\delta$ are the sites $i \in \Lambda$ such that $|\sigma_i^z| < \sqrt{1 - \delta^2}$, and the edges $E_\delta$ are the nearest-neighbor edges of $\Lambda$ with both endpoints that belong to $V_\delta$.

We can decompose each spin as $\sigma_i = (r_i \tau_i^1, r_i \tau_i^2, \sigma_i^{3,N})$, where $\tau_i  = (\tau_i^1, \tau_i^2) \in \S^1$ and $r_i = \sqrt{1 - |\sigma_i^{3,N}|^2}$. Notice that, conditionally on the $z$-coordinates $(\sigma_i^{3,N})_{i \in \Lambda}$, the transverse components $\tau_i$ is exactly as an $XY$ model on $\Lambda$ with random coupling constants $J_{ij} = \beta r_i r_j$. 

As was observed in many works such as \cite{patrascioiu1992phase, patrascioiu1993percolation,dunlop1985correlation, aru2025percolation,dario2025phase},  Ginibre correlation inequality implies that the two-point function of this system is monotone in the coupling constants. If we restrict the interactions strictly to the graph $G_\delta$ and lower-bound the coupling constants on these edges by $\beta \delta^2$ (since $r_i > \delta$ for all $i \in V_\delta$), we obtain the following lower bound:
\begin{equation}\label{eq:ginibre_reduction}
\langle \sigma_x^{1,2} \cdot \sigma_y^{1,2} \rangle_{\beta, \varepsilon} 
= \mathbb{E}_{\Lambda, \beta, \varepsilon} \left[ r_x r_y \langle \tau_x \cdot \tau_y \rangle_{J_{ij}} \right] 
\ge \delta^2 \mathbb{E}_{\Lambda, \beta, \varepsilon} \left[ \langle \tau_x \cdot \tau_y \rangle^{XY}_{\beta \delta^2, G_\delta(\sigma)} \right],
\end{equation}
where $\langle \cdot \rangle^{XY}_{\beta \delta^2, G_\delta(\sigma)}$ denotes the thermal expectation of a standard $XY$ model at inverse temperature $\beta \delta^2$ defined strictly on the random graph $G_\delta(\sigma)$. Note that this steps uses the assumption that $\<{\cdot}_{\beta, \eps} $ is the limit of finite volume measures that put either free (or Dirichlet) boundary conditions in the $xy$ plane. 

\subsection{Stochastic domination of the good boxes}\label{ss.LSS}

To show that the $XY$ model on $G_\delta$ exhibits a BKT phase, we need to prove that $G_\delta$ possesses a robust supercritical geometry. It is very natural to check whether we can directly show that $G_{\delta}$ stochastically dominates a Bernoulli percolation of parameter $p > p_c$ and then directly apply the result of Dario and the first author \cite{dario2025phase} (see also the work \cite{dario2026impact} in the case of the Villain model). This is indeed the case if $|\lambda|$ is small enough as can be seen by applying the standard criterion of Liggett, Schonmann and Stacey \cite{liggett_domination_1997}. However, when $\lambda$ gets close to 1 (or $-1$) we are not able to prove this.
 We nonetheless recover enough supercritical behavior of $G_{\delta}$ to ensure that the XY model has -- with high probability-- polynomial decay of correlations on $G_{\delta}$ provided that $\beta$ is large enough. To this effect, we use the following geometric renormalization scheme developed in \cite{dario2025phase}.

\begin{definition}[Good boxes]
Let $R \ge 1$ be an integer scale. We partition $\Z^2$ into boxes $\Lambda_R$ of side-length $R$. A box $\Lambda_R$ is called \emph{pre-good} in $G_\delta$ if $G_\delta \cap \Lambda_R$ contains a unique macroscopic crossing cluster, and all other connected components have a diameter strictly smaller than $R/100$. A box is called \emph{good} if it is pre-good, and if all overlapping boxes of side-lengths between $R/10$ and $R/2$ are also pre-good. 
\end{definition}

\begin{lemma}\label{lem:good_box_prob}
For any $-1<\lambda<1$ and any target probability $p < 1$, there exists $\delta > 0$, a scale $R$, and an inverse temperature $\beta_0$ such that for all $\beta \ge \beta_0$, 
\[
\mathbb{P}_{\beta, \varepsilon} \left( \Lambda_R \text{ is good} \;\middle|\; \sigma_{\Z^2 \setminus \Lambda_{3R}} \right) \ge p,
\]
uniformly in the conditioning $\sigma_{\Z^2 \setminus \Lambda_{3R}}\in (\S^{N-1})^{\Z^2 \setminus \Lambda_{3R}}$.
\end{lemma}

\begin{proof}
Notice that the event that $\Lambda_R$ is good depends only on the geometric structure of $G_\delta$ within $\Lambda_{3R/2}$. 

For a box $\Lambda_R$ to fail to be good, $G_\delta \cap \Lambda_{3R/2}$ must exhibit a topological obstruction. This implies the existence of a continuous path of "polar" sites (where $|\sigma_i^{3,N}| \ge \sqrt{1 - \delta^2}$) of length at least proportional to $R$. Because the energy functional $H_{\lambda}(\sigma)$ penalizes the $z$-components of the spins (as shown in Section \ref{sec:polar_lines}), every site in this polar path contributes an energy cost of order $(1+\lambda)\wedge(1-\lambda) \, \delta^2$. 

Thus, any configuration $\sigma$ violating the goodness of $\Lambda_R$ must satisfy $H^{\Lambda_{3R/2}}(\sigma) \ge C (1+\lambda)\wedge(1-\lambda) \delta^2  R$ for some universal constant $C > 0$. By Lemma \ref{lemma:uniform_decay_to_0_of_high_energy} the finite-volume probability of observing such a macroscopic energy fluctuation goes to zero as $\beta \to \infty$, uniformly in the boundary condition on $\partial \Lambda_{3R}$. 

Therefore, by taking $R$ appropriately and then $\beta$ sufficiently large, the conditional probability that $\Lambda_R$ is good, given any arbitrary configuration outside $\Lambda_{3R}$, can be made strictly greater than $p$.
\end{proof}

We now have a random field of indicator variables $\mathbf{1}_{\{\Lambda_R \text{ is good}\}}$ defined on the macroscopic grid $R\Z^2$. By Lemma \ref{lem:good_box_prob}, the conditional probability of a box being good, given the entire configuration of boxes outside its immediate geometric neighborhood, is uniformly bounded from below by $p \to 1$ as $\beta \to \infty$. 

This uniform conditional lower bound is exactly the hypothesis required by the classic stochastic domination theorem of Liggett, Schonmann, and Stacey \cite[Theorem 0.0]{liggett_domination_1997}. Consequently, for any percolation parameter $p_{perc} < 1$, there exists $\beta$ large enough such that the random field of good boxes stochastically dominates an i.i.d. Bernoulli site percolation of parameter $p_{perc}$ on the macroscopic grid.

\subsection{Renormalization and the BKT phase}

We are now exactly in the setting analyzed in \cite{dario2025phase}. We have an $XY$ model on a random graph $G_\delta$ which, on a macroscopic scale $R$, stochastically dominates a highly supercritical Bernoulli percolation of good geometric boxes. 

By applying the specific path modifications developed in \cite[Section 5.2]{dario2025phase} —which sequentially use the Ginibre inequality to extend paths, duplicate intersecting vertices, and trace correlations through the macroscopic cluster—the two-point function of the $XY$ model on $G_\delta$ can be bounded from below by the two-point function of an $XY$ model on a deterministic extended lattice $\Z^2_{R'}$. 

Specifically, following the exact same reasoning of Dario and the first author \cite{dario2025phase}, if $p_{perc}$ is chosen large enough, the stochastic domination yields:
\begin{equation}
\mathbb{E}_{\beta,\varepsilon} \left[ \langle \tau_{x} \cdot \tau_{y} \rangle^{XY}_{\beta \delta^2, G_\delta} \right] 
\ge 
\mathbb{E}_{p_{perc}} \left[ \langle \cos(\theta_x - \theta_y) \rangle_{\mu^{XY, \Z^2_{R'}}_{\beta', \beta', r}} \right],
\end{equation}
where $\mathbb{Z}^2_{R'}$ is the extended graph obtained by replacing edges with $1d$ lines made of  $R'$ vertices (where $R' = 2 (2R+1)^2$ and $R$ as in the previous paragraph)  and $\beta' = \frac{1}{8} \beta \delta^2$ is the renormalized effective inverse temperature resulting from the Ginibre path-splitting operations.

By \cite[Proposition A.1]{dario2025phase}, the $XY$ model on the extended lattice $\Z^2_{R'}$ undergoes a BKT phase transition, meaning that for sufficiently large $\beta'$, its two-point function decays polynomially. Their result applies for boundary conditions for which the Ginibre inequality holds (in particular free and Dirichlet).
\margin{C: same pbm here, maybe we have a small issue with initial $\S^2$-free boundary conditions ?N : No issue and no need for MMSP actually}

 Since $\beta' \propto \beta \delta^2$, we can make $\beta'$ arbitrarily large by taking our initial $\beta$ to infinity.  
Combining this with \eqref{eq:ginibre_reduction}, we conclude that for $\beta$ large enough, the transverse components of our original model exhibit a power-law lower bound:
\[
\langle \sigma_i^1 \sigma_j^1 + \sigma_i^2 \sigma_j^2 \rangle_{\beta, \varepsilon} \ge \frac{c}{\|x - y\|^{\alpha(\beta)}},
\]
for some $\alpha(\beta) > 0$. This completes the proof of Theorem \ref{th.main}.

\begin{remark}\label{rmk:upper_hemisphere}
  Theorem \ref{th.main} applies to the hemisphere model with the same proof. This is in fact a stronger result, see Proposition \ref{prop:half-sphere-is-dominated-by-full-sphere}
\end{remark}

\subsection{Quantitative analysis of the proof}\label{sec:thmmain_quant}

The aim of this short paragraph is to turn our proof in an explicit lower bound on the transition point $\beta_{BKT}$. The claim is the following.


\begin{prop}[Quantitative version of Theorem \ref{th.main}] \label{prop:thmmain_quant}
  With the notations of Theorem \ref{th.main}, for all $\lambda$ with $|\lambda|<1$ close enough to 1,
 \begin{align*}  
 \bar{\beta}\le C \frac{\Big(\log(\frac 1 {\kappa(\lambda)})\Big)^4}{\kappa(\lambda)^{6}}
\end{align*}
for some $C$ that depends only on $N$ and where $\kappa(\lambda):= (1+\lambda)\wedge(1-\lambda)$. 
\end{prop}

This proposition also holds for the hemisphere model $\S_2^+$. Thus combining Theorem \ref{thm:exponential_decay} implies the following bounds -- when $\lambda \geq 0$ -- on the transition point $\beta_{\lambda, BKT}^+$ for the hemisphere model,

\begin{align*} 
\log(\frac 1 {1-\lambda} ) \ll \beta_{\lambda, BKT}^+ \ll  \frac{ \Big(\log(\frac 1 {1-\lambda}) \Big)^4}{(1-\lambda)^{6}}\,\,\,
\text{   when $\lambda\in[0,1)$.}
\end{align*}

\begin{proof}
We follow the proof step by step and indicate how to make each of them explicit. To avoid getting lost in technicalities, we assume $\lambda\in (0,1)$  and close to 1  but the argument can be made quantitative for $\lambda$ close to $-1$ as well. Proposition \ref{pr.GS_more_general_on_sites} give that $C(\lambda, \delta ) \ll \frac{C(\delta)}{(1-\lambda)^{3/2}}$. One could also be willing to optimize over $\delta$ but it does not lead to a qualitative improvement of the bound. We may thus fix $\delta$ to be $\tfrac 1 2$ for example. Thus, in order to have a ground state with Hamiltonian at most $c R$ for some small $c$ in $\Lambda^{R/2}$ (which ensures that the box is good), we need to chose $R$ such that $\frac{1}{(1-\lambda)^{3/2}} R^{1/2} \log(R) \le c R $, which is to say $R \gg \frac{\log^2(\frac 1 {1-\lambda})}{(1-\lambda)^3}$. (N.B. Using the bootstrap argument of Remark \ref{rmk:bootstrap} does not seem to lead to any significant improvement here).

There are two items in the proof of Theorem \ref{th.main} that depend on the value of $R$. First, we need the density $p_{perc}$ of the Bernoulli percolation that good boxes dominate to be large enough to ensure that there is a BKT transition for the XY model on percolation on the graph $\Z_{c R^2}^2$. Let $p$ denote the infimum of the probabilities that a box is good conditionally on what happens outside of boxes at distance $3 R$. If $p$ is large enough, then Theorem 1.3 of \cite{liggett_domination_1997}, guarantees that one can take $p_{perc} \ge  1 - C ( 1- p)^{1/100}$. 
\margin{C: il me semble que les roles de $p$ et $p'$ sont inversés ici non ? Dans le Lemme \ref{lem:good_box_prob}, on cherche plutot $p$ local que $p'$ ?, \green{N: Je pense que c est mieux mtnt} }

In order to have the BKT transition on $\Z_{c R^2}^2$, it is enough that  $p_{perc} \ge 1 - \exp( - c \beta_{\text{BKT},n})$ where $\beta_{\text{BKT},n }$ is the inverse temperature of the phase transition for the XY model on $Z_{n}^2$ and $n = c R^2$ according to \cite{dario2025phase}. By the proof of Proposition A.1 in \cite{dario2025phase}, this happens as soon as $\frac{I_0( C \beta )}{I_1(C \beta)} \le \exp( \frac{1}{4 (n-1)})$ where $I_k(\beta ) = \int_{\theta \in \S^1 } \exp(\beta \cos(\theta)) \cos(k \theta) d \theta$ is the modified Bessel function. When $\beta $ goes to $1$, $ \frac{I_0( C \beta )}{I_1(C \beta)} = 1 - \frac{c'}{\beta} + o(\beta)$, it is thus sufficient to take $\beta \ge c n \propto  R^2$. Thus $\beta_{\text{BKT},n} \le c R^2$. 
\margin{C: autre facon de voir les choses (double-check). Sur $M$ aretes, si on fait un $XY$ de temperature inverse $B*M$, ca converge vers un villain de temperature $B$. Donc de ce point de vue $R$ tres grand, le scaling $\beta \geq \text{Gde Constante} R^2$ est logique. \textbf{Je me suis arrete, la pour ce calcul, on pourra refaire ensemble au tableau ? (et je rajouterai qqs details $:)$ !) }}

By the proof of  Lemma \ref{lemma: long_polar_lines_are_sparse_with_finite_beta}, we see that in order for $H^{\Lambda_{R/2}}(\sigma)$ to be smaller than $c R$ for some small $c$ with probability $p$, we need $\exp( - \beta R + 4 \log(R) R^2) \le p $, that is $\exp( - \beta R + \log(R) R^2/2 + C R^2) \le 1$ which is the case if $\beta \gg \log(R)R$. We also need $\beta$ to be large enough so that the XY model on the percolation of parameter $p$ on $\Z_n^2$, at inverse temprature $c \beta$ has polynomial decay of correlations. This happens as soon as $c \beta \ge 2 \beta_{\text{BKT},n}$ which requires $\beta \gg R^2$. To summarize we obtain polynomial decay of correlations in the $\lambda$-anisotropic model when, $\beta \gg R^2 \gg \frac{ \log^4(1 - \lambda)}{( 1 - \lambda)^6}$.
\end{proof}


\newsection{Extension to the other smooth deformations}\label{s.OtherD}

In this Section, we shortly outline how to adapt the proof of Theorem \ref{th.main} to the case of $\S^2_+$ (Theorem \ref{th.mainS+}) as well as to the case of other smooth deformations (Theorem \ref{th.allS}).

\bnum
\item First, for the case of $\S^2_+$, i.e. Theorem \ref{th.mainS+}, the proof follows exactly the same steps, line to line. (The analysis of the groundstates is the same and the probabilistic analysis which relies on the behaviour is also identical except integrals are over $\S^2_+$ instead of $\S^2$). 

\item We then turn to item (4) from Definition \ref{d.modif2}. In this perturbation, spins are in $\S^2$ and interact via  $SO(3)$ invariant Dirichlet energy but a small potential is added at each site $i\in \Lambda \subset \Z^2$:  $V(\sigma_i) = \eps (\sigma_i^3)^2$. We may also run almost the same argument here: first uniformly over all boundary conditions around $\p \Lambda_n$, the energy inside $\Lambda_n$ is still upper bounded by $O(n)$. The same Fubini step gives us a scale $\tfrac n 2  \leq k \leq n$ for which the Dirichlet energy along $\p \Lambda_k$ is of order $O(1)$. 
Because in this item (4) variant, spins still do not like to be close neither to the north pole nor to the south pole (because of the potential energy term $\eps (\sigma_i^3)^2$), only an order 1 amount of spins are far from the equator. Exactly as we had in our proof of Theorem \ref{th.main}. This allows us to use the same projection on $\S^1$ trick as did for Theorem \ref{th.main}. The rest of the proof is the same modulo straighforward adaptations in the probabilistic argument which compares positive temperature Gibbs measure with the behaviour of groundstates. 

\item Finally, let us consider the case of item (3) from Definition \ref{d.modif2}. (The remaining cases being adaptations of these 3 cases). In this item (3), spins now belong the the filled $\eps$-ellipsoid $\mathbb{B}^3_\eps$ and interact according to the isotropic scalar product $\<{\cdot}_1$. 

As we did in the proof of Theorem \ref{th.main}, it is natural to consider in this case the energy functional 
\begin{align*}\label{}
H(\sigma):= \sum_{i\sim j} \Big( 1 - \<{\sigma_i, \sigma_j}_1  \Big)\,.
\end{align*}
The constant 1 is there to make sure that a constant spin field pointing in the direction of a point along the equator would have energy 0. While a constant field pointing in, say, the north direction would have an energy $|\Lambda|(1- (1-\eps)) =  |\Lambda| \eps$.  This also shows that spins need to stay away from north and south poles as much as possible. 

We may thus follow the same analysis as for Theorem \ref{th.main} or Item (4) above with one distinction: spins here are no longer on $\S^2$, but instead in the volume $\B_\eps^3$. Because of this, we need to adapt two things: (a), first to make sure that we do not loose too much energy when projecting spins to $\S^1$ and second (b) condition the spin field not only on $\{ \sigma_i^3\}_{i\in \Lambda}$ but also on $r_i:= \sqrt{(\sigma_i^1)^2 + (\sigma_i^2)^2}$ which is no longer measurable with respect to $\sigma_i^3$.  

For part (a), this can be seen thanks to the Fubini step which is still valid and gives us a scale $k$ so that the Energy  (as redefined above) along $\p \Lambda_k$ is of order 1. By the definition of the above energy $H(\sigma)$, this means that only an order 1 of sites may carry a spin far from the $\S^1$ equator circling along the volume $\B_\eps^3$. Thanks to this, one can still project to $\S^1$ from $\p \Lambda_k$ to $\p \Lambda_{k-1}$ to the cost of an order 1 energy loss and proceed for the rest of the proof as we did in Theorem \ref{th.main}. For part (b), once we know that there is a groundstate in the bulk of $\Lambda_R$ whose energy is upper bounded by $R^{1/2} \log(R)$, we deduce that few points in $\B_\eps^3$ are far from the equator circle $\S^1$, this allows us to condition both on $\sigma_i^3$ and $r_i$ so that it still leaves a supercritical percolation cluster of large conductances for the remaining conditional $XY$ model. 
\enum

\section{A correlation inequality proof of Theorem \ref{th.main}  in $\S^3$}\label{s.CIP}

Kunz, Pfister and Vuillermot \cite{kunz_inequalities_1976} established a correlation inequality that allows to lower bound the two point function of the anisotropic XY model to the two point function of the Ising model. Interestingly, their proof generalizes to the anisotropic spin $O(4)$ model with Hamiltonian $H(\sigma) = -  \sum_{i \sim j } \sigma_i^{1,2} \cdot \sigma_j^{1,2} + \lambda \sigma_i^{3,4} \cdot \sigma_j^{3,4}$. To our knowledge, this had not been noticed before.  The correlation inequality is the following.
\begin{prop} 
For any $|\lambda| < 1$,
\[
\langle \sigma_i^{1,2} \cdot \sigma_j^{1,2} \rangle_{\lambda , \beta} \ge \frac{1}{4} \langle \sigma_i \cdot \sigma_j \rangle^{XY}_{\frac{1 - |\lambda|}{4} \beta}
\]
\end{prop}
This proposition implies Theorem \ref{th.main} in the special $N = 4$ case. Notice that in this case, it gives a better bound on the phase transition temperature than what we obtain with our technique in $\S^2$ or $\S^2_+$. Indeed, it shows that \[\beta_{BKT}(\lambda) \le \frac{ 4 \beta_{BKT}(\lambda = 0)}{1 - |\lambda|}~.\]

This suggests, although it does not quite prove, that if one considers the model with isotropic Hamiltonian $H(\sigma) = - \sum_{i \sim j } \sigma_i \cdot \sigma_j $ but anisotropic a priori measure $\mu(\sigma_i) = \exp( - \alpha \|\sigma_i^{3,4}\|^2) d \sigma_i$ should have a BKT phase transition for all $\alpha$ large enough. The physicists' prediction $\beta_{BKT}(\lambda)  \approx  - \frac{1}{\log(1 - \lambda)}$  hints on the fact that this should in fact hold for any $\alpha> 0$.  \margin{C: I did not get this remark (I remember we discussed it but still I am not sure I get it. We should discuss) \green{N: Let's discuss it this afternoon}}

It is surprising to us that this simple correlation inequality was not outlined neither in Dunlop's work \cite{dunlop1985correlation} nor in Bricmont--Lebowitz--Pfister's work \cite{bricmont_inequalities_1985}.

\begin{proof}

Following \cite{kunz_inequalities_1976}, let us introduce the following change of variable.
\begin{align*}
\sigma_i^{1,2} &= \frac{S_i + T_i}{\sqrt{2}} \\
\sigma_i^{3,4} &= \frac{S_i - T_i}{\sqrt{2}}
\end{align*}
with $S_i \in \R^2$, $T_i \in \R^2$ and $\|S_i\|^2 + \|T_i\|^2  = 1$

$(S_i, T_i)$ is the image of $(\sigma_i^{1,2}, \sigma_i^{3,4} )$ by an orthogonal symmetry thus the measure of $(S_i, T_i)$  is the uniform measure on $\S^3$. The Hamiltonian of $(S_i, T_i)$ is given by 
\begin{align*}
H(S_i, T_i) &= - \sum_{i \sim j} \left( \sigma_i^{1,2} \cdot \sigma_j^{1,2} + \lambda \, \sigma_i^{3,4} \cdot \sigma_j^{3,4} \right) \\
&= - \sum_{i \sim j} \left[ \frac{1}{2}(S_i + T_i)\cdot(S_j + T_j) + \frac{\lambda}{2}(S_i - T_i)\cdot(S_j - T_j) \right] \\
&= - \sum_{i \sim j} \left[ \left(\frac{1+\lambda}{2}\right)(S_i \cdot S_j + T_i \cdot T_j) + \frac{1-\lambda}{2}(S_i \cdot T_j + S_j \cdot T_i) \right]
\end{align*}

Let us write:
\begin{align*}
S_i &= \rho_i U_i \quad \text{where } \rho_i \in \mathbb{R}, \; U_i \in \mathbb{S}^1 \\
T_i &= \tau_i V_i \quad \text{where } \tau_i \in \mathbb{R}, \; V_i \in \mathbb{S}^1
\end{align*}
where $\rho_i \ge 0$, $\tau_i \ge 0$ and $\rho_i^2 + \tau_i^2 = 1$.

The two point function is given by :
\begin{align*}
\langle \sigma_i^{1,2} \cdot \sigma_j^{1,2} \rangle &= \frac{1}{2} \langle (S_i + T_i) \cdot (S_j + T_j) \rangle \\
&= \frac{1}{2} \langle (\rho_i U_i + \tau_i V_i) \cdot (\rho_j U_j + \tau_j V_j) \rangle
\end{align*}

Let us show that the inequality we want to prove actually holds conditionally on  $(\rho_i)_{i \in \Lambda}$, Conditionally on the $\rho_i$ the model can be seen as a ferromagnetic XY model on the extended graph $\Lambda \times \{0,1\}$ (each vertex of $\Lambda$ is split in two vertices one carries $U_i$ and the other one $V_i$).
\margin{C: these graphs are also called "slabs" in percolation literature, but we probably don't need to use this terminology.}

The conditional Hamiltonian is:
\[
H_{\rho}(U_i, V_i) = - \sum_{i \sim j} \left[ \frac{1+\lambda}{2} (\rho_i \rho_j U_i \cdot U_j + \tau_i \tau_j V_i \cdot V_j) + \frac{1-\lambda}{2} (\rho_i \tau_j U_i \cdot V_j + \tau_i \rho_j V_i \cdot U_j) \right]
\]
By Ginibre inequality,
\[
\langle (\rho_i U_i + \tau_i V_i)\cdot(\rho_j U_j + \tau_j V_j) \rangle_{\lambda, \rho}
\]
is greater than the same quantity where we shall replace the coupling constants $\frac{1+\lambda}{2}$ and $\frac{1-\lambda}{2}$ by $\frac{1-|\lambda|}{2}$ in the Hamiltonian. We denote by $ \langle \rangle_{\lambda, \rho}^{-}$ this less correlated measure.

We thus have
\begin{align}
\langle \sigma_i^{1,2} \cdot \sigma_j^{1,2} \rangle_{\rho} &\ge \frac{1}{2} \langle (\rho_i U_i + \tau_i V_i)\cdot(\rho_j U_j + \tau_j V_j) \rangle_{\lambda, \rho}^- \nonumber \\
&=  \frac{1}{2}  \rho_i \rho_j \langle U_i \cdot U_j\rangle_{\lambda, \rho }^- + \rho_i \tau_j \langle  U_i \cdot V_j \rangle_{\lambda, \rho }^-+ \tau_i \rho_j \langle  V_i \cdot U_j \rangle_{\lambda, \rho }^- +  \tau_i \tau_j\langle V_i \cdot V_j \rangle_{\lambda, \rho }^- \label{eq:Lowerbound}
\end{align}

\noindent For every $i \in \Lambda$, either $\rho_i $ or $\tau_i = \sqrt{1 - \rho_i^2} $ has to be greater than  $\frac{1}{\sqrt{2}}$. For any  $i $, let us define $W_i = U_i $ if $\rho_i \ge \frac{1}{\sqrt{2}}$ and $W_i = V_i $ otherwise. We denote by $W'_i$ the other spin.

We then have (using Ginibre first inequality (positivity of two-point functions)),

\begin{align*}
\langle \sigma_i^{1,2} \cdot \sigma_j^{1,2}\rangle_{\rho}  \ge \E_{\rho}(\frac{1}{4} \langle W_i W_j\rangle) 
\end{align*}
as $W_i W_j$ appears in one of the four terms in the of the RHS of Equation (\ref{eq:Lowerbound}). 

The Hamiltonian of $(W_i, W'_i)$ dominates (in the sense of Ginibre inequality)  $\bar{H}(W, W')= - \sum_{i \sim j} \frac{1}{2} \frac {1- |\lambda|} 2 W_i W_j$ which is just  a fraction of the Hamiltonian of the XY model on $\Lambda$ where the spins are the $W_i$.
We thus get for any $\rho$, $ \langle W_i W_j\rangle_{\beta , \lambda} \ge  \langle W_i W_j\rangle^{XY}_{\beta  \frac{1 - |\lambda|}{4}}$.
\end{proof}

\section{Long-range order in the $xy$ plane for $3d$ anisotropic $\S^2$ model}\label{s.d3}

Below is a version of  Theorem \ref{th.main} in the case of  higher dimensions $d\geq 3$. 
We will only outline below how it can be proved using the same ideas as for Theorem \ref{th.main}. 

\begin{theorem}\label{th.d3}
Let us consider the same $\lambda$-perturbation of the classical Heisenberg model (as defined in Definition \ref{d.modif}) in boxes $\Lambda \subset \Z^d$ with $d\geq 3$. Then for any $-1<\lambda<1$, there exists $\beta_d(\lambda)$ sufficiently large so that for all $\beta \geq \beta_d(\lambda)$ there is long-range-order in the $xy$ plane (for free, wired or periodic boundary conditions on $\Lambda_n$, uniformly in $n$). 
\end{theorem}

\begin{remark}\label{r.d3}
Note that long range order for the $\lambda$-anisotropic model can also  be obtained via more standard techniques, using the infrared bound. Our proof has {\em pros}  and {\em cons} compared to the infrared bound (\cite{frohlich1976infrared}). It detects the direction of the long range order which the infrared bound would not allow doing without using the monotonicity arguments of Section \ref{subsec:monotonicity}. More importantly, the proof requires less symmetry compared to the reflection positivity techniques: it works with inhomogeneous coupling constants and for a larger set of boundary conditions (e.g. free, Dirichlet, periodic). Its major drawback though is that the part of the phase diagram where it detects LRO degenerates as the flattening $\lambda\to 0$. The LRO phase does not degenerate as $\lambda \to 1$. The difference between both approaches is due to the fact that we detect with this proof LRO in the $xy$ plane. 
\end{remark}

\ni
{\em Short sketch of proof of Theorem \ref{th.d3}.}

The proof in Section \ref{sec:proof_of_main_thm} extends in any dimension and allows seeing that the two points function of the $\lambda$-anisotropic model is lower bounded by the two point function of the XY model on a percolation on the extended lattice $\Z^d_n$ for some large $n$ that depends on $\lambda$  as defined before. 
Indeed, Lemma \ref{lem:good_box_prob} also holds in higher dimension thanks to Proposition \ref{pr.GS_more_general_on_sites}. 
\margin{C: used to be the longer "Indeed, Lemma \ref{lem:good_box_prob} also holds in higher dimension thanks to Proposition \ref{pr.GS_more_general} with $R^{3/4}$ replace by $R^{d-1  - 1/4}$ thanks to Proposition \ref{pr.GS_more_general_on_sites}." I shortened it. \green{Thanks} }
Let us highlight  one technical step which needs to be adapted when going from $d=2$ to $d\geq 3$. For simplicity, we explain the main difference between $d=2$ and $d=3$ here. In the planar case $d=2$, in a box of scale $R$, if one removes at most $R^{1/2} \log(R)$ points, then there is a unique macroscopic cluster of good sites which remains. In dimension $d=3$, one may need to remove up to $R^{3/2} \log(R)$ points, i.e. more than a linear in $R$ number of "bad sites". This is sufficient to allow in principle several clusters of bad sites, all of them with a macroscopic diameter (of order $R$). But this is not sufficient to allow the coexistence of several clusters of good sites occupying a positive fraction of the volume of $\Lambda_R$. This is because $R^{3/2} \log(R) \ll R^2$ (and when $d\geq 3$ because $R^{d- \tfrac 3 2} \log(R) \ll R^{d-1}$). Because of this, the definition of "good box" needs to be changed as follows: when removing $R^{3/2} \log(R)$ points, there is now a unique macroscopic cluster of \underline{volume} $\Omega(R^3)$ which remains in the complement. This new definition of "good box" is sufficient to run the rest of the argument. Indeed, 
since Dario and the first author established long range order for the XY model on the extended lattice via Nishimori line techniques in Section A.2 \cite{dario2025phase} at low temperature in dimension $d \ge 3$, we obtain long range order in the plane of the equator for the $\lambda$-model at low temperature.

%

\section{\for{toc}{Matching bound for the exponential phase in $\S_+^2$}\except{toc}{A matching bound for the exponential phase in the hemisphere model $\S_+^2$}}\label{s.hemisphere}

In this section we consider the $\S^2_+$ valued anisotropic defined in Section \ref{sec:hemisphere_intro} and prove Theorem \ref{thm:exponential_decay}. To avoid heavy superscripts in the differential calculus that follows, we locally change our spin notation in this section. Rather than our global notation $\sigma_i = (\sigma_i^1, \sigma_i^2, \sigma_i^3)$, we denote the spin at site $i$ by $u_i = (x_i, y_i, z_i) \in \S^2_+$, so that $\S^2_+ = \{ u \in \R^3 \text{ such that } \|u\| = 1 \text{ and } z > 0\}$.

We write $d u$ for the uniform measure on $(\S^2_+)^{\Lambda}$, that is the product of the uniform measure on the sites. The uniform measure $d u_i $ can be written $d u_i = \frac{d x_i d y_i }{z_i} $ where $d x_i $ and $d y_i$ are the Lebesgue measures on the $x_i$ and $y_i$ coordinates. 

As defined in the introduction, the finite-volume Gibbs expectation for this model can be written in our local coordinates as 
\[ \langle f \rangle_{\beta, \lambda} = \frac{1}{Z_{\beta, \lambda}}  \int f(u) \exp\left( - \beta H_{\lambda} (u)\right) d u \] 
where $Z_{\beta, \lambda} = \int \exp\left( - \beta H_{\lambda} (u)\right) d u$ is the partition function. 


As for Theorem \ref{th.main}, everything we state in this section also holds for a slight variation in the definition of $H_{\lambda}$, $H'_\vareps = \frac{1}{2} \sum_{i \sim j} \|u_i - u_j\|^2 + \vareps z_i^2$.
\margin{C: j'ai laissé pour cette version le parametre $\eps$ qui est plus lisible ici que $1-\lambda$ $:)$} We denote by $\langle \cdot \rangle'_{\beta, \vareps}$ the Gibbs measure defined similarly.

When $\lambda=1$ (or $\vareps = 0$), Bauerschmidt, Helmuth and Swan \cite[Theorem 8.3]{bauerschmidt2021geometry} have shown that on any finite graph,
$\langle \sigma_i \cdot \sigma_j \rangle \le \exp\left( - c(\beta) |i-j|\right)$ for some constant  $c(\beta)$ that depends on $\beta$ and on the maximum degree of the graph. Their proof relies on a random walk isomorphism theorem that is obtained from an integration by part technique that is reminiscent of the one McBryan and Spencer used in \cite{mcbryandecay1977} to show exponential decay of correlations for the discrete Gaussian free field conditioned to remain in an interval. See Bricmont--El Mellouki--Fröhlich \cite{bricmontrandom1986} for an in depth study of this model \margingreen{Lucas sent me this paper in Santiago and it solves all the questions I had about this model (mostly, does the variance of a spin goes to zero in dimension $2$ when $\beta \to \infty$, the answer is yes but at a much smaller rate than for $d \ge3 $ and gives a simple proof of uniqueness of Gibbs measure (due to Sokal))}. We recall and extend their random walk isomorphism to the $0 \le \lambda \le 1$ model and use it to deduce Theorem \ref{th.mainS+}. Let us point out that on a general transitive graph, we do not need the Gibbs measure to be translation invariant under the full symmetry group of the graph, any transitive subgroup suffices.\margingreen{N:I removed the repetition of Thm 1.7 }




The proof is structured as follows. We state and prove the isomorphism theorem between the anisotropic hemisphere model and a nearest neighbor random walk  model in Section \ref{sec:random_walk_isomorphism}. In Section \ref{sec:proof_of_expo_decay}, we explain how to deduce Theorem \ref{th.mainS+} from the random walk process.

\subsection{The random walk isomorphism}
\label{sec:random_walk_isomorphism}
This section is dedicated to describe the connection that we use between the (anisotropic) hemisphere model and the vertex diminished jump process (VDJP) that we shall define in Section \ref{sec:vdjp}. These type of results are called {\em isomorphism theorems} by Bauerschmidt, Helmuth and Swan in \cite{bauerschmidt2021geometry}.
\subsubsection{Local symmetry of $\S^2_+$ and integration by parts}

The integration by parts formula arises from (local) symmetries of spaces, let us first describe the local symmetry of $\S^2_+$ that we use.

Let $T$ be the differential operator on functions defined on $\S^2_+$ that we see in coordinates as $B(0,1)$ through the diffeomorphism $(x,y) \mapsto (x,y, \sqrt{1 - x^2 - y^2})$ defined by $T f(x,y) = z \partial_x f(x,y) $. We point out that $T x = z $ and $T z =  z \frac{\partial}{\partial x} \sqrt{ 1 - x^2 - y^2 } =  - x$. Thus, $T$ is the generator of the rotation around the $y$ axis in $\S^2$. Notice also that if $f(x,y) = g(x,y,z)$, then $T f = z \partial_x g - x \partial_z g $. By a slight abuse of notations, we will also denote by $T$ the operator on functions defined on $\S^2_+$ defined by $T g  = z \partial_x g - x \partial_z g $  The operator $T$ is not the generator of a true symmetry of $\S^2_+$ but of a local symmetry in the following sense. Let $g $ be any compactly supported smooth function on $\S^2_+$. \margin{C: slight abuse of notation here? $f$ is defined on $\D \subset \R^2$ while $g$ on $\S_+^2$, right ?, \green{N: is it ok likes this, i added one sentence?} } Then,
\[
\int T g = 0 ~.
\]
where the integration is with respect to the uniform measure on $\S^2_+$.
\margin{C: here the integration is w.r.t to uniforms measure on $\S^2_+$ or the projected measure on $\D\subset \R^2$ ? more consistent with the notation $f$ ? \textbf{Ah Ok, but now $f$ is viewed as a function defined on $\S^2_+$ right ?? }}
This is analogous to the fact the integral of the derivative of any compactly supported symmetric function on $\R$ is zero which can be seen as encoding the translation symmetry of the real line.
This can be proved as follows. As $g$ is compactly supported, we can extend $g$ in a smooth way on the full sphere $\S^2$ by setting $g(u) = 0 $ for $u$ in the lower hemisphere. The operator $T$ also extends on the full sphere (it is given by the same expression in the lower hemisphere map). Now, notice that if $R_\theta$ is the matrix of the rotation of angle $\theta $ around the $y$ axis, the integral $\int_{u \in \S^2} g(R_\theta u ) d u$ does not depend on $\theta$.
Differentiating with respect to $\theta$ and taking $\theta = 0$ provides the desired result as 

\[ R_\theta (u ) =  \left(x \cos(\theta) + z \sin(\theta ) , y , z \cos(\theta) - x \sin(\theta)\right) \]
and \[ \frac{d}{d \theta}  g(R_\theta u)_{|\theta=0} = z  \partial_x g - x \partial_z g = T g ~. \]

\margin{C: also here, slight abuse of notation between $f$ and $g$. Formally $T f$ is not defined this way. }
\margin{C: thanks for adding this proof !!}



\subsubsection{The vertex diminished jump process (VDJP)}
\label{sec:vdjp}
Let us now describe the random walk process that Bauerschmidt, Helmuth and Swan have related to the hemisphere model and that we will then adapt to the anisotropic hemisphere model.

The process can be viewed as a continuous time Markov process defined on $\Lambda \times (\R_{\ge0})^{\Lambda}$. Informally, we start with a configuration of an initial time at each site $L_j (0) \ge 0 $ for $j$ in $\Lambda$ and an initial position $i_0$ in $\Lambda$. We denote by $X(t)$ the position of the particle at time $t$. When the particle is at site $i$, it may jump on a neighbor $j$ at rate $\beta L_j$. As long as the particle is at $i$, $L_i$ decreases linearly at speed $1$ that is $L_j (t) = L_j (0) - \int_0^t \mathds{1}_{(X(t) = j)} d t$. As soon as one of the $L_j$ hits $0$, the process is killed that is $X(t) = j$ for all later $t$. For reasons that will be clear soon, we denote by $(i, (z_j)_{j \in \Lambda})$ elements of $\Lambda \times \R_{\ge0}^{\Lambda}$.

More formally, the generator $L_{\beta}$ of the process $(X(t), (L_i (t))_{i \in \Lambda}$ acts on functions defined on $\Lambda \times \R_{ \ge 0}$ which take the value $0$ if any of the $z_i = 0$ by 
\margin{C: stupid question : the generator $L_\beta$ is not allowed to act on functions which do not vanish on this boundary condition ?}
\margingreen{N: The short answer is yes it does apply but you have to assume that $z_i>0$ for it to apply. If you try to write down the proof that this is the generator of the process for a function that does not vanish then the expression in only correct  when there is some remaining time, when the time is $0$, then the generator should just be $0$. In other words the correct generator for more general functions is the same time the indicator that $z$ is not $0$. }

\[ L_{\beta} f(i,z) = \sum_{j, j  \sim i } z_j \beta_{i j}  (f(j,z) - f(i,z)) - \partial_{z_i} f(i,z) ~. \]

\subsubsection{The $\lambda$-isomorphism theorem}

From now on, we denote by $T_i = z_i \partial_{x_i}$ the differential operator acting on functions $f: (\S^2_+)^{\Lambda} \to \R$. The following $\lambda$-isomorphism theorem which is a generalization of Lemma 4.2 in \cite{bauerschmidt2021geometry} holds.

\begin{theorem} \label{th:not_integrated}Let $\rho : (\S^2_+)^{\Lambda} \to \R$  and $f: \Lambda \times (\R_{> 0})^{\Lambda} \to \R$ be compactly supported, then

\begin{align} \label{eq:formule_ibp} \<{ \sum_{j \in \Lambda } (T_j \rho) f(j,z)}_{\beta, \lambda} &= \sum_{j \in \Lambda } \<{x_j \rho(u) ( - L_{\beta} f )(j,z)}_{\beta, \lambda} - \beta (1-\lambda) \sum_{j \in \Lambda } \langle x_j (\sum_{i, i  \sim j } z_i) \rho(u) f(j,z)\rangle_{\beta, \lambda} \nonumber \\ &\quad  ~. \end{align}\end{theorem}

\margin{C: Here $T_j$ seems to act on a function parametrized by $\S^2_+$ instead of $\D$ ?}

For the modified Hamiltonian the result can be written as follows
\begin{align*} \langle \sum_{j \in \Lambda } (T_j \rho) f(j,z)\rangle'_{\beta, \vareps} &= \sum_{j \in \Lambda } \langle x_j \rho(u) ( - L_{\beta} f )(j,z)\rangle '_{\beta, \vareps} - 2 \beta \vareps \sum_{j \in \Lambda } \langle x_j z_j \rho(u) f(j,z)\rangle'_{\beta, \vareps}~. \nonumber \\ 
\end{align*}

We can also write an integrated version that carries a more probabilistic meaning. Let $\E_{j,z}$ denote the expectation of the VDJP starting at $j$ with initial times $z = (z_i)_{i \in \Lambda}$.

\begin{theorem} \label{th:integrated}
For any $\rho : (\S^2_+)^{\Lambda} \to \R$  and any  compactly supported function $f: \Lambda \times (\R_{> 0})^\Lambda \to \R$, we have 

\begin{align} \label{eq:integrated_version}
& \sum_{j \in \Lambda}\langle (T_j \rho ) \E_{j,z} \int_{t = 0}^{\infty} f(X_t, L_t) d t \rangle_{\beta, \lambda}  \nonumber \\ & = \sum_{j \in \Lambda} \langle x_j \rho f(j,z)\rangle_{\beta, \lambda}  - \beta (1-\lambda) \sum_{j \in \Lambda } \langle x_j \rho(u) \, (\sum_{\substack{i, i \sim j }} z_i) \E_{j,z} \int_0^\infty f(X_t, L_t) d t \rangle_{\beta, \lambda} \nonumber \\ \end{align}
\end{theorem}

\begin{proof}[Proof of Theorem \ref{th:integrated} using Theorem \ref{th:not_integrated}]
Let $f_t (j,z) = \E_{j,z}  f(X_t, L_t) $
By definition of the VDJP, 
$ \frac{d}{d t } f_t = L_{\beta} f_t$. Using also that $\E$ commutes with $L_\beta$, the integrated version is thus the result of integrating Equation (\ref{eq:formule_ibp}) applied to $f_t$ with respect to $t$. Notice that $t \mapsto f_t$ is bounded in $L^1(\R_+)$ as the process$(X_t,L_t)$ is killed at a time that is bounded by a geometric random variable.
\margin{C: I did not get this comment, isn't $f$ bounded since it is compactly supported ? In which case is is not only $L^1$ but $L^\infty$ ?}
\end{proof}

\subsubsection{Proof of Theorem \ref{th:not_integrated}}

We drop the $\beta$ and $\lambda$ in the subscripts of all measures to make notations less heavy.
By the integration by part formula on $u_j$, 
\margin{C: Definitively too fast for me here $:)$ !}
\begin{equation} \label{eq:ibpfirststep}
\langle (T_j \rho) f(j,z)] = - [\rho  T_j f(j,z) ] + \beta [\rho T_j H_{\lambda}\rangle ~.\end{equation}

Let us now compute $T_j H_{\lambda}$. To simplify slightly the computation, we note that, up to a constant, 

\[ H_{\lambda} (u) = - \sum_{i \sim j } u_i \cdot u_j + (1 - \lambda)z_i z_j \] and we use this expression to compute $T_j H_\lambda$.

\begin{align*} T_j H_\lambda (u) &= - \sum_{ i, i \sim j } u_i \cdot (z_j \partial_{x_j} u_j ) + ( 1 - \lambda) (\sum_{i, i \sim j } z_i )   \partial_{x_j} z_j \nonumber \\ 
&= - \sum_{ i, i \sim j }(x_i z_j - x_j z_i) - (1 - \lambda) x_j \sum_{ i, i \sim j } z_i ~. 
\end{align*}

By a discrete integration by part on the sum, we see that

\begin{align*} \sum_{j \in \Lambda} \rho(u) f(j,z) \sum_{i, i \sim j } ( x_i z_j - x_j z_i ) = \sum_{j \in \Lambda } x_j \rho(u) \sum_{i, i \sim  j} z_i (f(i,z) - f(j,z))
\end{align*}

and thus 

\[ \sum_{j \in \Lambda} \rho(u) \, f(j,z) \sum_{i, i \sim j } ( x_i z_j - x_j z_i ) = \sum_{j \in \Lambda } x_j \rho(u) (L_\beta f(j,z))  + \sum_{j \in \Lambda} x_j \rho(u) \, (\partial_z f)(j,z) ~. \]

Note that $T_j f = z_j \partial_{x_j} (f(j,z)) = - z_j  \frac{x_j}{z_j} \partial_{z_j} f(j,z) = - x_j \partial_{z_j} f(j,z)$.

Thus 
\[ \sum_{j \in \Lambda} \rho(u) f(j,z) \sum_{i, i \sim j } ( x_i z_j - x_j z_i )= \sum_{j \in \Lambda } x_j \rho (L_\beta f(j,z)) - \sum_{j \in \Lambda} \rho(u) \, (T_j f)(j,z) \]

We thus obtain 
\[
\sum_{j \in \Lambda } \rho \, f(j,z) T_j H_\lambda - \rho \, (T_j f )(j,z) = - \sum_{j \in \Lambda } x_j \rho (L_\beta f(j,z)) - (1 - \lambda) \beta  \sum_{j \in \Lambda} f(j,z ) \rho \sum_{i , i \sim j } z_i ~. \]

Integrating this last identity with respect to $\langle\cdot \rangle$ and combining with Equation (\ref{eq:ibpfirststep}) leads to Theorem \ref{th:not_integrated}. \qed

\subsection{Proof of Theorem \ref{thm:exponential_decay}}
\label{sec:proof_of_expo_decay}
Now that we have derived a suitable isomorphism theorem for the anisotropic hemisphere model, we describe how to obtain Theorem \ref{thm:exponential_decay}.

\subsubsection{Application of Theorem \ref{th:integrated}}
Let us first specify what functions we apply Theorem \ref{th:integrated} to. This part is identical to the proof of Theorem 4.4 in \cite{bauerschmidt2021geometry}.

We take $\rho = x_a $ and $f(j,z) = \delta_{j,b} g(z) $ for some compactly supported function $g$ in Equation (\ref{eq:integrated_version}).

Note that $T_j \rho =  \delta_{j,a} z_a$. This leads to 

\[ \langle z_a  \E_{a,z} \int_{t = 0}^{\infty} g(L_t) \mathds{1}(X_t = b) d t \rangle  = \  \langle x_a x_b g(z)\rangle - \beta (1-\lambda) \sum_{j \in \Lambda} \langle x_j x_a (\sum_{i, i \sim j } z_i) \E_{j,z}( \int_{t = 0}^{\infty} g(L_t)\mathds{1}_{X_t = b} d t )\rangle~. \]

This identity also holds with $g = 1$ by dominated convergence.

By Ginibre inequality,  $\langle x_j x_a | z \rangle \ge 0  $, and thus
\[ \langle x_j x_a (\sum_{i, i \sim j } z_i) \E_{j,z}( \int_{t = 0}^{\infty} \mathds{1}_{(X_t = b)} d t )\rangle \le d \langle x_j x_a \E_{j,z}\left (\int_{t = 0}^\infty \mathds{1}_{(X_t = b)} \right ) d t \rangle \]
where $d$ is the maximum degree of a site. This leads to 
\margin{C: you also use $[z_a  \E_{a,z} \int_{t = 0}^{\infty} g(L_t) \mathds{1}(X_t = b) d t ] \leq [1  \E_{a,z} \int_{t = 0}^{\infty} g(L_t) \mathds{1}(X_t = b) d t ]$ here right ? N: Indeed }

\[ \langle x_a x_b\rangle  \le   \langle  \E_{a,z} \int_{t = 0}^{\infty} \mathds{1}_{(X_t = b)} d t \rangle + d \beta  (1-\lambda)  \sum_{j \in \Lambda} \langle x_a x_j \E_{j,z} \int_{t = 0}^\infty \mathds{1}_{(X_t = b)} d t  \rangle \]

Now notice that $\E_{j,z} \int_{t = 0}^{\infty} \mathds{1}(X_t = b) d t \le \exp( -c(\beta) d(j,b))$ as at each step, the walk has a positive probability (that depends on $\beta$) of dying before the next step. We can take $c(\beta) =  \exp( - 2 \beta d )$ as $\exp(- 2 \beta d )$ is the probability that the minimum of $d$ independent exponential random variables of parameter $\beta$ (i.e. an exponential random variable of parameter $\beta d$) is greater than $1$.

Using Ginibre inequality one more time, we see that 
\margin{C: I think there is an issue of normalisation here. Shouldn't we have a partition function $Z_\lambda$ in front of $\exp(- c(\beta) d(a,b)) $ ? \textbf{In fact, I am not sure I have understood why we stick to unormalized $[..]$ rather than usual Gibbs measure $\<{...}$. This may induce small issues such as this one and overcomplicate things for not much unless I miss something ? -> \green{I removed any unnormalized measure, The only reason to stick to this was for the super symmetric model}}}

\begin{equation}\label{eq:convolutionupperbound}  \langle x_a x_b \rangle  \le \exp(- c(\beta) d(a,b)) + d \beta  (1-\lambda)  \sum_{j \in \Lambda} \langle x_a x_j \rangle \exp(- c(\beta) d(j,b))~. \end{equation}

Note that  if $\Lambda \subset \Z^d$  (or any transitive graph) the same inequality holds for any infinite volume subsequential limit.  \rm\margin{C: whatever boundary conditions ? should still have Ginibre right ?}

\margin{C: Si $\lambda<0$, cette technique dit encore qqchose mais de moins interessant que à $\lambda=0$, le $\beta$ de decroissance exponentielle est DECROISSANT en $\lambda$. !! Il etait super quand $\lambda$ tend vers 1 !! mais vers $-1$ c'est nul. Mais c'est pas pour rien, car de toute facon il n'a peut etre pas decroissance exponentielle !}

\subsubsection{A convolution inequality}

Let us now deduce Theorem \ref{thm:exponential_decay} from Equation \ref{eq:convolutionupperbound}. 
To this end, we prove the following lemma. 
\margin{C: I feel I have seen the same Lemma not long ago in the literature, maybe Random walks on groups ? Can't quite remember unfortunately, we should investigate $:)$}

\begin{lemma} \label{lemma:convolution}  Let $f$ be a nonnegative and  bounded function defined on $\Z^d $ that satisfies 

\[ f \le g + \kappa f \ast g \]

where $f \ast g$ is the convolution of $f$ and $g$ that is 
$(f \ast g )(a)= \sum_{k \in \Z^d} f(a-k) g(k) $

Then if $g(k) \le \exp( - c d(0,k))$ for some $ c > 0$ and if $\kappa \|g\|_{L^1} < 1 $, then it holds that $f(k) \le  \exp( - c' d(0,k))$ for some $c' > 0 $ that depends on $c$, $d$ and $\kappa \|g\|_{L^1}$.

\end{lemma}
This lemma immediately implies Theorem \ref{thm:exponential_decay} by taking $f(k) = \langle x_0x_k \rangle = \langle x_a x_{a +k} \rangle$ for any $a \in \Z^d$ and $g(k) = \exp( - c(\beta) d(0,k))$.
Indeed, 
\[ \|g\|_{L^1} = \sum_{k \in \Z^d} \exp( - c d(0,k)) = (\sum_{k \in \Z} \exp( - c(\beta) k ))^d = (\frac{1}{ 1 - \exp(- c(\beta))})^d ~.\]
As we can take $c(\beta) =  \exp( - 2 d \beta)$, we see that  $\|g\|_{L^1}  \le  \exp( 4 ( \frac{e}{e-1}) \beta d^2)$ by convexity.

Let us now prove Lemma \ref{lemma:convolution}. We start with the following remark. 
Let $f$ and $g$ be two bounded real valued functions defined on $\Z^d$ such that
$f(k) \le g(k)  + \kappa (f \ast g)(k) $ and 
\begin{equation} \label{eq:convolution} h (k) = g(k) + \kappa (h \ast g)(k) ~. \end{equation}
  Then if, $\|g\|_{L_1} < \frac{1}{\kappa}$, it holds that $f \le h$.

Indeed, by applying the inequality and equality on $f$ and $h$ repeatedly we obtain  \[ h - f \le \kappa g \ast (h - f ) \le \kappa^n g^n \ast (h -f).  \]

Thus, \[ h - f    \le \kappa^n \|g\|_{L^1}^n  \|h - f \|_{\infty} \to_{n \to  \infty} 0 ~.\] When $g$ is bounded in $L^1$ and $\kappa \|g\|_{L^1} < 1 $, the  functional equation in $h$,
$h = g + \kappa h \ast g $ always admits exactly one solution that is bounded in $L^1$ that we make explicit using Fourier series.

For any function $f: \Z^d \to \R$ and $\theta \in (\S^1)^d$, we write $\hat{f}(\theta) = \frac{1}{(2 \pi )^d} \sum_{k \in \Z^d} f(\theta \cdot k ) \exp( i k \cdot \theta ) $.

In Fourier mode, the convolution equation writes
$ \hat{h} = \kappa \, \hat{g} \,   \hat{h} + \hat{g}$ that is $ \hat{h} = \frac{\hat{g}}{1- \kappa \hat{g}}  = \sum_{k= 1}^{\infty} \hat{g}^k \kappa^{k-1}  $ where $\hat{h}$ is the Fourier transform of $h$ and \[ \kappa \hat{g}(\theta) = \kappa \sum_{k \in \Z^d} \exp(-i k \theta) g(k)\]  has a module smaller than $1$
as $\kappa \|g\|_{L^1} < 1$.

Thus inverting the Fourier transform, 

\[ h(n) = \sum_{k = 1}^{\infty} (g)^{\ast k}(n) \kappa^{k-1}  = \frac{1}{\kappa} \E_0( N_n^{\kappa }) \]

where $N_n^{\kappa}$ is the number of visit at $n$ of a massive random walk  with increment distribution proportional to $g$ and that has a probability of dying of $ 1 - \kappa \|g\|_{L^1}$ at each step where we do not count the starting point as a visit.

Let $(X_j)_j$ be a sequence of i.i.d. random variables with the distribution of the increments of the random walk conditioned on not dying. If \[\P(d(0,X_1) \ge k) = \sum_{a | d(0,a \ge  k)} \frac{g(k)}{\|g\|_1} \] has exponential tail, then so does $\kappa h(k) = \E_0(N_k^\kappa)$.

Indeed, let $T^\dagger$ denote the time at which the random walk dies and $T_k$ the first times it is at distance greater or equal to $k$ from $0$.
Then by the Markov property,
\[\E_0(N_k^\kappa ) \le \P(T_k \le T^\dagger) \E(T^\dagger)~. \]
The expectation $\E(T^\dagger)$ is finite as $T^\dagger$ is a geometric random variable. For any $\alpha > 0 $, \[ \P(T_k \le T^\dagger) \le  \P(T_k \le \alpha k )+ \P(T^\dagger > \alpha k)~.\]
Now as the $d(0,X_i)$ are i.i.d. random variables that are stochastically dominated by a geometric random variable of a parameter that we shall call $a$,

 \begin{align*} \P(T_k \le \alpha k ) &\le \P(\sum_{j = 1}^{\alpha k }  d(0, X_j) \ge k) \le \E(\exp(\lambda d(0,X_i)))^{\alpha k } \exp(- c \lambda k ) \nonumber \\ &\le (a \frac{1}{1- (1 - a) \exp(\lambda)})^{\alpha k } \exp(- c  \lambda k ) ~. 
 \end{align*}
Taking $\exp(\lambda) = 2 $ and some small $\alpha$ that depends on $a$, we obtain,

\[ \P(\sum_{j = 1}^{\alpha k }  d(0, X_j) \ge k) \le \exp( - c \ln(2)  k )  (\frac{a}{1 - 2(1-a)})^{\alpha k }\le \exp(  - c' k) ~ \] provided that $\alpha$ is small enough.

On the other hand $T^\dagger $ is a geometric  random variable so $\P(T^\dagger > \alpha k ) \le \exp(- c \alpha k )$ for some $c > 0$.

We have thus proved Lemma \ref{lemma:convolution} on $\Z^d$. To extend it to any non-exponential growth transitive graph (replacing $d$ by $1/2$ of the maximum degree of the graph in the statement), notice that $g(n)  = \frac{1}{\kappa} \E_0(N_n^\kappa)$ where $h(n) = N_n^{\kappa}$ is defined as before is still a solution to Equation (\ref{eq:convolution}) although the Fourier derivation does not make sense anymore. The rest of the proof can be done similarly. Non-exponential growth is needed to guarantee that $g$ is in $L^1$ and that $d(0,X_1)$ has exponential tails.

\section{Open Problem}

The main open question raised by this article is
\begin{OP}\label{op.1}
What happens for the $\eps$-perturbation of the Riemannian metric on the sphere $\S^2$ as defined in~\eqref{e.metric} in Definition \ref{d.modif} ? 

Considerations based on the Ricci flow evolution of this metric, following Friedan \cite{friedan1985nonlinear} suggest the phenomenology may be radically different.

More precisely, we believe that exponential decay may hold for those metric perturbations for any fixed $\eps$ at large enough $\beta$ in contrast to our main result Theorem \ref{th.main} on scalar product perturbation and Theorem \ref{th.allS} on other possible smooth perturbations.
\end{OP} 

\margin{\textbf{:C: maybe remove that ?} Yet, as hinted to us by Stephen Shenker, one shouldn't rule out the possibility that quadratic terms confining the spin systems towards $\S^1$ kick in along the RG flow. }

\margingreen{N: What do you think of this:, we could also conjecture that it is $ln(\eps)$ according to Rychkov (but does it apply to half sphere?) 
What is the asymptotic value of $\beta_{BKT}(\varepsilon)$ when $\varepsilon \to 0$ when $\varepsilon \to 0$ for the hemisphere model? \textbf{C: in retrospect, maybe we don't want to advertise this second conjecture ? AI will solve it otherwise... Are you OK ? The first one is fine, it looks super difficult $:)$ }}

\appendix 

\section{Correlations inequalities for the anisotropic model }

In this section we discuss some natural correlations inequalities and establish monotonicity in $\lambda$ of some observables of the model when $N = 3, 4$. Most of these correlations inequalities were first proved in two 1976 papers of Kunz--Pfister--Vuillermot \cite{kunz_inequalities_1976} and Dunlop \cite{dunlop1985correlation} We also extend the Simon--Lieb inequality to the anisotropic model when $N = 3$ and $N= 4$. Our proof of Simon--Lieb's inequality is based on the local Ward identity proof of Aizenman and Simon \cite{aizenman_local_1980} for the (isotropic) $N= 3,4$ models of the same inequality. All these correlations inequalities rely heavily on writing the $N = 4$ model as two correlated XY model and the $N = 3$ case as an XY model correlated with an XY model (although for the sake of brevity we do not state it exactly in this way for $N = 3$). In particular, they do not extend to the $N \ge 5$ case.

\subsection{Monotonicity in $\lambda$}\label{subsec:monotonicity}

\begin{proposition}\label{prop:monotonicity}
With free or periodic boundary condition if $N = 3, 4$, that is when the spins take value in $\S^2$, the following hold.
The function $\lambda \mapsto \<{ \sigma_i^1 \sigma_j^1   }_{\beta, \Lambda , \lambda}$ is non-increasing in $\lambda$ for $\lambda$ in $[0, + \infty  )$. The function $\lambda \mapsto \<{\sigma_i^3 \sigma_j^3}_{\beta, \Lambda , \lambda}$ is non-decreasing in $\lambda$ for $\lambda$ in $[0, + \infty )$.
\margin{C: this is very strange ? at $\lambda=-1$ and $\lambda=1$, this maybe exponentially decaying, but not at $\lambda=0$ ? \textbf{also maybe, we should advertise this monotony a bit more ? For example in Figure \ref{f.PhaseDiagram1} ? } N: yes the mistake was only in the intervals, we can advertise it more, do you mean to change the diagram, or just the caption?}

\end{proposition}

As was already noticed by Bricmont, Lebowitz and Pfister in \cite{bricmont_inequalities_1985}, this proposition is a consequence of Theorem 1 of the article of Dunlop \cite{dunlop_correlation_1976} (Theorem \ref{thm: DunlopversionofGinibre}), that ensures that 
$\<{\sigma_i^1 \sigma_j^1  \sigma_k^3 \sigma_l^3 } \le  \<{\sigma_i^1 \sigma_j^1}\<{  \sigma_k^3 \sigma_l^3 } $ for any $i, j , j$ and $l$ in $\Lambda$ and that 

\[ \<{\sigma_i^3 \sigma_j^3  \sigma_k^3 \sigma_l^3 } \ge  \<{\sigma_i^3 \sigma_j^3}\<{  \sigma_k^3 \sigma_l^3 } ~.\]

The proof is very specific to the $N = 3,4$ case, it can also be applied to $\S^3$ valued model with Hamiltonian $H(\sigma) = - \sum_{ i \sim j }  \lambda_1 \sigma_i^{1}  \sigma_j^{1} + \lambda_2 \sigma_i^{2}  \sigma_j^{2}  + \lambda_3 \sigma_i^{3}  \sigma_j^{3} + \lambda_4 \sigma_i^{4}   \sigma_j^{4}$ with $\lambda_1 \ge |\lambda_2| $ and $\lambda_3 \ge |\lambda_4|$.

We choose to use Dunlop's settings instead of the one of Kunz--Pfister--Vuillermot but this is only a matter of taste as they lead to the same results. Let us recall this setting. Consider the following coordinate systems in $\R^4$. If $\sigma \in \R^4$, we write $\sigma_i^{1,2} = (\rho \cos(\phi), \rho \sin(\phi)) $ and 
$\sigma_i^{3,4} = (\tau 1 \cos(\psi), \tau \sin(\psi))$ with $\rho^2 + \tau^2 = 1$. Let us assume that $\mu$ is some measure on $\R_+^2$. We see $\R^3 = \R^3 \times \{0\} $ as being embedded in $\R^4$  Let $Q_1$ be the set of functions of $(\rho, \psi)$ that is stable under multiplications and additions and that is generated by functions of the form $f(\tau) \cos( n \cdot \psi)$ where $f$ is some nonnegative, nondecreasing and bounded function of $\tau$ and $n: \Lambda \to \Z $. Let $Q_2$ be defined analogously for the variables $\tau$ and $ \psi$.

Let us assume that $\mu$ satisfies the hypothesis of negative correlation \margingreen{ N: I wonder why this is called like this}:

\[ \int \prod_{ \substack{f \in A_1 \\ g \in A_2} } (f(\rho ) \pm  f(\rho'))    (g(\tau' ) \pm  g(\tau)) d \mu(\rho , \tau) d \mu(\rho', \tau)' \ge 0 \]

where $A_1$ and $A_2$ are any finite sets of nondecreasing, nonnegative bounded functions.

\margingreen{ Ce qui fait tout marcher ici est que cette hypothèse est vraie pour la projection de la mesure uniforme sur $\S^3$ sur les coordonnées $(\rho, \tau)$. Cela revient à dire que  $\int \prod_f (f(r) \pm f(r')) \prod_g (g(\sqrt{1 -r'^2}) \pm g(\sqrt{1 - r^2}))  d \alpha(r) d \alpha(r') \ge 0$ pour n'importe quel ensemble de fonctions $\{f\}$ et $\{g\}$ croissantes et positives  et pour une certaine mesure sur $\alpha$. En fait on voit, que c'est vérifié par toutes les mesures $\alpha$. En effet, comme $g(\sqrt{1 - r^2}) $ est décroissante, 
\[\int \prod_f (f(r) \pm f(r'))   d \alpha(r) d \alpha(r') \]. En inversant les rôles de $r$ et $r'$, on réalise qu'il y a un nombrede paire de moins  dès lors que l'intégrale est non nulle, l'intégrande est alors positif car on peut regroupe les termes par deux et $(f(r) - f(r') )(g(r) - g(r'))\ge 0$ pour tout $r$.}

Let $f \in Q_1$ and $g \in Q_2$, and let $(\sigma_i)_{i \in \Lambda}$ be a family of random vectors of $\R^4$ with distribution proportional to $\exp(- f - g ) d \mu(\rho, \tau ) \prod_{i \in \Lambda} (d \phi_i d \psi_i).$

\begin{theorem}{ Theorem 1 of \cite{dunlop_correlation_1976}}\label{thm: DunlopversionofGinibre}
For any $f_1, g_1 \in Q_1 $ and $g_2 \in Q_2$.

$\<{f_1 g_1 } \ge \<{f_1}\<{ g_1}$ and $\<{f_1 g_2} \le \<{f_1 } \<{g_2}$
\end{theorem}

The $N =3$ and $N = 4$ anisotropic models both fit in Dunlop's settings thanks to the following observation due to Kunz--Pfister--Vuillermot.  The Hamiltonian of the model may be written \begin{align} H(\sigma)  = - \sum_{ i \sim j }  & \frac{\lambda_1 + \lambda_2}{2} \rho_i \rho_j \cos(\phi_i - \phi_j) + (\lambda_1 - \lambda_2 ) \rho_i \rho_j \frac{\cos(\psi_i + \psi_j)}{2}\\  \label{eq:Hamiltonian_ala_KPV} &+ \frac{\lambda_3 +\lambda_4}{2} \tau_i \tau_j \cos(\psi_i - \psi_j)     + \frac{\lambda_3 - \lambda_4}{2} \tau_i \tau_j \cos(\psi_i + \psi_j) \end{align} where $\lambda_4 $ can be conveniently set to any value when  $N = 3$. When $N = 4$, the measure $\mu$ is the uniform measure on $\S^3$ and when $N= 3$,
we choose the measure $\mu$ on $\S^3$ defined by $d \mu  = d \rho d \phi (\frac{1}{2} (\delta_{\psi = 0} + \delta_{\psi =1}) )$. It is easily checked that both measures satisfy the negative correlation assumption. We thus obtain the following correlation inequalities.

\[ \<{\prod_i \sigma_i^a  \prod_j \sigma_j^a } \ge \<{\prod_i \sigma_i^a }\<{\prod_j \sigma_j^a } \] for any $a = 1, 3$ and that 

\[ \<{\prod_i \sigma_i^1  \prod_j \sigma_j^3 } \le \<{\prod_i \sigma_i^3 }\<{\prod_j \sigma_j^3 } \] which proves Proposition \ref{prop:monotonicity}.

Notice that when $\lambda_1 \ge  |\lambda_3|$ and $\lambda_2 \ge |\lambda_4|$  we can switch the roles of the second and third coordinates which leads to 

\begin{equation}\label{eq:plane_correl_ineq}\<{\prod_i \sigma_i^1  \prod_j \sigma_j^2 } \le \<{\prod_i \sigma_i^2 }\<{\prod_j \sigma_j^2 } ~. 
\end{equation}

Thus, this holds when $\lambda \in [-1,1]$ for the $\lambda$-anisotropic model.

 We crucially need this last inequality for the proof of the Simon--Lieb inequality below.

\subsection{Simon--Lieb inequality}
\label{sec:simon-Lieb}
Aizenman and Simon generalized the well-known Simon--Lieb inequality of the Ising model to the classical Heisenberg model in \cite{aizenman_local_1980}. We extend this result to the anisotropic model.

\begin{theorem}[Simon--Lieb inequality for the anisotropic model]\label{thm:simon_lieb}
  For any $\lambda \in [-1,1]$, for any finite subset $A$ of $\Lambda$, $ i \in A $, $j \in \Lambda \backslash A $, we have,

  \[ \<{\sigma_i^1 \sigma_j^1 }_{\beta, \lambda} \le  \beta \sum_{l \in A , k \in A^c , l \sim k }\<{\sigma_i^1 \sigma_l^1 }_{\beta, \lambda}\<{\sigma_k^1 \sigma_j^1 }_{\beta, \lambda} \]
\end{theorem}
\begin{remark}\label{rk:hamiltonian}
  The result also holds with Hamiltonian \[ H(\sigma) = - \sum_{ i \sim j }  \lambda_1 \sigma_i^{1}  \sigma_j^{1} + \lambda_2 \sigma_i^{2}  \sigma_j^{2}  + \lambda_3 \sigma_i^{3}  \sigma_j^{3} + \lambda_4 \sigma_i^{4} \cdot  \sigma_j^{4}\]  with $\lambda_1 = \lambda_2 \ge |\lambda_3| \ge |\lambda_4|$.
\end{remark}
\begin{proof}
  This proof is due to Aizenman and Simon and is based on a local Ward identity. We recall it for the convenience of the reader. The only new ingredient we need to extend it to the anisotropic case is Equation (\ref{eq:plane_correl_ineq}). For $A \subset \Lambda$, we define $\gamma_t^{A}$ as one parameter group that acts $\Lambda^{\S^2}$ by rotation of angle $t$ of all the spins in $A$ and as the identity on the spin outside $A$. That is \[\gamma_t^A(\sigma)_i^l = \begin{cases}

    \sigma_i^l \text{ if $i \notin A$ or if $l = 3,4$ } \\
    \cos(t) \sigma_i^1 + \sin(t) \sigma_i^2 \text{ if $i \in A $ and $l = 1 $}\\
    - \sin(t) \sigma_i^1 + \cos(t) \sigma_i^2 \text{ if $i \in A $ and $l = 2$ }
  \end{cases} \]
  Let $f$ be any observable, then as $\frac{d}{d t }\<{\gamma_t^A(f) \exp( - \gamma_t^A(\beta  H ) ) }_{\beta = 0 } = 0$, \[ \<{\frac{d}{d t}_{|t = 0} \gamma_t^A(f)  } = - \beta \<{ f \frac{d}{dt}_{|t = 0 }\gamma_t^A(\beta  H)  } \] 

  Let $f = \sigma_i^2 \sigma_j^1$ for some $i$ in $A$ and $j \notin A$.
  The identity reads
  \begin{align*}
  \<{\sigma_i^1 \sigma_j^1 } & = \beta \sum_{k \in \Lambda, l \notin \Lambda, k \sim l } \<{ \sigma_i^2 \sigma_k^2 \sigma_j^1 \sigma_l^1 -  \sigma_i^2 \sigma_k^1 \sigma_j^1 \sigma_l^2} \\
  &\le   \beta \sum_{k \in \Lambda, l \notin \Lambda, k \sim l } \<{ \sigma_i^2 \sigma_k^2 \sigma_j^1 \sigma_l^1 } \\
  &\le  \beta \sum_{k \in \Lambda, l \notin \Lambda, k \sim l } \<{ \sigma_i^2 \sigma_k^2 }\<{\sigma_j^1 \sigma_l^1 } = \beta \sum_{k \in \Lambda, l \notin \Lambda, k \sim l } \<{ \sigma_i^1\sigma_k^1 }\<{\sigma_j^1 \sigma_l^1 }~.
\end{align*}

In the first line, the term $\<{\sigma_i^2 \sigma_k^1 \sigma_j^1 \sigma_l^2} $ is nonnegative by Ginibre first inequality which holds under our assumptions as can be seen by expanding the Hamiltonian in Equation (\ref{eq:Hamiltonian_ala_KPV}) where the role of $\lambda_2$ and $\lambda_3$ are switched as a power series. The second step is an application of Equation (\ref{eq:plane_correl_ineq}).
\end{proof}

A well known corollary of the Simon--Lieb inequality is Corollary \ref{cor:sharpness}.
\begin{proof}[Proof of corollary \ref{cor:sharpness}]
  Let $G(m) = \sup_{i, j \in \Z^2, |i-j| \ge m} \<{ \sigma_i^1 \sigma_j^1 }$. 
  If $|i - j| = m$, and $m > k $, consider a square box $\Lambda_k$ of size length $2 k$ around $i$, we have 
  \[ \<{ \sigma_i^1 \sigma_j^1 } \le \sum_{r \in \Lambda_k, l \notin \Lambda_k, r \sim l} \beta  \<{ \sigma_i^1 \sigma_r^1 }\<{\sigma_j^1 \sigma_l^1 }  \le \beta c  G(m-k-1)  \]  as 
  $|j -l| \ge m -k - 1 $ for any $l$ outside $\Lambda_k$ that is a neighbor of a vertex in $\Lambda_k$ and as $\sum_{r \in \Lambda_k, l \notin \Lambda_k, r \sim l} \<{ \sigma_i^1 \sigma_r^1 } \le  8 k \frac{c}{  8 k }  \le c  $. 
  Thus, $G(m) \le (c \beta)^t G(m- t (k + 1))$ for any $t$ such that $m- (t-1)(k+1) > k$ and $G(m)$ thus decays exponentially.

\end{proof}

\section{\for{toc}{Comparing the anisotropic $\S^2_+$ to the  anisotropic $\S^2$}\except{toc}{Comparing the anisotropic hemisphere sphere model to the  anisotropic sphere model}}
\label{sec:comparing_half_and_full}

In this section, we prove Proposition \ref{prop:half-sphere-is-dominated-by-full-sphere}.
First we write the coordinates of a point in $\mathbb{S}^2$ in polar coordinates but instead of using an angle for the $z$ axis, we use a radius $r$ between $0$ and $1$ which represents the distance to the vertical distance to the $x y$ plane and a sign $\eta = \pm 1$, which encodes if the point is in the upper or lower hemisphere.

With these coordinates, a point in $\mathbb{R}^3$ has Cartesian coordinates \[ (r \cos(\theta), r \sin(\theta), \eta \sqrt{1 - r^2})~.\] In this system of coordinates the uniform measure $\mu$ can be written
$\mu(d \theta, d r) \propto \, d \theta \frac{r}{\sqrt{1 -r^2}} d r$ where $d \theta$ is the uniform measure on the circle and $d r$ is the Lebesgue measure on $[0,1]$. 

As we will use it a lot, we set $a(r) := \sqrt{1 - r^2} \ .$
We denote by $\kappa$ the probability measure $\kappa(d r) = \frac{1}{Z} \frac{r}{a(r)} d r$.

In this system of coordinates, the Hamiltonian of the anisotropic sphere model is written as

\[ H(\theta, r, \eta) = \sum_{i \sim j} [r_i r_j \cos(\theta_i - \theta_j) +  \lambda \eta_i a(r_i) \eta_j a(r_j)] \]

and the Hamiltonian of the upper half sphere model can be written as a function of $\theta$ and $r$ as

\[ H(\theta, r) = \sum_{i \sim j} [r_i r_j \cos(\theta_i - \theta_j) + \lambda a(r_i) a(r_j))] \]
because $\eta \equiv 1$ in the upper hemisphere.

\subsection{The full sphere model is a half sphere model in random environment}
\label{sec:random_evt_half}

To prove Proposition  \ref{prop:half-sphere-is-dominated-by-full-sphere}, we need to reinterpret the full sphere model as a model in the upper half sphere in random environment. We will imagine that we sample the anisotropic sphere model in the following way. We first sample the $\eta$ signs according to their distribution, then we sample the radii $r$ conditionally on $\eta$. Finally, we sample the angles conditionally on $r$ and $\eta$. We will consider $\eta$ and $r$ as being the "environment" and $\theta$ as being the "true randomness".

Formally, note that:

\[ \langle \cdot \rangle_{\beta, \Lambda,\lambda} = \frac{1}{Z_{\beta, \Lambda, \lambda}} \sum_{\eta \in \{\pm 1\}^\Lambda} Z_{\beta,\Lambda , \lambda, \eta} \ \langle \cdot \rangle_{\beta, \Lambda, \lambda,\eta} \]

where 
\begin{itemize}
\item $Z_{\beta, \Lambda, \lambda}$ is the partition function of $\langle \cdot \rangle_{\beta, \Lambda,\lambda}$
\item $\langle \cdot \rangle_{\beta, \Lambda, \lambda,\eta}$ is the measure with Hamiltonian $H_{\Lambda, \lambda, \eta}(r, \theta) = \sum_{i \sim j} [r_i r_j \cos(\theta_i - \theta_j) + \lambda \eta_i a(r_i) \eta_j a(r_j)]$ and whose reference measure is the uniform measure in the upper half hemisphere of $\mathbb{S}^2$
\item $Z_{\beta,\Lambda , \lambda, \eta}$ is the partition function of the model $\langle \cdot \rangle_{\beta, \Lambda, \lambda,\eta}$.
\end{itemize}

We define the following probability measure on $\{\pm 1\}^{\Lambda}$:

\[ \mathbb{P}(\eta) := \frac{Z_{\beta,\Lambda , \lambda, \eta}}{Z_{\beta, \Lambda, \lambda}} \]

which allows to see that

\begin{equation}\label{eq:random_evt_in_eta}
\langle \cdot \rangle_{\beta, \Lambda,\lambda} = \sum_{\eta \in \{\pm 1\}} \mathbb{P}(\eta) \langle \cdot \rangle_{\beta, \Lambda, \lambda,\eta} \ .
\end{equation}

We will show that for all $\eta \in \{\pm 1\}^{\Lambda}$ and $B \subset \Lambda$, we actually have

\begin{equation}\label{eq:all_other_etas_are_bigger}
\langle r_B \cos(m \cdot \theta) \rangle_{\beta, \Lambda, \lambda, \eta} \geq \langle r_B \cos(m \cdot \theta) \rangle_{\beta, \Lambda, \lambda}^+ 
\end{equation}
where $r_B = \prod_{i \in B} r_i$.
Notice that Equation (\ref{eq:random_evt_in_eta}) together with Equation (\ref{eq:all_other_etas_are_bigger}) immediately imply Proposition \ref{prop:half-sphere-is-dominated-by-full-sphere}.

\subsection{Stochastic domination and proof of  Equation (\ref{eq:all_other_etas_are_bigger})}

We do the same operation for the $r$ variable as for the $\eta$ variable, i.e. we write for all test functions $f$,

\[ \langle f(r, \theta) \rangle_{\beta, \Lambda, \lambda, \eta} = \int_r \langle f(r, \theta) \rangle_{\beta, \Lambda, \lambda, r}^{XY} \mu_\eta(d r) \]

where
\begin{itemize}
\item $\langle \cdot \rangle_{\beta, \Lambda, \lambda, r}^{XY}$ is the XY model with random coupling constants $(r_i r_j \mathbbm{1}_{i \sim j})_{i,j}$, that is the model with Hamiltonian: $H_{\beta, \Lambda,r}(\theta) = \sum_{i \sim j} r_i r_j \cos(\theta_i - \theta_j)$ and reference measure the uniform measure on $\mathbb{S}^1$ 
\item $\mu_\eta(d r) = \frac{Z_{\beta, \Lambda, r}^{XY}}{Z_{\beta, \Lambda,\lambda, \eta}} \exp\left(\lambda \sum_{i \sim j} a(r_i)a(r_j) \eta_i \eta_j\right) \kappa^{\otimes \Lambda}(d r)$ with $Z_{\beta, \Lambda, r}^{XY}$ being the partition function of $\langle \cdot \rangle_{\beta, \Lambda, \lambda, r}^{XY}$ and $Z_{\beta, \Lambda, \lambda, \eta}$ the partition function of $\langle \cdot \rangle_{\beta, \Lambda, \lambda, \eta}$.
\end{itemize}

To simplify notations we drop all subscripts $\beta$, $\Lambda$ and $\lambda$ in what follows.

The following lemma is key to the proof of \ref{eq:all_other_etas_are_bigger}.

\begin{lemma}\label{lem:stoch_domi_for_mu_eta}
For all $\eta \in \{\pm 1\}^\Lambda$, $\mu_\eta$ stochastically dominates $\mu_{\mathbf{1}}$.
\end{lemma}

This lemma together with Ginibre inequality imply Equation (\ref{eq:all_other_etas_are_bigger}). Indeed, conditioning with respect to $r$, we have

\begin{align*} \langle r_B \cos(m \cdot \theta) \rangle_{\beta, \Lambda, \lambda, \eta} 
  &= \langle r_B \cos(m \cdot \theta) \mid (r_i)_{i \in \Lambda} \rangle_{\beta, \Lambda, \lambda, \eta} \\ &= \mathbb{E}_{\mu_\eta} \left[ r_B \langle \cos(m \cdot \theta) \mid (r_i)_{i \in \Lambda} \rangle_{\beta, \Lambda, \lambda, \eta} \right] 
\end{align*}
as $r_B$ is $r$-measurable.

Now note that $\langle \cos(m \cdot \theta) \mid (r_i)_{i \in \Lambda} \rangle_{\beta, \Lambda, \lambda, \eta} = \langle \cos(m \cdot \theta) \rangle_{\beta, \Lambda, J}^{XY}$ 
where $J_{i j} := r_i r_j \mathbbm{1}_{i \sim j}$.

By Ginibre inequality, $\langle \cos(m \cdot \theta) \rangle_{\beta, \Lambda, J}^{XY}$ is nondecreasing in $r$, and $r_B$ is clearly nondecreasing in $r$. Thus, $r_B \langle \cos(m \cdot \theta) \mid (r_i)_{i \in \Lambda} \rangle_{\beta, \Lambda, \lambda, \eta}$ is a nondecreasing function of $r$ and by stochastic domination we have,

\[ \mathbb{E}_{\mu_\eta} \left[ r_B \langle \cos(m \cdot \theta) \mid (r_i)_{i \in \Lambda} \rangle_{\beta, \Lambda, \lambda, \eta} \right] \geq \mathbb{E}_{\mu_{\mathbf{1}}} \left[ r_B \langle \cos(m \cdot \theta) \mid (r_i)_{i \in \Lambda} \rangle_{\beta, \Lambda, \lambda, 1} \right] \ . \]

\begin{proof}[Proof of Lemma \ref{lem:stoch_domi_for_mu_eta}]
We use the following generalized version of Holley's criterion  \cite{holley_remarks_1974} for stochastic domination which is due to Preston. It is a special case of \cite[Theorem 3]{preston_generalization_1974}.
\begin{lemma}[Preston]\label{crit:stoch_holley}
Let $\mu$ and $\nu$ be two probability measures on $[0,1]^\Lambda$ with density $f_\mu$ and $f_\nu$ respectively with respect to some product probability measure on $[0,1]^\Lambda$. If for all configurations $r^1$ and $r^2$ with $r^1 \leq r^2$, we have 
\begin{equation} \label{eq:Holley}
f_\mu (r^1 \wedge r^2) f_\nu (r^1 \vee r^2) \geq f_\mu (r^1) f_\nu (r^2) 
\end{equation}
then $\nu$ stochastically dominates $\mu$. 
\end{lemma}

The density with respect to $\kappa^{\otimes \Lambda}$ of $\mu_\eta$ is 
$f_\eta = \frac{Z_r^{XY}}{Z_\eta} \exp\left(\beta \lambda \sum_{i \sim j} a(r_i) a(r_j) \eta_i \eta_j\right)$.

For $i,j$ in $\Lambda$, let $g_{i,j, \eta}(r) := \exp(\beta \lambda a(r_i) a(r_j) \eta_i \eta_j)$. We claim that for all $i \sim j$ in $\Lambda$,
\begin{equation} \label{eq:holley_for_exp}
g_{i,j,1}(r^1 \wedge r^2) g_{i,j,\eta}(r^1 \vee r^2) \geq g_{i,j,1}(r^1) g_{i,j,\eta}(r^2)
\end{equation}

and that 
\begin{equation} \label{eq:holley_for_Z}
Z_{r^1 \wedge r^2}^{XY} Z_{r^1 \vee r^2}^{XY} \geq Z_{r^1}^{XY} Z_{r^2}^{XY} \ .
\end{equation}

By combining Equation (\ref{eq:holley_for_exp}) and Equation (\ref{eq:holley_for_Z}), we obtain Equation (\ref{eq:Holley}) for $f_\eta$ and $f_{\mathbf{1}}$ and thus desired stochastic domination.

As $a(r)$ is a decreasing function of $r$, Equation (\ref{eq:holley_for_exp}) is equivalent to checking that 
\[ \exp(\beta  \lambda \epsilon (x \wedge y) (x' \wedge y') + (x \vee y) (x' \vee y')) \geq \exp(\beta \lambda (x x' + \epsilon y y')) \]
for all $x, x', y, y' \in [0,1]$ and $\epsilon = \pm 1$, i.e. that \[  \epsilon (x \wedge y) (x' \wedge y') + (x \vee y) (x' \vee y') - x x' - \epsilon y y' \geq 0 ~.\]
If $\epsilon = 1$, this is the rearrangement inequality with two numbers. If $\epsilon = -1$ this is even bigger than if $\epsilon = 1$ as 

\[ (x \wedge y) (x' \wedge y') - y y' \leq 0 \ . \]

We now turn to Equation (\ref{eq:holley_for_Z}). We drop the XY in the superscripts of the partition functions.
For a fixed $i \in \Lambda$ and configuration $r \in \Lambda$, we define $r(x)$ as the configuration equal to $r$ except at $i$ where $r_i$ is replaced by $x$. We claim that if $r \leq r'$ then $Z_r(x)/Z_{r'(x)}$ is nondecreasing. 

This implies Equation (\ref{eq:holley_for_Z}) as it is possible to go from $Z_{r^1 \wedge r^2}/Z_{r^1}$ to $Z_{r^2}/Z_{r^1 \vee r^2}$ by increasing the coordinates that differ between $r^1 \wedge r^2$ and $r^2$.

The fact that $(Z_{r'(x)}/Z_r(x))$ is nondecreasing is a consequence of Ginibre inequality. Indeed,

\[ \frac{d}{d x} \left(\frac{Z_{r'(x)}}{Z_r(x)}\right) = \beta \left( \langle \sum_{i \sim j} \cos(\theta_i - \theta_j) \rangle_{\mu_\beta,r'(x)} - \beta \langle \sum_{i \sim j} \cos(\theta_i - \theta_j) \rangle_{\mu_\beta,r(x)} \right) \]
which is nonnegative as $r' \geq r$. 

\end{proof}

\section{\for{toc}{Diverging scales are needed}\except{toc}{Diverging scales are needed (Lemma \ref{l.NP})}}\label{a.NP}

{\em Proof of Lemma \ref{l.NP}.}

Fix $\lambda<1$ (we think of it as very close to 1).  
For any $R\geq 1$, we fix the boundary condition along $\p \Lambda_R$ to be constant and pointing in the north direction. 

Assume  that there exists a groundstate $\sigma=(\sigma_i)_{i\in \Lambda_R}$ of $\S^2$ spins which  achieves a better $\lambda$-energy than the constant field $\sigma_{NP}$ pointing in the north direction. 

As $\sigma \neq \sigma_{NP}$, this implies that there exists $i_0 \in  \Lambda_R$ such that $\sigma_{i_0} \neq NP:=(0,0,1)$. 

We now distinguish two cases.

\ni
\textbf{Case 1.} For all $i \in \Lambda_R$ one has $\sigma_{i}^z \geq 0$.

Define in this case 
\begin{align*} 
i_1 \in  \mathrm{Argmin}\{ \sigma_i^z, i \in  \Lambda_R \} 
\end{align*}
and 
\begin{align*} 
\delta:= \sqrt{(\sigma_{i_1}^x)^2 + (\sigma_{i_1}^y)^2}\,.
\end{align*}
Notice that by our assumption, $\delta>0$. 

Up to a fixed additive constant, the $\lambda$-Energy functional may be rewritten 
\begin{align*} 
H_\lambda(\sigma) = \frac 1 2 \sum_{i\sim j} \| \sigma_i -\sigma_j\|_2^2 + (1-\lambda) \sum_{i \sim j} \sigma_i^z \sigma_j^z\,.
\end{align*}

Using that the Dirichlet energy is the sum of the $xy$ Dirichlet energy and the $z$-component Dirichlet energy, notice that -- at best -- the $xy$ Dirichlet energy of $\sigma$ will be greater than the one-component Dirichlet energy of the harmonic extension in $\Lambda_R$ with $0$-boundary conditions on $\p \Lambda_R$ and set to be equal to $\delta$ at site $i_1$. If $i_1$ is in the bulk of $\p \Lambda_R$ this energy is of order $\frac {\delta^2} {\log R}$  and if $i_1$ gets closer to the boundary, it is straightforward to show (by adding the constraints of the corresponding variational problem) that the Dirichlet energy in the $xy$ plane is at least $ C \frac {\delta^2} {\log R}$  for some universal $C>0$ which does not depend on $R$. (N.B. To see why the $2d$ Dirichlet energy is bounded from below by the $1d$ harmonic extension, one can use an $xy$ basis suitably oriented in the direction of the $2d$ vector $(\sigma_{i_1}^x, \sigma_{i_2}^y)$). 

We have thus identified that $\sigma$ needs to pay at least $ C \frac {\delta^2} {\log R}$ of Dirichlet energy while the constant field $\sigma_{NP}$ does not pay any such Dirichlet Energy. The reason why $\sigma$ may achieve a better energy than the constant field $\sigma_{NP}$ is due to the second (nearly potential) energy term. Indeed, the constant field will pay an energy exactly equal to 
\begin{align*} 
(1-\lambda) |\{ i\sim j, i\in \Lambda_R, \, j\in \Lambda_R \cup \p \Lambda_R\}| * 1^2
\end{align*}
while, even in the most favorable case, the groundstate configuration $\sigma$ will pay at least 
\begin{align*} 
(1-\lambda) |\{ i\sim j, i\in \Lambda_R, \, j\in \Lambda_R \cup \p \Lambda_R\}| * \sqrt{1-\delta^2} \sqrt{1-\delta^2}\,.
\end{align*}
Indeed, because of our definition of case 1 and also by definition of $i_1$, any site $i\in \Lambda_R$ (and also $\p \Lambda_R$) is such that $\sigma_i^z \geq \sqrt{1-\delta^2}$. (N.B This step is the reason why we need to divide our analysis in 2 cases!).

This means in particular that the groundstate configuration $\sigma$ will gain at most 
\begin{align*} 
(1-\lambda) \Big[ 1 - (1-\delta^2) \Big] \,  \left |\{ i\sim j, i\in \Lambda_R, \, j\in \Lambda_R \cup \p \Lambda_R\} \right |  \leq c  (1-\lambda) \delta^2 R^2\,,
\end{align*}
where $c R^2$ is an upper bound on the number of interaction-edges involved.

Since $\sigma$ is assumed to be a groundstate with lower energy than the constant field $\sigma_{NP}$, this requires that the following inequality is satisfied
\begin{align*} 
C \frac {\delta^2} {\log R}  < c  (1-\lambda) \delta^2 R^2 
\end{align*}
or equivalently 
\begin{align*} 
R^2 \log(R)  > \tilde c \frac 1 {1-\lambda}
\end{align*}
which gives us for some universal constant $\bar c>0$ that the scale $R$ needs to be larger than 
\begin{align}\label{e.Rl}
R(\lambda):= \frac {\bar c}
{\sqrt{(1-\lambda) \log \frac 1 {1-\lambda}}}
\end{align}

\ni
\textbf{Case 2.} There exists $i_1 \in \Lambda_R$ such that $\sigma_{i_1}^z < 0$.  (This case probably never happens but analysing it seems easier than ruling it out). 

In this case, we do not need to define a parameter $\delta$: indeed the Dirichlet energy of the vertical direction (which needs to interpolate from $+1$ along $\p \Lambda_R$ to a negative value somewhere in $\Lambda_R$) is at least $\frac c {\log R}$. At the very best --with a chessboard configuration of $\pm 1$-- the potential energy of $\sigma$ cannot be smaller than $-c (1-\lambda )R^2$. We thus find that in this unlikely case the scale $R$ needs to satisfy 
\begin{align*} 
\frac C {\log R} - c (1-\lambda ) R^2 <   c(1-\lambda) R^2
\end{align*}
which gives us 
\begin{align*} 
\frac C {\log R} <   2 c(1-\lambda) R^2
\end{align*}
which again gives us a similar lower bound $R\geq R(\lambda)$ as in~\eqref{e.Rl}
\qed

\newsection[Poisson kernel estimates on (hyper)cubic domains]{A lemma on the Poisson kernel of the simple random walk of a (hyper)cubic domain.}
 Recall that if $D \subset \Z^d$,  $\partial D $ denotes the inner boundary of $D$.
We define $D^{\text{int}}$ as being the interior of $D$ that is  $ D^{\text{int}} =  D \setminus \partial D$.

Let $(S_k)_{k \ge 0} $ be a nearest neighbor simple random walk on $\Z^d$. We denote by $\P_x$ the distribution of the random walk starting at $x$.  


 For a domain $D$ in $\Z^d$, let $\tau_D = \inf \, \{k \ge 1, S_k \in \partial D \} $ and $\overline{\tau}_D = \inf \, \{k \ge 0, S_k \in \partial D \} $. We define the Poisson kernel on $D$ by  $K_D (x,y) = \P_x (S_{\tau_D } = y )$ for $x \in D $ and $y \in \partial D $.

 For $x, y \in D^{\text{int}}$, the Green function $G_D$ is defined as \[ G_D (x,y) = \E_{x}\left( \sum_{k = 0}^{\overline{\tau}_D-1} \1_{(S_k = y)}\right) ~. \] If $x $ or $y$ is not in $D^{\text{int}}$, we thus have $G_D (x,y) = 0 $.
 Notice that for any $x \in D^{\text{int}} $ and  $y \in \partial D$ that has a unique neighbor $u \in D^{\text{int}}$,  $H_D (x,y) = \frac{1}{2 d } G_D (x,u)$. 

 We say that a set $D $ has a good boundary if any $y \in \partial D$ has a unique neighbor in $D^{\text{int}}$.
 Notice that for any set $A \subset B \subset \Z^d$, $G_A (x,y) \le G_B (x,y)$ for any $x,y \in A $. If $A$ and $B$ have good boundaries, for any $x \in A $ and $y \in \partial A \cap \partial B$, $H_A (x,y) \le H_B (x,y)$.

Let $\Lambda = \Lambda_n = [0,n+1]^d \cap \Z^d$. This choice of notations is not consistent with the rest of the paper but it is much more convenient in this context.
\begin{lemma}\label{lemma:poisson_kernel}
Let $d\geq 2$. For any $u, v \in \partial \Lambda$,  $K_\Lambda (u,v) \le \frac{C_d}{d(u,v)^d} $ for some universal constant $C_d$ that depends only on $d$ and where $d(x,y)$ is the graph distance between $x$ and $y$.
\end{lemma}

Let $\mathcal{H} = \{(x,y) \in \Z^{d-1} \times \Z \mid y \ge 0 \}$

To prove this lemma, we use the following precise estimate of $K_{\mathcal{H}}$, (proposition 8.1.1 in \cite{Lawler_Limic_2010}).

\begin{prop}[\cite{Lawler_Limic_2010}]\label{prop:asymptotic_precise_poisson_demi_plan}
If $z = (x,y) \in \mathcal{H}$ and $y \ge 1$,

\[ K(z, 0) = \frac{2 |y| }{\omega_d |z|^d} + O \left(\frac{|y|^2}{|z|^{d+2}} \right) + O\left( \frac{1}{|z|^{d+1}}\right) \] 

where $\omega_d  $ is the surface area of the $d-1$ dimensional sphere of radius $1$.
\end{prop}

Let us prove Lemma \ref{lemma:poisson_kernel} from this proposition. If $u$ and $v$ are on the same face of $\Lambda_n$, by symmetry we may assume that $u$ and $v$ are on the bottom face (the boundary of $\mathcal{H}$). The lemma thus follows from the fact that $K_{\Lambda} (u,v) \le K_{\mathcal{H}}(u,v)  $ and from Proposition \ref{prop:asymptotic_precise_poisson_demi_plan}.

If $u$ and $v$ are on opposite faces, the lemma follows from a simple (one dimensional) gambler's ruin argument and an application of the local central limit theorem that we do not detail. 

It is only when $u $ and $v$ are on adjacent faces that we need to use the precise estimate. Let us assume without loss of generality that $u$ in on the bottom face and that $v$ is on the "left face": $\{(0,b) \in \Z \times \Z^{d-1} \}$. Let $\mathcal{C}$ be the set of points of $\Z^d$ that are above the face of $x $ and the face of $y$, that is $\mathcal{C} = \{(a,b,c) \in \Z \times \Z^{d-2} \times \Z \mid \, a> 0, c >0 \}$. We only need to show that $K_{\mathcal{C}} (u,v) = O\left(\frac{1}{d(u,v)^d}\right)$ as $K_\Lambda (u,v) \le K_{\mathcal{C}} (u,v)$ as both $\mathcal{C}$ and $\Lambda$ have a good boundary.

For any point $a,b \in \mathcal{C}$, Let $\overline{a}$ be the symmetric of $a$ through the left face. We claim that $G_{\mathcal{C}} (a,b) = G_{\mathcal{H}} (a,b) - G_{\mathcal{H}} (\overline{a},b) $ for any $a,b \in \mathcal{C}^{\text{int}}$. This follows from the fact that for any fixed $a \in \mathcal{C}^{\text{int}}$, the function $G_{\mathcal{C} } (a, \cdot ) - (G_{\mathcal{H}} (a,\cdot) - G_{\mathcal{H}} (\overline{a},\cdot))$ is bounded and harmonic on $\mathcal{C}$ and vanishes on the boundary of $\mathcal{C}$.


Thus, if $u'$ (respectively  $v'$) is the unique neighbor of $u$ (respectively  $v$) in $\mathcal{C}^{\text{int}}$, then  
\begin{align*}
K_{\mathcal{C}}(u,v)  &= \frac{1}{2d} G_{\mathcal{C}}(u',v') = \frac{1}{2d} \left(G_{\mathcal{H}}(u',v') - G_{\mathcal{H}}(\overline{u}',v')\right)   \\ 
&= K_{\mathcal{H}}(v',u) -  K_{\mathcal{H}}(v',\overline{u})  \\ & = \frac{2}{\omega_d}  \frac{|v'_2|}{|v'-u|^d} - \frac{2}{\omega_d}  \frac{|v'_2|}{|v'-\overline{u} |^d} +  O\left(\frac{|v|^2}{|u - v |^{d+2}} \right) + O\left(\frac{1}{|u- v |^{d+1}}\right) \\ &= O\left(\frac{1}{|u- v| ^d} \right) ~. 
\end{align*}

\bibliographystyle{alpha}
\bibliography{references-S2}

\end{document}